\documentclass[12pt]{article}                
\usepackage{graphicx, xcolor}
\usepackage{psfrag}
\usepackage{amsmath, amssymb, amsthm}
\usepackage{mathptmx}
\usepackage{esvect, bm, bbm}
\usepackage{hyperref}
\usepackage{cleveref}

\hypersetup{
	colorlinks=true,
	linkcolor={blue!80!black},
	citecolor={green!50!black},
	urlcolor={red!50!black}
}

\newcommand{\N}{\mathbb{N}}

\newcommand{\Q}{\mathbb{Q}}
\newcommand{\R}{\mathbb{R}}
\newcommand{\C}{\mathbb{C}}

\newcommand{\1}{\mathbbm{1}}
\newcommand{\A}{\mathcal{A}}
\newcommand{\CC}{\mathcal{C}}
\newcommand{\EE}{\mathcal{E}}
\newcommand{\F}{\mathcal{F}}
\newcommand{\M}{\mathcal{M}}
\newcommand{\p}{\varphi}
\renewcommand{\P}{\mathbb{P}}
\renewcommand{\O}{\Omega}

\newcommand{\asigma}{\vv{\sigma}}
\newcommand{\bsigma}{\bm{\sigma}}
\newcommand{\basigma}{\vv{\bm{\sigma}}}
\newcommand{\arho}{\vv{\rho}}
\newcommand{\brho}{\bm{\rho}}
\newcommand{\barho}{\vv{\bm{\rho}}}
\newcommand{\atau}{\vv{\tau}}
\newcommand{\btau}{\bm{\tau}}
\newcommand{\batau}{\vv{\bm{\tau}}}

\newcommand{\x}{\vv{x}}
\newcommand{\bB}{\bm{B}}
\newcommand{\bL}{\boldsymbol{L}}
\newcommand{\bv}{\boldsymbol{v}}
\newcommand{\bW}{\bm{W}}
\newcommand{\bX}{\boldsymbol{X}}
\newcommand{\bY}{\boldsymbol{Y}}

\newcommand{\GSE}{\mathrm{\mathbf{GSE}}}
\newcommand{\GP}{\mathrm{GP}}
\newcommand{\VGP}{\mathrm{\mathbf{GP}}}

\newcommand{\abs}[1]{\lvert#1\rvert}
\newcommand{\norm}[1]{\lVert#1\rVert}
\newcommand{\Norm}[1]{{\vert\kern-0.1ex\vert\kern-0.1ex\vert#1\vert\kern-0.1ex\vert\kern-0.1ex\vert}}

\newcommand{\inc}{\nearrow}
\newcommand{\dec}{\searrow}

\DeclareMathOperator{\ud}{\mathrm{d}}

\DeclareMathOperator{\E}{\mathbb{E}}

\DeclareMathOperator{\Cov}{Cov}
\DeclareMathOperator{\tr}{tr}
\DeclareMathOperator{\diag}{diag}
\DeclareMathOperator{\Sum}{\operatorname{Sum}}

\DeclareMathOperator{\Par}{\mathcal{P}}

\DeclareMathOperator{\sgn}{\operatorname{sgn}}

\newcommand{\step}[1]{{\underline{\textit {Step #1}}}}

\newtheorem{theorem}{Theorem}[section]
\newtheorem{lemma}[theorem]{Lemma}
\newtheorem{proposition}[theorem]{Proposition}
\newtheorem{corollary}[theorem]{Corollary}

\begin{document}

\title{\vspace{-1cm}The $\ell^p$-Gaussian-Grothendieck problem\\ with vector spins
}
\author{Tomas Dominguez\thanks{\textsc{\tiny Department of Mathematics, University of Toronto, tomas.dominguezchiozza@mail.utoronto.ca}}\\
}

\date{}
\maketitle
\vspace*{-0.7cm}
\begin{abstract}
We study the vector spin generalization of the $\ell^p$-Gaussian-Grothendieck problem. In other words, given integer $\kappa\geq 1$, we investigate the asymptotic behaviour of the ground state energy associated with the Sherrington-Kirkpatrick Hamiltonian indexed by vector spin configurations in the unit $\ell^p$-ball. The ranges $1\leq p\leq 2$ and $2<p<\infty$ exhibit significantly different behaviours. When $1\leq p\leq 2$, the vector spin generalization of the $\ell^p$-Gaussian-Grothendieck problem agrees with its scalar counterpart. In particular, its re-scaled limit is proportional to some norm of a standard Gaussian random variable. On the other hand, for $2<p<\infty$ the re-scaled limit of the $\ell^p$-Gaussian-Grothendieck problem with vector spins is given by a Parisi-type variational formula.
\end{abstract}

\section{Introduction and main results}

Given an $N\times N$ matrix $A=(a_{ij})$ and some $1\leq p<\infty$, the $\ell^p$-Grothendieck problem consists in maximizing the quadratic form $\smash{\sum_{i,j=1}^N a_{ij}\sigma_i\sigma_j}$ over all vectors $\bsigma=(\sigma_1,\ldots,\sigma_N)\in \R^N$ with unit $\ell^p$-norm, $\smash{\norm{\bsigma}_p^p=\sum_{i=1}^N \abs{\sigma_i}^p=1}$.
In other words, it involves computing the quantity
\begin{equation}\label{eqn: GP deterministic}
\GP_{N,p}(A)=\max\Big\{\sum_{i,j=1}^N a_{ij}\sigma_i\sigma_j\mid \sum_{i=1}^N \abs{\sigma_i}^p=1\Big\}.
\end{equation}
In the case $p=2$, this is the maximum eigenvalue of the symmetric matrix $(A + A^T )/2$. On the other hand, the limiting case $p=\infty$ has been extensively studied in the mathematics and computer science literature for its applications to combinatorial optimization, graph theory and correlation clustering \cite{AlonNaor, Charikar, Fiedler, Guruswami, Motzkin}. The range $2<p<\infty$ can be thought of as an interpolation between the spectral and the correlation clustering problems \cite{Kindler}, while the range $1<p<2$ seems to be unexplored in the literature. Finding an efficient algorithm to solve the $\ell^p$-Grothendieck problem whenever $p\neq 2$ is generally difficult \cite{Kashin, KhotSafra, Megretski, Montanari, Nemirovski}, so it is natural to study the $\ell^p$-Grothendieck problem for random input matrices first. In fact, it should help understand the typical behaviour of \eqref{eqn: GP deterministic}. This leads to the $\ell^p$-Gaussian-Grothendieck problem,
\begin{equation}\label{eqn: GP scalar}
\GP_{N,p}=\max\Big\{\sum_{i,j=1}^N g_{ij}\sigma_i\sigma_j\mid \sum_{i=1}^N \abs{\sigma_i}^p=1\Big\},
\end{equation}
where $(g_{ij})$ are independent standard Gaussian random variables.

The asymptotic behaviour of \eqref{eqn: GP scalar} was studied in great detail in \cite{WeiKuo}. It was discovered that the re-scaled limit of \eqref{eqn: GP scalar} exhibits significantly different behaviour in the ranges $1\leq p\leq 2$ and $2<p<\infty$; in the former, it is proportional to some norm of a Gaussian random variable, and in the latter, it is given by a Parisi-type variational formula. In this paper, we will show that this behaviour persists in the vector spin generalization of \eqref{eqn: GP scalar}. Our work is motivated and greatly influenced by \cite{WeiKuo}; however, new ideas are needed to treat the range $2<p<\infty$. These will be detailed at a later stage, and they will allow us to avoid the key truncation step in \cite{WeiKuo} as well as its associated technicalities. Therefore, specializing our arguments to the scalar setting, $\kappa=1$, yields a simpler proof of the main result in \cite{WeiKuo}.

Before we describe the vector spin analogue of \eqref{eqn: GP scalar}, let us mention that another motivation for investigating this optimization problem comes from the study of spin glass models. In the language of statistical physics, the quadratic form $\sum_{i,j=1}^N g_{ij}\sigma_i\sigma_j$ is known as the Hamiltonian of the Sherrington-Kirkpatrick (SK) mean-field spin glass model, and the quantity \eqref{eqn: GP scalar} corresponds to the ground state energy of the SK model on the unit $\ell^p$-sphere. From this perspective, the vector spin generalization of \eqref{eqn: GP scalar} which we will study in this paper is very natural; it also appears in the computer-science literature \cite{AlonNaor, Charikar, Guruswami, Khot, Kindler} when studying the convex relaxation of \eqref{eqn: GP scalar}. 

Let us now describe the vector spin generalization of \eqref{eqn: GP scalar} using the notation and terminology of statistical physics. Fix an integer $\kappa\geq 1$ throughout the remainder of this paper. The Hamiltonian of the vector spin SK model is the random function of the $N\geq 1$ vector spins taking values in $\R^\kappa$,
\begin{equation}
\basigma=\big(\asigma_1,\ldots,\asigma_N\big)\in (\R^\kappa)^N,
\end{equation}
given by the quadratic form
\begin{equation}\label{eqn: GP pre-Hamiltonian}
H_N^\circ(\basigma)=\sum_{i,j=1}^N g_{ij}\big(\asigma_i,\asigma_j\big),
\end{equation}
where the interaction parameters $(g_{ij})$ are independent standard Gaussian random variables and $(\cdot,\cdot)$ is the Euclidean inner product on $\R^\kappa$. Denote the coordinates of each spin $\asigma_i$ by
\begin{equation}
\asigma_i=(\sigma_i(1),\ldots,\sigma_i(\kappa))\in \R^\kappa,
\end{equation}
write the configuration of the $k$'th coordinates as
\begin{equation}
\bsigma(k)=(\sigma_1(k),\ldots,\sigma_N(k))\in \R^N,
\end{equation}
and introduce the $\ell^{p,2}$-norm on the Euclidean space $(\R^\kappa)^N$,
\begin{equation}\label{eqn: GP lp norm}
\norm{\basigma}_{p,2}^p=\sum_{i=1}^N \norm{\asigma_i}_2^p.
\end{equation}
The $\ell^p$-Gaussian-Grothendieck problem with vector spins consists in maximizing the Hamiltonian \eqref{eqn: GP pre-Hamiltonian} over the unit $\ell^{p,2}$-sphere. In other words, it involves computing the quantity
\begin{equation}\label{eqn: GP vector}
\VGP_{N,p}=\max\big\{H_N^\circ (\basigma)\mid \norm{\basigma}_{p,2}=1\big\}.
\end{equation}

To handle the range $1\leq p\leq 2$, we will use the Gaussian Hilbert space approach to the Grothendieck inequality \cite{AlonNaor,Charikar, Guruswami, Kindler} in order to show that for any $N\times N$ matrix $A=(a_{ij})$,
\begin{equation}\label{eqn: GP key for 1p2}
\GP_{N,p}(A)=\max\Big\{\sum_{i,j=1}^N a_{ij}\big(\asigma_i,\asigma_j\big)\mid \norm{\basigma}_{p,2}=1\Big\}.
\end{equation}
This identity was mentioned in \cite{Khot}, but no proof was given. Combining \eqref{eqn: GP key for 1p2} with theorem 1.1 and theorem 1.2 in \cite{WeiKuo} will immediately give our main result for $1\leq p\leq 2$.

\begin{theorem}\label{GP main result 1p2}
If $1<p<2$, then almost surely,
\begin{equation}
\lim_{N\to\infty}\frac{1}{N^{1/p^*}}\VGP_{N,p}=2^{\frac{1}{2}-\frac{2}{p}}\big(\E \abs{g}^{p^*}\big)^{1/p^*},
\end{equation}
where $p^*$ is the Hölder conjugate of $p$ and $g$ is a standard Gaussian random variable. On the other hand, if $p=1$ or $p=2$, then almost surely,
\begin{equation}
\lim_{N\to \infty}\frac{1}{\sqrt{\log N}}\GP_{N,1}=\sqrt{2}=\lim_{N\to \infty}\frac{1}{\sqrt{N}}\GP_{N,2}.
\end{equation}
\end{theorem}

The range $2<p<\infty$ will require substantially more work, and will occupy the majority of this paper. It will be convenient to introduce a re-scaled version of the Hamiltonian \eqref{eqn: GP pre-Hamiltonian},
\begin{equation}\label{eqn: GP Hamiltonian}
H_N(\basigma)=\frac{1}{\sqrt{N}}\sum_{i,j=1}^N g_{ij}\big(\asigma_i,\asigma_j\big),
\end{equation}
as well as a normalized version of the $\ell^{p,2}$-norm \eqref{eqn: GP lp norm},
\begin{equation}
\Norm{\basigma}_{p,2}^p=\frac{1}{N}\norm{\basigma}_{p,2}^p=\frac{1}{N}\sum_{i=1}^N\norm{\asigma_i}_2^p.
\end{equation}
If we denote the classical SK Hamiltonian on $\R^N$ by
\begin{equation}
H_N^k(\bsigma(k))=\frac{1}{\sqrt{N}}\sum_{i,j=1}^N g_{ij}\sigma_i(k)\sigma_j(k),
\end{equation}
we may express the vector spin Hamiltonian \eqref{eqn: GP Hamiltonian} as
\begin{equation}
H_N(\basigma)=\sum_{k=1}^\kappa H_N^k(\bsigma(k)).
\end{equation}
It is readily verified that for two spin configurations $\basigma^l,\basigma^{l'}\in (\R^\kappa)^N$ and two integers $1\leq k,k'\leq \kappa$, 
\begin{equation}
\E H_N^k(\bsigma^l(k))H_N^{k'}(\bsigma^{l'}(k'))=N\big(R_{l,l'}^{k,k'}\big)^2,
\end{equation}
where
\begin{equation}
R_{l,l'}^{k,k'}=\frac{1}{N}\sum_{i=1}^N\sigma_i^l(k)\sigma_i^{l'}(k')
\end{equation}
is the overlap between $\bsigma^l(k)$ and $\bsigma^{l'}(k')$. We will denote the matrix of all such overlaps by
\begin{equation}\label{eqn: GP self-overlap}
R(\basigma^l,\basigma^{l'})=\big(R_{l,l'}^{k,k'}\big)_{k,k'\leq \kappa}=\frac{1}{N}\sum_{i=1}^N \asigma_i^l{\asigma_i^{l'}}^T.
\end{equation}
The covariance structure of the vector spin Hamiltonian \eqref{eqn: GP Hamiltonian} may therefore be expressed in terms of this matrix-valued overlap as
\begin{equation}\label{eqn: GP covariance}
\E H_N(\basigma^l) H_N(\basigma^{l'})=N\sum_{k,k'=1}^\kappa\big(R_{l,l'}^{k,k'}\big)^2=N\norm{R_{l,l'}}_{\text{HS}}^2,
\end{equation}
where $\norm{\gamma}_{\mathrm{HS}}^2=\sum_{k,k'}\abs{\gamma_{k,k'}}^2$ denotes the Hilbert-Schmidt norm on the space of $\kappa\times \kappa$ matrices. Writing 
\begin{equation}
S_p^N=\{\basigma\in (\R^\kappa)^N\mid \Norm{\basigma}_{p,2}=1\}
\end{equation}
for the unit normalized-$\ell^{p,2}$-sphere, the $\ell^p$-Gaussian-Grothendieck problem with vector spins may be recast as the task of computing the ground state energy
\begin{equation}\label{eqn: GP reformulated}
\GSE_{N,p}=N^{\frac{2}{p}-\frac{3}{2}}\VGP_{N,p}=\frac{1}{N} \max_{\basigma\in S_p^N}H_N(\basigma).
\end{equation}
Using Chevet's inequality as in section 3 of \cite{WeiKuo}, it is easy to see that this is the correct scaling of \eqref{eqn: GP vector} when $2< p<\infty$. To study the constrained optimization problem \eqref{eqn: GP reformulated}, it is natural to remove the normalized-$\ell^{p,2}$-norm constraint by considering the model with an $\ell^{p,2}$-norm potential. For each $t>0$ define the Hamiltonian
\begin{equation}\label{eqn: GP unconstrained Hamiltonian}
H_{N,p,t}(\basigma)=H_N(\basigma)-t\norm{\basigma}_{p,2}^p,
\end{equation}
and introduce the unconstrained Lagrangian
\begin{equation}\label{eqn: GP unconstrained Lagrangian}
L_{N,p}(t)=\frac{1}{N}\max_{\basigma \in (\R^\kappa)^N}H_{N,p,t}(\basigma).
\end{equation}
Our first noteworthy result in the range $2<p<\infty$, which we now describe, will relate the asymptotic behaviour of the unconstrained Lagrangian \eqref{eqn: GP unconstrained Lagrangian} to the limit of the ground state energy \eqref{eqn: GP reformulated}.

Consider the space of $\kappa\times \kappa$ Gram matrices,
\begin{equation}
\Gamma_\kappa=\big\{\gamma\in \R^{\kappa\times \kappa}\mid  \gamma \text{ is symmetric and non-negative definite}\big\},
\end{equation}
endowed with the Loewner order $\gamma_1\leq \gamma_2$ if and only if $\gamma_2-\gamma_1\in \Gamma_\kappa$, and denote by $\Gamma_\kappa^+$ the subspace of positive definite matrices in $\Gamma_\kappa$,
\begin{equation}
\Gamma_\kappa^+=\{\gamma\in \Gamma_\kappa\mid \gamma \text{ is positive definite}\}.
\end{equation}
For each $D\in \Gamma_\kappa$ write
\begin{equation}\label{eqn: GP configurations with self-overlap D}
\Sigma(D)=\big\{\basigma\in (\R^\kappa)^N\mid R(\basigma,\basigma)=D\big\}
\end{equation}
for the set of spin configurations $\basigma\in (\R^\kappa)^N$ with self-overlap $D$, and introduce the constrained Lagrangian
\begin{equation}\label{eqn: GP Lagrangian with self-overlap D}
L_{N,p,D}(t)=\frac{1}{N}\max_{\basigma\in \Sigma(D)}H_{N,p,t}(\basigma).
\end{equation}
In \cref{GP sec2}, we will show that the constrained Lagrangian \eqref{eqn: GP Lagrangian with self-overlap D} admits a deterministic limit $L_{p,D}(t)$ with probability one, and in \cref{GP sec4}, we will establish the following asymptotic formula for the unconstrained Lagrangian \eqref{eqn: GP unconstrained Lagrangian}.

\begin{theorem}\label{GP limit of unconstrained Lagrangian}
If $2<p<\infty$, then almost surely the limit $L_p(t)=\lim_{N\to \infty}L_{N,p}(t)$ exists for every $t>0$. Moreover, with probability one,
\begin{equation}\label{eqn: GP limit of unconstrained Lagrangian}
L_p(t)=\sup_{D\in \Gamma_\kappa}L_{p,D}(t)=\sup_{D\in \Gamma_\kappa^+}L_{p,D}(t).
\end{equation}
\end{theorem}

\noindent Subsequently, in \cref{GP sec5}, we will use the basic properties of convex functions to derive the following key relationship between the limits of \eqref{eqn: GP unconstrained Lagrangian} and \eqref{eqn: GP reformulated}.

\begin{theorem}\label{GP in terms of limiting Lagrangian}
If $2<p<\infty$, then almost surely the limit $\GSE_p=\lim_{N\to\infty}\GSE_{N,p}$ exists and is given by
\begin{equation}\label{eqn: GP GSE in terms of limiting Lagrangian}
\GSE_p=\frac{p}{2}\Big(\frac{p}{2}-1\Big)^{\frac{2}{p}-1}t^{\frac{2}{p}}L_p(t)^{1-\frac{2}{p}}
\end{equation}
for every $t>0$.
\end{theorem}

\noindent This result reduces the study of the ground state energy \eqref{eqn: GP reformulated} to that of the Lagrangian \eqref{eqn: GP Lagrangian with self-overlap D} with positive definite self-overlaps $D\in \Gamma_\kappa^+$. The main result of this paper will be a Parisi-type variational formula for the limit $L_{p,D}(t)$ of \eqref{eqn: GP Lagrangian with self-overlap D} when $D\in \Gamma_\kappa^+$. Together with \eqref{eqn: GP GSE in terms of limiting Lagrangian}, \eqref{eqn: GP limit of unconstrained Lagrangian} and \eqref{eqn: GP reformulated}, this will give a Parisi-type variational formula for the $\ell^p$-Gaussian-Grothendieck problem with vector spins when $2<p<\infty$.

Given $D\in \Gamma_\kappa^+$, we now describe the Parisi-type variational formula for the limit $L_{p,D}(t)$ of \eqref{eqn: GP Lagrangian with self-overlap D}. Let us call a path $\pi:[0,1]\to \Gamma_\kappa$ piecewise linear if there exists a partition $0=q_{-1}\leq q_0\leq \ldots\leq q_r=1$ of $[0,1]$ and matrices $(\gamma_j)_{-1\leq j\leq r}\subset \Gamma_\kappa$ with
\begin{equation}
\pi(s)=\gamma_{j-1}+\frac{s-q_{j-1}}{q_j-q_{j-1}}(\gamma_j-\gamma_{j-1})
\end{equation}
when $s\in [q_{j-1},q_j]$ for some $0\leq j\leq r$. Denote by $\Pi$ the space of piecewise linear and non-decreasing functions on $[0,1]$ with values in $\Gamma_\kappa$,
\begin{equation}
\Pi=\big\{ \pi : [0,1]\to\Gamma_\kappa \mid \pi \text{ is piecewise linear, } \pi(x)\leq \pi(x') \text{ for } x\leq x'\big\},
\end{equation}
and write $\Pi_D$ for the set of paths in $\Pi$ that start at $0$ and end at $D$,
\begin{equation}
\Pi_D=\big\{\pi \in \Pi\mid \pi(0)=0 \text{ and }\pi(1)=D\big\}.
\end{equation}
Notice that any path $\pi \in \Pi_D$ can be identified with two sequences of parameters,
\begin{align}
0&=q_{-1}\leq q_0\leq \ldots\leq q_{r-1}\leq q_r=1,\label{eqn: GP zero temperature qs}\\
0&=\gamma_{-1}=\gamma_0\leq \gamma_1\leq \ldots \leq \gamma_{r-1}\leq \gamma_r=D,\label{eqn: GP zero temperature gammas}
\end{align}
satisfying $\pi(q_j)=\gamma_j$ for $0\leq j\leq r$. Explicitly, the path $\pi$ is given by
\begin{equation}
\pi(s)=\gamma_{j-1}+\frac{s-q_{j-1}}{q_j-q_{j-1}}(\gamma_j-\gamma_{j-1})
\end{equation}
when $s\in [q_{j-1},q_j]$ for some $0\leq j\leq r$. Denote by $\mathcal{N}^d$ the set of finite measures on $[0,1]$ with finitely many atoms, and given $t>0$ and $\lambda\in \R^{\kappa(\kappa+1)/2}$ consider the function $f_{\lambda}^\infty:\R^\kappa\to \R$ defined by
\begin{equation}\label{eqn: GP zero-temperature terminal condition}
f_{\lambda}^\infty(\x)=\sup_{\asigma\in \R^\kappa}\Big(\big(\asigma, \x\big)+\sum_{k\leq k'}\lambda_{k,k'}\sigma(k)\sigma(k')-t\norm{\asigma}_2^p\Big).
\end{equation}
Notice that any discrete measure $\zeta\in \mathcal{N}^d$ may be identified with two sequences of parameters
\begin{align}
0&=q_{-1}\leq q_0\leq \ldots\leq q_{r-1}\leq q_r=1, \label{eqn: GP zero temperature measure qs}\\
0&=\zeta_{-1}\leq \zeta_0\leq\ldots\leq \zeta_{r-1}\leq \zeta_r<\infty, \label{eqn: GP zero temperature zetas}
\end{align}
satisfying $\zeta(\{q_j\})=\zeta_j-\zeta_{j-1}$ for $0\leq j\leq r$. Moreover, the sequences \eqref{eqn: GP zero temperature qs} and \eqref{eqn: GP zero temperature measure qs} can be taken to be the same by duplicating values in \eqref{eqn: GP zero temperature gammas} and \eqref{eqn: GP zero temperature zetas} if necessary. We will often abuse notation and write $\zeta$ both for the measure and its cumulative distribution function. Given independent Gaussian vectors $z_j=(z_j(k))_{k\leq \kappa}$ for $0\leq j\leq r$ with covariance structure
\begin{equation}\label{eqn: GP covariance of z}
\Cov z_j=\gamma_j-\gamma_{j-1},
\end{equation}
recursively define the sequence $(Y_l^{\lambda,\zeta,\pi})_{0\leq l\leq r}$ as follows. Let
\begin{equation}
Y_r^{\lambda,\zeta,\pi}((z_j)_{0\leq j\leq r})=f_\lambda^\infty\Big(\sqrt{2}\sum_{j=1}^rz_j\Big),
\end{equation}
and for $0\leq l\leq r-1$ let
\begin{equation}\label{eqn: GP Parisi sequence at zero temperature}
Y_l^{\lambda,\zeta,\pi}((z_j)_{0\leq j\leq l})=\frac{1}{\zeta_l}\log \E_{z_{l+1}}\exp \zeta_l Y_{l+1}^{\lambda,\zeta,\pi}((z_j)_{0\leq j\leq l+1}).
\end{equation}
This inductive procedure is well-defined by the growth bounds established in \cref{GP zero temperature terminal growth bound}. Introduce the Parisi functional,
\begin{equation}\label{eqn: GP Parisi functional at zero temperature}
\Par_\infty(\lambda,\zeta,\pi)=Y_0^{\lambda,\zeta,\pi}-\sum_{k\leq k'}\lambda_{k,k'}D_{k,k'}-\int_0^1 \zeta(s)\Sum\big(\pi(s)\odot \pi'(s)\big)\ud s,
\end{equation}
where $\Sum(\gamma)=\sum_{k,k'}\gamma_{k,k'}$ is the sum of all elements in a $\kappa\times \kappa$ matrix and $\odot$ denotes the Hadamard product on the space of $\kappa\times \kappa$ matrices.
We have made all dependencies on $\kappa,p,t$ and $D$ implicit for clarity of notation, but we will make them explicit whenever necessary. The following is our main result.

\begin{theorem}\label{GP main result}
If $2<p<\infty$, then for any $D\in \Gamma_\kappa^+$ and every $t>0$,
\begin{equation}\label{eqn: GP main result}
L_{p,D}(t)=\inf_{\lambda,\zeta,\pi}\Par_\infty(\lambda,\zeta,\pi),
\end{equation}
where the infimum is taken over all $(\lambda,\zeta,\pi)\in \R^{\kappa(\kappa+1)/2}\times \mathcal{N}^d\times \Pi_D$.
\end{theorem}

We close this section with a brief outline of the paper. \Cref{GP sec2} will be devoted to the range $1\leq p\leq 2$ and will include a proof of \cref{GP main result 1p2}. The rest of the paper will focus on the range $2<p<\infty$. In \cref{GP sec3}, we will use the Guerra-Toninelli interpolation \cite{GuerraToninelli, PanSKB} and the Gaussian concentration inequality \cite{PanL,PanSKB} to show that the constrained Lagrangian \eqref{eqn: GP Lagrangian with self-overlap D} admits a deterministic limit. In \cref{GP sec4}, we show that, in a certain sense, the limit of the constrained Lagrangian depends continuously on the constraint. This continuity result is inspired by lemma 7.1 in \cite{WeiKuo}. Unfortunately, lemma 7.1 in \cite{WeiKuo} does not extend to the vector spin setting since we can no longer modify overlaps by simply re-scaling spin configurations. To overcome this issue, we will revisit lemma 4 in \cite{PanVec}, originally designed to prove a vector spin version of the Ghirlanda-Guerra identities \cite{Ghirlanda}, and we will leverage Dudley's entropy inequality \cite{DudleyInequality, TalagrandC}. With this continuity result at hand, we will closely follow section 7 and section 8 in \cite{WeiKuo} to prove \cref{GP limit of unconstrained Lagrangian} and \cref{GP in terms of limiting Lagrangian}. This will be the content of \cref{GP sec5} and \cref{GP sec6}. In \cref{GP sec7}, we will introduce a free energy functional that depends on an inverse temperature parameter $\beta>0$ and is asymptotically equivalent to the constrained Lagrangian \eqref{eqn: GP Lagrangian with self-overlap D} after letting $\beta\to \infty$. For each finite $\beta>0$, a simple modification of the arguments in \cite{PanVec}, which we will not detail, gives a Parisi-type variational formula for the limit of the free energy functional. This is reviewed in \cref{GP sec8}. The rest of the paper is devoted to finding a similar Parisi-type variational formula after letting $\beta\to \infty$. This is where our approach differs substantially from that in \cite{WeiKuo}. In our attempt to generalize the truncation argument in sections 10-12 of \cite{WeiKuo} to the vector spin setting, we discovered that by a careful analysis of the terminal condition \eqref{eqn: GP zero-temperature terminal condition} and its positive temperature analogue, the proof for the scalar, $\kappa=1$, case could be considerably simplified. This simplified proof extended with minor modifications to the vector spin setting and is what we present between \cref{GP sec9} and \cref{GP sec11} of this paper. In particular, our arguments can be used to simplify the proof of the main result in \cite{WeiKuo}. The careful analysis of the terminal conditions is undertaken in \cref{GP sec9}. The resulting bounds are combined with the Auffinger-Chen representation \cite{AuffingerChenREP, Jagannath} in \cref{GP sec10} to compare the Parisi functional \eqref{eqn: GP Parisi functional at zero temperature} and its positive temperature counterpart. The specific form of the Auffinger-Chen representation that we use is a higher dimensional generalization of that in \cite{WeiKuo, WeiKuo2DPar}. The proof of \cref{GP main result} is finally completed in \cref{GP sec11}. For the reader's convenience, we have postponed a number of technical estimates to \cref{GP app1}, and we have included a review of elementary results in linear algebra in \cref{GP app2}.

\section{The range \texorpdfstring{$1\leq p\leq 2$}{1 ≤ p ≤ 2}}\label{GP sec2}

In this section we show that the $\ell^p$-Gaussian-Grothendieck problem with vector spins agrees with its scalar counterpart in the range $1\leq p\leq 2$ by proving \eqref{eqn: GP key for 1p2}. Recall the definition \eqref{eqn: GP deterministic} of the $\ell^p$-Grothendieck problem $\GP_{N,p}(A)$ for an arbitrary $N\times N$ matrix $A=(a_{ij})$.

\begin{lemma}\label{GP key for 1p2 proof}
For any $N\times N$ matrix $A=(a_{ij})$ and every $1\leq p\leq 2$,
\begin{equation}\label{eqn: GP key for 1p2 proof}
\GP_{N,p}(A)=\max\Big\{\sum_{i,j=1}^N a_{ij}\big(\asigma_i,\asigma_j\big)\mid \norm{\basigma}_{p,2}=1\Big\}.
\end{equation}
\end{lemma}

\begin{proof}
Given $\bsigma\in \R^N$ in the unit $\ell^p$-sphere, consider the vector spin configuration $\basigma\in (\R^\kappa)^N$ defined by
$$\basigma(k)=
\begin{cases}
\bsigma& \text{if } k=1,\\
0& \text{otherwise}.
\end{cases}$$
Notice that $\norm{\basigma}_{p,2}^p=\norm{\bsigma}_p^p=1$ and $\sum_{i,j=1}^N a_{ij}\sigma_i\sigma_j=\sum_{i,j=1}^N a_{ij}\big(\asigma_i,\asigma_j\big)$. It follows that
$$\sum_{i,j=1}^N a_{ij}\sigma_i\sigma_j\leq \max\Big\{\sum_{i,j=1}^N a_{ij}\big(\asigma_i,\asigma_j\big)\mid \norm{\basigma}_{p,2}=1\Big\},$$
and taking the maximum over all such $\bsigma\in \R^N$ gives the upper bound in \eqref{eqn: GP key for 1p2 proof}. To prove the matching lower bound, fix a vector spin configuration $\basigma\in (\R^\kappa)^N$ in the unit $\ell^{p,2}$-sphere. Let $g$ be a standard Gaussian random vector in $\R^\kappa$ and for each $1\leq i\leq N$ consider the random variable
$$X_i=\big(g,\asigma_i\big).$$ 
Observe that $\E X_iX_j=\sum_{k=1}^\kappa \sigma_i(k)\sigma_j(k)=\big(\asigma_i,\asigma_j\big)$. Normalizing the random vector $X=(X_i)_{i\leq N}$, it is easy to see that
\begin{equation}\label{eqn: GP key for 1p2 lower bound}
\sum_{i,j=1}^N a_{ij}\big(\asigma_i,\asigma_j\big)=\E \sum_{i,j=1}^N a_{ij}X_iX_j\leq \GP_{N,p}(A)\E \Big(\sum_{i=1}^N \abs{X_i}^p\Big)^{2/p}.
\end{equation}
To bound this further, denote by $\norm{\cdot}_{L^2}$ the $L^2$-norm defined by the law of $g$. Minkowski's integral inequality and the assumption $1\leq p\leq 2$ imply that
$$\E \Big(\sum_{i=1}^N \abs{X_i}^p\Big)^{2/p}=\norm{\norm{X}_p}_{L^2}^2\leq \norm{\norm{X}_{L^2}}_p^2=\Big(\sum_{i=1}^N \big(\E \abs{X_i}^2\big)^{p/2}\Big)^{2/p}=\norm{\basigma}_{p,2}^2=1.$$
Substituting this into \eqref{eqn: GP key for 1p2 lower bound} gives the lower bound in \eqref{eqn: GP key for 1p2 proof} and completes the proof.
\end{proof}

Applying this result to the random matrix $G_N=(g_{ij})_{i,j\leq N}$ conditionally on the disorder chaos shows that $\VGP_{N,p}=\GP_{N,p}$ for $1\leq p\leq 2$. \Cref{GP main result 1p2} is therefore an immediate consequence of theorem 1.1 and theorem 1.2 in \cite{WeiKuo}. This concludes our discussion of the $\ell^p$-Gaussian-Grothendieck problem for $1\leq p\leq 2$.

\section{The limit of the constrained Lagrangian}\label{GP sec3}

In this section we begin the proof of \cref{GP limit of unconstrained Lagrangian} by combining the Gaussian concentration inequality with the Guerra-Toninelli interpolation to show that the random quantity \eqref{eqn: GP Lagrangian with self-overlap D} almost surely admits a deterministic limit for every constraint $D\in \Gamma_\kappa$. As usual \cite{GuerraToninelli, PanSKB}, the proof will come down to proving super-additivity of an appropriate sequence and appealing to the classical Fekete lemma.

\begin{theorem}\label{GP Guerra Toninelli}
If $2<p<\infty$ and $t>0$, then for every $D\in \Gamma_\kappa$ the limit
\begin{equation}
L_{p,D}(t)=\lim_{N\to\infty}\E L_{N,p,D}(t)
\end{equation}
exists. Moreover, with probability one, $L_{p,D}(t)=\lim_{N\to\infty}L_{N,p,D}(t)$.
\end{theorem}

\begin{proof}
We will be working with systems of different sizes, so let us make the dependence of \eqref{eqn: GP configurations with self-overlap D} on $N$ explicit by writing $\Sigma_N(D)$. Given $\basigma\in \Sigma_N(D)$, the covariance structure of the vector spin Hamiltonian \eqref{eqn: GP Hamiltonian} established in \eqref{eqn: GP covariance} together with \cref{trace dominates HS norm} reveal that
$$\E H_N(\basigma)^2=N\norm{R(\basigma,\basigma)}_{\text{HS}}^2\leq N\tr(R(\basigma,\basigma))^2=N\tr(D)^2.$$
It follows by the Gaussian concentration of the maximum that for any $s>0$,
$$\P\big\{\abs{L_{N,p,D}(t)-\E L_{N,p,D}(t)}\geq s\big\}\leq 2\exp\Big(-\frac{Ns^2}{4\tr(D)^2}\Big).$$
Since the right-hand side of this expression is summable in $N$, the Borel-Cantelli lemma implies that
$$\limsup_{N\to\infty}\abs{ L_{N, p,D}(t)-\E L_{N,p,D}(t)}=0$$
with probability one. It is therefore sufficient to prove that the sequence $(\E L_{N,p,D}(t))_N$ admits a limit. We will do this through the Fekete lemma by showing that the sequence $(N\E L_{N,p,D}(t))_N$ is super-additive. This is equivalent to proving that for all integers $N,M\geq 1$,
\begin{equation}\label{eqn: GP Guerra-Toninelli super-additivity}
\E \max_{\basigma\in \Sigma_N(D)}H_{N,p,t}(\basigma)+\E \max_{\basigma\in \Sigma_M(D)}H_{M,p,t}(\basigma)\leq \E \max_{\basigma\in \Sigma_{N+M}(D)}H_{N+M,p,t}(\basigma).
\end{equation}
Given a spin configuration $\barho\in (\R^\kappa)^{N+M}$, write $\barho=(\basigma,\batau)$ for $\basigma\in (\R^\kappa)^N$ and $\batau\in (\R^\kappa)^M$. Consider three independent vector spin SK Hamiltonians $H_{N+M}(\barho)$, $H_N(\basigma)$ and $H_M(\batau)$ defined on $\Sigma_N(D)\times \Sigma_M(D)$, $\Sigma_N(D)$ and $\Sigma_M(D)$ respectively. For each $s\in [0,1]$ introduce the interpolating Hamiltonian on $\Sigma_N(D)\times \Sigma_M(D)$,
$$H_{N+M,s}(\barho)=\sqrt{s}H_{N+M}(\barho)+\sqrt{1-s}\big(H_N(\basigma)+H_M(\batau)\big)-t\norm{\basigma}_{p,2}^p-t\norm{\batau}_{p,2}^p.$$
Given an inverse temperature parameter $\beta>0$ and two probability measures $\mu_N$ and $\mu_M$ supported on $\Sigma_N(D)$ and $\Sigma_M(D)$ respectively, denote by
$$\p(s)=\frac{1}{\beta(N+M)}\E\log \int_{\Sigma_N(D)\times \Sigma_M(D)}\exp\beta H_{N+M,s}(\barho)\ud \mu_N(\basigma)\ud \mu_M(\batau)$$
the interpolating free energy and write $\langle \cdot\rangle_s$ for the Gibbs average with respect to the interpolating Gibbs measure
$$\ud G_{N+M}(\basigma,\batau)=\frac{\exp \beta H_{N+M,s}(\barho)\ud \mu_N(\basigma)\ud \mu_M(\batau)}{\int_{\Sigma_N(D)\times \Sigma_M(D)}\exp \beta H_{N+M,s}(\barho)\ud\mu_N(\basigma)\ud \mu_M(\batau)}.$$
The Gaussian integration by parts formula (see for instance lemma 1.1 in \cite{PanSKB}) yields
$$\p'(s)=\frac{1}{N+M}\E\Big\langle \frac{\partial H_{N+M,s}(\basigma)}{\partial s}\Big\rangle_s=\frac{1}{N+M}\E\big\langle C(\barho^1,\barho^1)-C(\barho^1,\barho^2)\big\rangle_s,$$
where
\begin{align*}
C(\barho^1,\barho^2)&=\frac{\beta (N+M)}{2}\Big(\norm{R(\barho^1,\barho^2)}_{\text{HS}}^2-\frac{N}{N+M}\norm{R(\basigma^1,\basigma^2)}_{\text{HS}}^2\\
&\qquad \qquad\qquad\quad-\frac{M}{N+M}\norm{R(\batau^1,\batau^2)}_{\text{HS}}^2\Big).
\end{align*}
Since
$$R(\barho^1,\barho^2)=\frac{N}{N+M}R(\basigma^1,\basigma^2)+\frac{M}{N+M}R(\batau^1,\batau^2),$$
the convexity of the square of a norm implies that $C(\barho^1,\barho^2)\leq 0$. Combined with the fact that $R(\barho^1,\barho^1)=R(\basigma^1,\basigma^1)=R(\batau^1,\batau^1)=D$, this shows that $\p'(s)\geq 0$ and therefore $\p(0)\leq \p(1)$. Letting $\beta\to\infty$ in this inequality and remembering that the $L^q$-norm tends to the $L^\infty$-norm as $q\to\infty$ yields
$$\E \max_{\basigma\in \Sigma_N(D)}H_{N,p,t}(\basigma)+\E \max_{\batau\in \Sigma_M(D)}H_{M,p,t}(\batau)\leq \E \max_{\barho\in \Sigma_N(D)\times \Sigma_M(D)}H_{N+M,p,t}(\barho).$$
Since $\Sigma_N(D)\times \Sigma_M(D)\subset \Sigma_{N+M}(D)$, this gives \eqref{eqn: GP Guerra-Toninelli super-additivity} and completes the proof.
\end{proof}

The heuristic validity of \cref{GP limit of unconstrained Lagrangian} should now be clear. From \eqref{eqn: GP self-overlap}, the self-overlap of any vector spin configuration is a Gram matrix in $\Gamma_\kappa$. This means that for every integer $N\geq 1$, the relationship between the unconstrained Lagrangian \eqref{eqn: GP unconstrained Lagrangian} and the constrained Lagrangian \eqref{eqn: GP Lagrangian with self-overlap D} is
\begin{equation}
L_{N,p}(t)=\sup_{D\in \Gamma_\kappa}L_{N,p,D}(t).
\end{equation}
Formally bringing the limit into the supremum and using the density of positive definite matrices in the space of non-negative definite matrices yields \eqref{eqn: GP limit of unconstrained Lagrangian}. To turn this heuristic into a rigorous argument, we will use a compactness argument. This will be done in \cref{GP sec5} and will require the continuity properties of the constrained Lagrangian \eqref{eqn: GP Lagrangian with self-overlap D} that we explore in the next section.

\section{Continuity of the constrained Lagrangian}\label{GP sec4}

In this section we prove that, in a certain sense, the limit of the constrained Lagrangian \eqref{eqn: GP Lagrangian with self-overlap D} is continuous with respect to the constraint $D\in \Gamma_\kappa$ by combining lemma 4 in \cite{PanVec} with the classical Dudley inequality as it is stated in equation (A.23) of \cite{TalagrandC}. 

Lemma 4 in \cite{PanVec} was originally designed to modify the vector spin coordinates in the mixed-$p$-spin model in order to prove the matrix-overlap Ghirlanda-Guerra identities. Using these identities, it is then possible to access the synchronization mechanism \cite{PanMS, PanPotts} and find a tight lower bound for the limit of the free energy through the Aizenman-Sims-Starr scheme \cite{JustinP, PanVec}. We will apply this lemma for a different purpose, and, as it turns out, we will need a more explicit expression for the constant $L>0$ appearing in the upper bound. For our purposes, it will be important that this constant is uniformly bounded for all $D\in \Gamma_\kappa$ with uniformly bounded trace. We will therefore repeat the proof of this result and carefully track the dependence of constants.

For each $\epsilon>0$ and $D\in \Gamma_\kappa$ denote by $B_\epsilon(D)$ the open $\epsilon$-neighbourhood of $D$,
\begin{equation}
B_\epsilon(D)=\{\gamma\in \Gamma_\kappa\mid \norm{\gamma-D}_\infty<\epsilon\},
\end{equation}
with respect to the sup-norm $\norm{\gamma}_\infty=\max_{k,k'}\abs{\gamma_{k,k'}}$ on the space of $\kappa \times \kappa$ matrices, and consider the set of spin configurations
\begin{equation}
\Sigma_\epsilon(D)=\big\{\basigma\in (\R^\kappa)^N\mid R(\basigma,\basigma)\in B_\epsilon(D)\big\}
\end{equation}
with self-overlap in the $\epsilon$-neighbourhood of $D$. Denote by $\lambda_1\geq \cdots\geq  \lambda_\kappa$ the real and non-negative eigenvalues of $D$ and let
\begin{equation}
D=Q\Lambda Q^T
\end{equation}
be the eigendecomposition of $D$ with diagonal matrix $\Lambda=\diag(\lambda_1,\ldots,\lambda_\kappa)$. Given $\epsilon>0$, let $0\leq m\leq \kappa$ be such that $\lambda_m\geq \sqrt{\epsilon}$ and $\lambda_{m+1}<\sqrt{\epsilon}$. Introduce the matrix
\begin{equation}
D_\epsilon=Q\Lambda_\epsilon Q^T,
\end{equation}
where $\Lambda_\epsilon=\diag(\lambda_1,\ldots,\lambda_m,0,\ldots,0)$. Given any $\basigma\in \Sigma_\epsilon(D)$, we will construct a $\kappa\times \kappa$ matrix $A_{\basigma}$ such that the self-overlap of the configuration $A_{\basigma} \basigma=(A_{\basigma}\asigma_{i})_{i\leq N}$ is equal to $D_\epsilon$ and such that, in a certain sense, $A_{\basigma}$ has small distortion. Notice that the self-overlap of $A_{\basigma} \basigma$ is given by
\begin{equation}\label{eqn: GP overlap of modified coordinates}
R(A_{\basigma} \basigma, A_{\basigma} \basigma)=\frac{1}{N}\sum_{i=1}^N(A_{\basigma}\asigma_{i})(A_{\basigma}\asigma_{i})^T=A_{\basigma}R(\basigma,\basigma)A_{\basigma}^T,
\end{equation}
so we will need a matrix with $A_{\basigma}R(\basigma,\basigma)A_{\basigma}^T=D_\epsilon$. In this context, small distortion will mean that the overlap of $\basigma$ with other configurations should not change much when $\basigma$ is replaced by $A_{\basigma} \basigma$. To control this distortion, fix a configuration $\barho\in (\R^\kappa)^N$ with $\Norm{\barho}_{2,2}^2\leq u$ for some $u>0$, let $\batau=A_{\basigma}\basigma-\basigma$ and observe that by the Cauchy-Schwarz inequality
\begin{align}\label{eqn: GP modified coordinates small distortion}
\norm{R(A_{\basigma}\basigma,\barho)-R(\basigma,\barho)}_{\text{HS}}&=\Big\lVert \frac{1}{N}\sum_{i=1} ^NA_{\basigma}\asigma_i\arho_i^T-\frac{1}{N}\sum_{i=1}^N\asigma_i\arho_i^T\Big\rVert_{\text{HS}}\notag\\
&\leq \frac{1}{N}\sum_{i=1}^N \norm{\atau_i\arho_i^T}_{\text{HS}}=\frac{1}{N}\sum_{i=1}^N \norm{\atau_i}_2\norm{\arho_i}_2\notag\\
&\leq \sqrt{u}\Big(\frac{1}{N}\sum_{i=1}^N \norm{\atau_i}_2^2\Big)^{1/2}= \sqrt{u}\tr(R(\batau,\batau))^{1/2}\notag\\
&=\sqrt{u}\tr\big((A_{\basigma}-I)R(\basigma,\basigma)(A_{\basigma}-I)^T\big)^{1/2},
\end{align}
where the last inequality follows from the fact that $\batau=(A_{\basigma}-I)\basigma$. We therefore need a matrix for which $\tr((A_{\basigma}-I)R(\basigma,\basigma)(A_{\basigma}-I)^T)$ is small. The construction of the matrix $A_{\basigma}$ is precisely the content of lemma 4 in \cite{PanVec}.

\begin{lemma}\label{GP modified coordinates}
Given $0<\epsilon < \kappa^{-2}$, $D\in \Gamma_\kappa$, and $R\in B_\epsilon(D)$, there exists a matrix $A=A(R)$ such that $ARA^T=D_\epsilon$ and
\begin{equation}\label{eqn: GP modified coordinates bound}
\tr\big((A-I)R(A-I)^T\big)\leq C(\tr(D)+1)\sqrt{\epsilon}
\end{equation}
for some constant $C>0$ that depends only on $\kappa$.
\end{lemma}

\begin{proof}
The proof proceeds in two steps: first we reduce the problem to the case when $D=\Lambda$ and then we use Gershgorin's theorem to conclude. For the reader's convenience, Gershgorin's theorem has been transcribed as \cref{Gershgorin} in the appendix.
\newpage
\noindent\step{1: reducing to $D=\Lambda$}\\
Let us suppose temporarily that the result holds when $D$ is a diagonal matrix. Since $Q$ is an orthogonal matrix and the Hilbert-Schmidt norm is rotationally invariant,
$$\norm{Q^TRQ-\Lambda}_\infty\leq\norm{R-D}_{\text{HS}}\leq \kappa \epsilon.$$
We may therefore find a matrix $A(Q^TRQ)$ with $A(Q^TRQ)Q^TRQA(Q^TRQ)^T=\Lambda_\epsilon$ and
$$\tr\big((A(Q^TRQ)-I)Q^TRQ(A(Q^TRQ)-I)^T\big)\leq C \tr(\Lambda)\sqrt{\epsilon}.$$
If we set $A=QA(Q^TRQ)Q^T$, it is easy to see that $ARA^T=Q\Lambda_\epsilon Q^T=D_\epsilon$ and
\begin{align*}
\tr\big((A-I)R(A-I)^T\big)&=\tr\big(Q(A(Q^TRQ)-I)Q^TRQ(A(Q^TRQ)-I)Q^T\big)\\
&\leq C\tr(D)\sqrt{\epsilon}.
\end{align*}
The last inequality uses the cyclicity of the trace, the orthogonality of $Q$ and the fact that $\tr(D)=\tr(\Lambda)$. This shows that it suffices to prove the result when $D=\Lambda$ and $R\in B_\epsilon(\Lambda)$.\\
\step{2: proof for $D=\Lambda$}\\
Introduce the matrices $\smash{R_m=(R_{k,k'})_{k,k'\leq m}}$  and $\smash{\Lambda_m=\diag(\lambda_1,\ldots,\lambda_m)}$ consisting of the first $m$ rows and columns of $R$ and $\Lambda$ respectively. Consider the matrix $\smash{\tilde{R}_m=\Lambda_m^{-1/2}R_m\Lambda_m^{-1/2}}$. Since $R_m\in B_\epsilon(\Lambda_m)$ and $\Lambda_m$ is diagonal with all elements bounded below by $\sqrt{\epsilon}$, it is readily verified that $\norm{\tilde{R}_m-I}_\infty\leq \sqrt{\epsilon}$. Gershgorin's theorem implies that all the eigenvalues of $\tilde{R}_m$ are within $m\sqrt{\epsilon}$ from $1$. The assumption $\epsilon<\kappa^{-2}$ implies that $\tilde{R}_m$ is invertible and allows us to define the matrix
$$B=B(R_m)=\Lambda_m^{1/2}\tilde{R}^{-1/2}_m\Lambda_m^{-1/2}.$$
Using the fact that $R_m=\Lambda_m^{1/2}\tilde{R}_m\Lambda_m^{1/2}$, it is easy to see that $BR_mB^T=\Lambda_m$ and
$$(B-I)R_m(B-I)^T=\Lambda_m^{1/2}(I-\tilde{R}_m^{1/2})^2\Lambda_m^{1/2}.$$
Since the eigenvalues of $\tilde{R}_m$ are within $m\sqrt{\epsilon}$ from $1$, so are the eigenvalues of $\tilde{R}_m^{1/2}$. Observe that $\smash{\tilde{R}_m^{1/2}}$ is symmetric and non-negative definite, so it admits an eigendecomposition $\smash{\tilde{R}_m^{1/2}=\tilde{Q}_m\tilde{\Lambda}_m\tilde{Q}_m^T}$. It follows by the orthogonality of $\tilde{Q}_m$ that
$$\norm{I-\tilde{R}_m^{1/2}}_{\text{HS}}=\norm{I-\tilde{\Lambda}_m}_{\text{HS}}\leq m\norm{I-\tilde{\Lambda}_m}_\infty\leq \kappa^2\sqrt{\epsilon}.$$
The cyclicity of the trace, the Cauchy-Schwarz inequality and \cref{trace dominates HS norm} now give
\begin{align*}
\tr\big((B-I)R_m(B-I)^T\big)&=\tr\big(\Lambda_m(I-\tilde{R}_m^{1/2})^2\big)\leq \norm{\Lambda_m}_{\text{HS}}\norm{I-\tilde{R}_m^{1/2}}_{\text{HS}}^2\\
&\leq \kappa^4\tr(\Lambda_m)\epsilon.
\end{align*}
Finally, define the matrix $A$ by filling all rows and columns of $B$ from $m+1$ to $\kappa$ with zeros. It is clear that $ARA^T=\Lambda_\epsilon$. If we denote by $T=(R_{k,k'})_{k,k'\geq m+1}$ the matrix consisting of the last $\kappa-m$ rows and columns of $R$, then
\begin{align*}
\tr\big((A-I)R(A-I)^T\big)&=\tr\big((B-I)R_m(B-I)^T\big)+\tr(T)\\
&\leq \kappa^4\tr(\Lambda_m)\epsilon+(\kappa-m)\epsilon +\sum_{k=m+1}^\kappa \lambda_i\\
&\leq \big(\kappa^4\tr(\Lambda)+2\kappa\big)\sqrt{\epsilon}.
\end{align*}
We have used the fact that $T\in B_\epsilon(\Lambda)$ in the second inequality. This completes the proof.
\end{proof}

This result allows us to map any spin configuration $\basigma\in (\R^\kappa)^N$ with self-overlap in the $\epsilon$-neighbourhood of $D\in \Gamma_\kappa$ to a modified spin configuration $A_{\basigma}\basigma$ that is not too far from $\basigma$ and has a configuration-independent self-overlap $D_\epsilon$. These two facts will be fundamental to understanding the continuity of the constrained Lagrangian \eqref{eqn: GP Lagrangian with self-overlap D}. We will now quantify the distance between $\basigma$ and $A_{\basigma}\basigma$ in two different ways: with respect to the normalized-$\ell^{2,2}$-norm and relative to the canonical metric associated with the Hamiltonian \eqref{eqn: GP Hamiltonian},
\begin{equation}
\ud(\basigma^1,\basigma^2)=\Big(\E\big(H_N(\basigma^1)-H_N(\basigma^2)\big)^2\Big)^{1/2}.
\end{equation}
It will be convenient to notice that for any $\basigma\in (\R^\kappa)^N$,
\begin{equation}\label{eqn: GP 2-2 norm as trace of self overlap}
\Norm{\basigma}_{2,2}^2=\frac{1}{N}\sum_{i=1}^N\sum_{k=1}^\kappa\abs{\sigma_i(k)}^2=\sum_{k=1}^\kappa R(\basigma,\basigma)_{k,k}=\tr(R(\basigma,\basigma)),
\end{equation}
and to write
\begin{equation}
B^N_2(u)=\big\{\basigma\in (\R^\kappa)^N\mid \Norm{\basigma}_{2,2}^2\leq u\big\}
\end{equation}
for the ball of radius $\sqrt{u}$ with respect to the normalized-$\ell^{2,2}$-norm.

\begin{corollary}\label{GP modified coordinates norm bounds}
If $0<\epsilon<\kappa^{-2}$ and $D\in \Gamma_\kappa$, then for any $\basigma\in \Sigma_\epsilon(D)$,
\begin{equation}
\Norm{A_{\basigma}\basigma-\basigma}_{2,2}\leq C(\tr(D)+1)^{1/2}\epsilon^{1/4},
\end{equation}
where $C>0$ is a constant that depends only on $\kappa$.
\end{corollary}

\begin{proof}
By \eqref{eqn: GP 2-2 norm as trace of self overlap}, \eqref{eqn: GP overlap of modified coordinates} and \cref{GP modified coordinates},
$$\Norm{A_{\basigma}\basigma-\basigma}_{2,2}^2=\tr((A_{\basigma}-I)R(\basigma,\basigma)(A_{\basigma}-I))\leq C(\tr(D)+1)\epsilon^{1/2}.$$
This finishes the proof.
\end{proof}

\begin{corollary}\label{GP modified coordinate canonical metric bound}
If $u>1$ and $\basigma^1,\basigma^2\in B_2^N(u)$, then
\begin{equation}\label{eqn: GP canonical metric bounded by norm}
\ud(\basigma^1,\basigma^2)\leq 2N^{1/4}u^{3/4}\norm{\basigma^1-\basigma^2}_{2,2}^{1/2}.
\end{equation}
In particular, if $0<\epsilon<\kappa^{-2}$ and $D\in \Gamma_\kappa$, then for any $\basigma\in \Sigma_\epsilon(D)$,
\begin{equation}\label{eqn: GP modified coordinate canonical metric bound}
\ud(\basigma,A_{\basigma}\basigma)\leq C N^{1/2}(\tr(D)+1)\epsilon^{1/8},
\end{equation}
where $C>0$ is a constant that depends only on $\kappa$.
\end{corollary}

\begin{proof}
By the reverse triangle inequality,
\begin{align*}
\ud(\basigma^1,\basigma^2)^2&=N\big(\norm{R(\basigma^1,\basigma^1)}_{\text{HS}}^2+\norm{R(\basigma^2,\basigma^2)}_{\text{HS}}^2-2\norm{R(\basigma^1,\basigma^2)}_{\text{HS}}^2\big)\\
&\leq N\big(\norm{R(\basigma^1,\basigma^1)-R(\basigma^1,\basigma^2)}_{\text{HS}}(\norm{R(\basigma^1,\basigma^1)}_{\text{HS}}+\norm{R(\basigma^1,\basigma^2)}_{\text{HS}})\\
&\quad+\norm{R(\basigma^2,\basigma^2)-R(\basigma^1,\basigma^2)}_{\text{HS}}(\norm{R(\basigma^2,\basigma^2)}_{\text{HS}}+\norm{R(\basigma^1,\basigma^2)}_{\text{HS}})\big).
\end{align*}
To bound this further, notice that by \eqref{eqn: GP self-overlap} and the Cauchy-Schwarz inequality,
\begin{align*}
\norm{R(\basigma^1,\basigma^1)-R(\basigma^1,\basigma^2)}_{\text{HS}}&\leq \frac{1}{N}\sum_{i=1}^N \norm{\asigma_i^1(\asigma_i^1-\asigma_i^2)^T}_{\text{HS}}=\frac{1}{N}\sum_{i=1}^N \norm{\asigma^1_i}_2\norm{\asigma^1_i-\asigma_i^2}_{2}\\
&\leq \Norm{\basigma^1}_{2,2}\Norm{\basigma^1-\basigma^2}_{2,2}.
\end{align*}
Similarly, $\norm{R(\basigma^1,\basigma^2)}_{\text{HS}}\leq \Norm{\basigma^1}_{2,2}\Norm{\basigma^2}_{2,2}$.
It follows that for any $\basigma^1,\basigma^2\in B_2^N(u)$,
$$\ud(\basigma^1,\basigma^2)^2\leq 4N^{1/2}u^{3/2}\norm{\basigma^1-\basigma^2}_{2,2}.$$
Taking square roots yields \eqref{eqn: GP canonical metric bounded by norm}. To prove \eqref{eqn: GP modified coordinate canonical metric bound}, observe that for any $\basigma\in \Sigma_\epsilon(D)$,
\begin{equation}\label{eqn: GP bound on 2-2 norm of configuration in SigmaeD} 
\Norm{\basigma}_{2,2}^2=\tr(R(\basigma,\basigma))\leq \tr(D)+\epsilon \kappa\leq \tr(D)+1.
\end{equation}
Invoking \cref{GP modified coordinates norm bounds} and \eqref{eqn: GP canonical metric bounded by norm} implies that
$\ud(\basigma,A_{\basigma}\basigma)\leq C N^{1/2}(\tr(D)+1)\epsilon^{1/8}$. This completes the proof.
\end{proof}

Combining \cref{GP modified coordinates norm bounds} and \cref{GP modified coordinate canonical metric bound} with Dudley's entropy inequality, we will now show that, in a certain sense, the constrained Lagrangian \eqref{eqn: GP Lagrangian with self-overlap D} is continuous with respect to the constraint $D\in \Gamma_\kappa$. To state this continuity result precisely, for each $\epsilon>0$ and $D\in \Gamma_\kappa$ introduce the relaxed constrained Lagrangian
\begin{equation}\label{eqn: GP relaxed constrained Lagrangian}
L_{N,p,D,\epsilon}(t)=\frac{1}{N}\max_{\basigma\in \Sigma_\epsilon(D)}H_{N,p,t}(\basigma).
\end{equation}

\begin{proposition}\label{GP continuity of constrained Lagrangian}
If $2<p<\infty$, then for each $0<\epsilon<\kappa^{-2}$, every $t>0$ and all $D\in \Gamma_\kappa$,
\begin{equation}
\limsup_{N\to\infty} L_{N,p,D,\epsilon}(t)\leq  L_{p,D_\epsilon}(t)+C(1+tp)(\tr(D)+1)^{p/2}\epsilon^{1/64}
\end{equation}
for some constant $C>0$ that depends only on $\kappa$.
\end{proposition}

\begin{proof}
To simplify notation, let $C>0$ denote a constant that depends only on $\kappa$ whose value might not be the same at each occurrence. By the Gaussian concentration of the maximum and a simple application of the Borel-Cantelli lemma, it suffices to prove that
\begin{equation}\label{eqn: GP continuity of constrained Lagrangian goal}
\limsup_{N\to\infty} \frac{1}{N}\E\max_{\basigma\in \Sigma_\epsilon(D)}H_{N,p,t}(\basigma)\leq L_{p,D_\epsilon}(t)+C(1+tp)(\tr(D)+1)^{p/2}\epsilon^{1/64}.
\end{equation}
To simplify notation, let $u=\tr(D)+1$. Notice that $\Norm{\basigma}_{2,2}^2\leq u$ for every $\basigma\in \Sigma_\epsilon(D)$ by \eqref{eqn: GP bound on 2-2 norm of configuration in SigmaeD}. Invoking \cref{GP modified coordinate canonical metric bound} and \cref{GP modified coordinates norm bounds} gives
\begin{equation}\label{eqn: GP continuity of constrained Lagrangian}
\frac{1}{N}\E\max_{\basigma\in \Sigma_\epsilon(D)}H_{N,p,t}(\basigma)\leq \frac{1}{N}\E\max_{\basigma\in \Sigma(D_\epsilon)}H_{N,p,t}(\basigma)+\frac{1}{N}(I)+\frac{t}{N}(II),
\end{equation}
where
\begin{align*}
(I)&=\frac{1}{N}\E \max_{\basigma\in \Sigma_\epsilon(D)}\abs{H_N(\basigma)-H_N(A_{\basigma}\basigma)}\leq \E \max_{\ud(\basigma^1,\basigma^2)\leq CuN^{1/2}\epsilon^{1/8}}\abs{H_N(\basigma^1)-H_N(\basigma^2)}\\
(II)&=\max_{\basigma\in \Sigma_\epsilon(D)}\big(\Norm{A_{\basigma}\basigma}_{p,2}^p-\Norm{\basigma}_{p,2}^p\big).
\end{align*}
To bound the first of these terms, for each $\epsilon>0$ denote by $\mathcal{N}(A,d,\epsilon)$ the $\epsilon$-covering number of the set $A\subset (\R^\kappa)^N$ with respect to the metric $d$ on $(\R^\kappa)^N$, and write $B_N$ for the Euclidean unit ball in $(\R^\kappa)^N$. Dudley's entropy inequality and \cref{GP modified coordinate canonical metric bound} imply that
\begin{align*}
(I)&\leq C\int_0^{CuN^{1/2}\epsilon^{1/8}}\sqrt{\log \mathcal{N}\big(B_2^N(u),\ud,\delta\big)}\ud \delta\\
&\leq C \int_0^{CuN^{1/2}\epsilon^{1/8}}\sqrt{\log \mathcal{N}\big(B_2^N(u),\norm{\cdot}_{2,2},2^{-2}u^{-3/2}N^{-1/2}\delta^2\big)}\ud \delta\\
&\leq C\int_0^{CuN^{1/2}\epsilon^{1/8}}\sqrt{\log \mathcal{N}\big(B_N, \norm{\cdot}_{2,2}, 2^{-2}u^{-2}N^{-1}\delta^2\big)}\ud \delta.
\end{align*}
At this point, recall that the covering number of the Euclidean unit ball $B_N$ in $(\R^\kappa)^N$ satisfies
$$\Big(\frac{1}{\epsilon}\Big)^{N\kappa}\leq \mathcal{N}(B_N,\norm{\cdot}_{2,2},\epsilon)\leq \Big(\frac{2}{\epsilon}+1\Big)^{N\kappa}$$
for every $\epsilon>0$. A proof of this bound may be found in corollary 4.2.13 of \cite{Vershynin}. Combining this with a change of variables reveals that
\begin{align}\label{eqn: GP continuity of constrained Lagrangian bound (I)}
(I)&\leq CN^{1/2}\kappa^{1/2}\int_0^{CuN^{1/2}\epsilon^{1/8}}\sqrt{\log(1+8u^2N\delta^{-2})}\ud \delta \notag\\
&=CNu\int_{C\epsilon^{-1/16}}^\infty \frac{\sqrt{\log (1+\delta)}}{\delta^{3/2}}\ud \delta\leq CNu\int_{C\epsilon^{-1/16}}^\infty \frac{\sqrt{ \delta^{1/2}}}{\delta^{3/2}}\ud \delta \notag\\
&\leq CNu\epsilon^{1/64}.
\end{align}
To bound the term $(II)$, notice that for any $x,y>0$,
\begin{equation}\label{eqn: GP sum to power p}
(x+y)^p-x^p=\int_0^1 \frac{\ud}{\ud t}(x+ty)^p \ud t=p\int_0^1(x+ty)^{p-1}y\ud t\leq py(x+y)^{p-1}.
\end{equation}
If $\basigma\in \Sigma_\epsilon(D)$ is such that $\Norm{A_{\basigma}\basigma}_{p,2}> \Norm{\basigma}_{p,2}$, then applying this inequality with $x=\Norm{\basigma}_{p,2}$ and $y=\Norm{A_{\basigma}\basigma}_{p,2}-\Norm{\basigma}_{p,2}$ gives
\begin{align}\label{eqn: GP difference of p norms 1}
\Norm{A_{\basigma}\basigma}_{p,2}^p-\Norm{\basigma}_{p,2}^p&\leq p\Norm{A_{\basigma}\basigma-\basigma}_{p,2}\Norm{A_{\basigma}\basigma}_{p,2}^{p-1}\\
& \leq p\Norm{A_{\basigma}\basigma-\basigma}_{2,2}\Norm{A_{\basigma}\basigma}_{2,2}^{p-1}.\label{eqn: GP difference of p norms 2}
\end{align}
The second inequality uses the fact that $\ell^{2,2}$ is continuously embedded in $\ell^{p,2}$ for $p>2$. Since this bound holds trivially when $\Norm{A_{\basigma}\basigma}_{p,2}\leq  \Norm{\basigma}_{p,2}$, we deduce from \cref{GP modified coordinates norm bounds} that
\begin{equation}\label{eqn: GP continuity of constrained Lagrangian bound (II)}
(II)\leq Cpu^{p/2}\epsilon^{1/4}.
\end{equation}
Substituting \eqref{eqn: GP continuity of constrained Lagrangian bound (I)} and \eqref{eqn: GP continuity of constrained Lagrangian bound (II)} into \eqref{eqn: GP continuity of constrained Lagrangian} and letting $N\to\infty$ yields \eqref{eqn: GP continuity of constrained Lagrangian goal}. This completes the proof.
\end{proof}

In the heuristic proof of \cref{GP limit of unconstrained Lagrangian} given at the end of \cref{GP sec3}, we used the density of positive definite matrices in the space of non-negative definite matrices to obtain the second equality in \eqref{eqn: GP limit of unconstrained Lagrangian}. When we come to the rigorous proof of this equality, the argument will be more subtle as \cref{GP continuity of constrained Lagrangian} does not quite give continuity. We will instead content ourselves with controlling the limit of the constrained Lagrangian \eqref{eqn: GP Lagrangian with self-overlap D} for a non-negative definite matrix $D\in \Gamma_\kappa$ by that for some positive definite matrix in $\Gamma_\kappa^+$ through the following bound.

\begin{proposition}\label{GP replacing D by D+}
If $2<p<\infty$, then for each $0<\epsilon<\kappa^{-2}$, every $t>0$ and all $D\in \Gamma_\kappa$,
\begin{equation}
L_{p,D}(t)\leq L_{p,D+\epsilon I}+C(1+tp)(\tr(D)+1)^{p/2}\epsilon^{1/64}
\end{equation}
for some constant $C>0$ that depends only on $\kappa$.
\end{proposition}

\begin{proof}
Fix $N>2\kappa$ and $\basigma\in \Sigma(D)$. Endow $\R^N$ with the inner product
$$\langle \brho,\btau\rangle=\frac{1}{N}\sum_{i=1}^N \brho_i\btau_i.$$
Since $N>2\kappa$, there exist mutually orthogonal vectors $\btau_{\basigma}(1),\ldots,\btau_{\basigma}(\kappa)$ that are also orthogonal to each of the vectors $\bsigma(1),\ldots,\bsigma(\kappa)$ and satisfy $\langle \btau_{\basigma}(k),\btau_{\basigma}(k)\rangle=\kappa^{-1}$ for $1\leq k\leq \kappa$. Consider the configuration $\barho_{\basigma}\in (\R^\kappa)^N$ defined by $\brho_{\basigma}(k)=\bsigma+\sqrt{\epsilon}\btau_{\basigma}$. By orthogonality,
$$R(\barho,\barho)_{k,k'}=\langle \brho(k),\brho(k')\rangle=\langle\bsigma(k),\bsigma(k')\rangle+\epsilon \delta_{k,k'}=R(\basigma,\basigma)_{k,k'}+\epsilon\delta_{k,k'},$$
where $\delta_{k,k'}=1$ if $k=k'$ and is zero otherwise. This means that $\barho_{\basigma}\in \Sigma(D+\epsilon I)$. Moreover, the normalization of the vectors $\btau_{\basigma}(k)$ implies that
$$\Norm{\barho_{\basigma}-\basigma}_{2,2}^2=\frac{1}{N}\sum_{k=1}^\kappa \norm{\brho_{\basigma}(k)-\bsigma(k)}_2^2=\epsilon \sum_{k=1}^\kappa \langle \btau_{\basigma}(k),\btau_{\basigma}(k)\rangle=\epsilon.$$
If we let $u=\tr(D)+1$, then \eqref{eqn: GP canonical metric bounded by norm} reveals that
$$\ud (\barho_{\basigma},\basigma)\leq 2N^{1/2}u^{3/4}\epsilon^{1/4}\leq 2N^{1/2}u\epsilon^{1/8},$$
while an identical argument to that used to obtain \eqref{eqn: GP difference of p norms 2} yields
$$\Norm{\barho_{\basigma}}_{p,2}^p-\Norm{\basigma}_{p,2}^p\leq p\Norm{\barho_{\basigma}-\basigma}_{2,2}\Norm{\barho_{\basigma}}_{2,2}^{p-1}\leq p\sqrt{\epsilon}u^{p/2}.$$
It follows that
\begin{align*}
\frac{1}{N}H_{N,p,t}(\basigma)&\leq \frac{1}{N}H_{N,p,t}(\barho_{\basigma})+\frac{1}{N}\abs{H_N(\basigma)-H_N(\barho_{\basigma})}+t\big(\Norm{\barho_{\basigma}}_{p,2}^p-\Norm{\basigma}_{p,2}^p\big)\\
&\leq L_{N,p,D+\epsilon I}(t)+\frac{1}{N} \max_{\ud(\basigma^1,\basigma^2)\leq 2uN^{1/2}\epsilon^{1/8}}\abs{H_N(\basigma^1)-H_N(\basigma^2)}+tp\sqrt{\epsilon}u^{p/2}.
\end{align*}
Taking the maximum over configurations $\basigma\in \Sigma(D)$ and using Dudley's entropy inequality exactly as in the proof of \cref{GP continuity of constrained Lagrangian} gives
$$\E L_{N,p,D}(t)\leq \E L_{N,p,D+\epsilon I}(t)+C(1+tp)u^{p/2}\epsilon^{1/64}$$
for some constant $C>0$ that depends only on $\kappa$. Letting $N\to\infty$ completes the proof.
\end{proof}

The results established in this section together with the arguments in section 7 of \cite{WeiKuo} will allow us to give a rigorous proof of \cref{GP limit of unconstrained Lagrangian}. The proof will consist of two key steps. First, we will use \cref{GP continuity of constrained Lagrangian} to express a version of the Lagrangian \eqref{eqn: GP unconstrained Lagrangian} localized to a ball of fixed but arbitrary radius $u>0$ as a supremum of constrained Lagrangians \eqref{eqn: GP Lagrangian with self-overlap D}. Then, we will modify the scaling arguments in section 7 of \cite{WeiKuo} to show that the unconstrained Lagrangian \eqref{eqn: GP unconstrained Lagrangian} can be obtained by taking the supremum of these localized Lagrangians over all radii $u>0$. The formula obtained by taking these successive suprema will be equivalent to the first equality in \eqref{eqn: GP limit of unconstrained Lagrangian}. As previously mentioned, the second equality will follow immediately from \cref{GP replacing D by D+}. The purpose of restricting the supremum to positive definite matrices is technical and will be emphasized when we prove \cref{GP bound on almost maximizer}.

\section{The limit of the unconstrained Lagrangian}\label{GP sec5}

In this section we combine \cref{GP continuity of constrained Lagrangian} with the arguments in section 7 of \cite{WeiKuo} to prove \cref{GP limit of unconstrained Lagrangian}. As explained at the end of \cref{GP sec4}, we will first find a formula for the limit of the localized Lagrangian
\begin{equation}\label{GP: localized Lagrangian}
L_{N,p,u}(t)=\frac{1}{N}\max_{\Norm{\basigma}_{2,2}^2\leq u}H_{N,p,t}(\basigma)
\end{equation}
defined for each $u>0$. If $\Gamma_{\kappa,u}$ denotes the set of matrices in $\Gamma_\kappa$ with trace at most $u$,
\begin{equation}
\Gamma_{\kappa,u}=\{D\in \Gamma_\kappa\mid \tr(D)\leq u\},
\end{equation}
then \eqref{eqn: GP 2-2 norm as trace of self overlap} implies that for every $t>0$,
\begin{equation}\label{eqn: GP localized Lagrangian as supremum}
L_{N,p,u}(t)=\sup_{D\in \Gamma_{\kappa,u}}L_{N,p,D}(t).
\end{equation}
A compactness argument similar to that in lemma 3 of \cite{PanVec} can be used to show that this equality is preserved in the limit.

\begin{proposition}\label{GP limit of localized Lagrangian}
If $2<p<\infty$, then for every $t>0$ and $u>0$, the limit
$L_{p,u}(t)=\lim_{N\to\infty}\E L_{N,p,u}(t)$
exists and is given by
\begin{equation}
L_{p,u}(t)=\sup_{D\in \Gamma_{\kappa,u}}L_{p,D}(t).
\end{equation}
Moreover, with probability one, $L_{p,u}(t)=\lim_{N\to\infty}L_{N,p,u}(t)$.
\end{proposition}

\begin{proof}
Given $\epsilon>0$, observe that the collection of sets $B_\epsilon(D)$ for $D\in \Gamma_{\kappa,u}$ forms an open cover of the compact set $\Gamma_{\kappa,u}$. It is therefore possible to find integer $n\geq 1$ and $D^1,\ldots,D^n\in \Gamma_{\kappa,u}$ with $\smash{\Gamma_{\kappa,u}\subset \bigcup_{i\leq n}B_\epsilon(D^i)}$, or equivalently
$\smash{B_2^N(u)\subset \bigcup_{ i\leq n}\Sigma_\epsilon(D^i)}$. With this in mind, given a probability measure $\mu^N$ supported on $B_2^N(u)$, an inverse temperature parameter $\beta>0$ and a subset $S\subset B_2^N(u)$, consider the free energy
$$F_N^\beta(S)=\frac{1}{N\beta}\log \int_S \exp \beta H_{N,p,t}(\basigma)\ud \mu^N(\basigma).$$
By monotonicity of the logarithm and the inclusion $B_2^N(u)\subset \bigcup_{ i\leq n}\Sigma_\epsilon(D^i)$,
\begin{align*}
F_N^\beta(B_2^N(u))&\leq \frac{\log n}{N\beta}+\frac{1}{N\beta}\log\max_{1\leq i\leq n}\int_{\Sigma_\epsilon(D^i)}\exp \beta H_{N,p,t}(\basigma)\ud \mu^N(\basigma)\\
&=\frac{\log n}{N\beta}+\max_{1\leq i\leq n}F_N^\beta(\Sigma_\epsilon(D^i)).
\end{align*}
The Gaussian concentration inequality implies that the free energy $F_N^\beta(S)$ deviates from its expectation by more than $1/\sqrt{N}$ with probability at most $Le^{-N/L}$, where the constant $L$ does not depend on $\beta,N$ or $S$. We deduce from this that with probability at least $1-Le^{-N/L}$,
$$F_N^\beta(B_2^N(u))\leq \frac{1}{\sqrt{N}}+\frac{\log n}{N\beta}+\max_{1\leq i\leq n}\E F_N^\beta(\Sigma_\epsilon(D^i)).$$
Letting $\beta\to \infty$ and remembering that the $L^q$-norm converges to the $L^\infty$-norm reveals that with probability at least $1-Le^{-N/L}$,
$$L_{N,p,u}(t)\leq \frac{2}{\sqrt{N}}+\max_{1\leq i\leq n}L_{N,p,D^i,\epsilon}(t).$$
The Borel-Cantelli lemma and \cref{GP continuity of constrained Lagrangian} now give a constant $C>0$ that depends only on $\kappa$ with
$$\limsup_{N\to\infty} L_{N,p,u}(t)\leq \max_{1\leq i\leq n}\big(L_{p,D_\epsilon^i}(t)+C(1+tp)(\tr(D^i)+1)^{p/2}\epsilon^{1/64}\big).$$
Since $\tr(D_\epsilon^i)\leq \tr(D^i)\leq u$, this can be bounded further by
$$\limsup_{N\to\infty} L_{N,p,u}(t)\leq \sup_{D\in \Gamma_{\kappa,u}}L_{p,D}(t)+C(1+tp)(u+1)^{p/2}\epsilon^{1/64}.$$
Remembering that $L_{N,p,D}(t)\leq L_{N,p,u}(t)$ for every $N\geq 1$ and $D\in \Gamma_{\kappa,u}$, it follows that
\begin{align*}
\sup_{D\in \Gamma_{\kappa,u}}L_{p,D}(t)&\leq \liminf_{N\to\infty}L_{N,p,u}(t)\\
&\leq \limsup_{N\to\infty}L_{N,p,u}(t)\leq \sup_{D\in \Gamma_{\kappa,u}}L_{p,D}(t)+C(1+tp)(u+1)^{p/2}\epsilon^{1/64}
\end{align*}
Letting $\epsilon\to 0$ and using the Gaussian concentration of the maximum completes the proof.
\end{proof}

This result reduces the proof of \cref{GP limit of unconstrained Lagrangian} to establishing the asymptotic version of the equality
\begin{equation}\label{eqn: GP unconstrained Lagrangian supremum of localized Lagrangian}
L_{N,p}(t)=\sup_{u>0}L_{N,p,u}(t).
\end{equation}
This will be done using the techniques in section 7 of \cite{WeiKuo} and relying upon the identity
\begin{equation}\label{eqn: GP localized Lagrangian re-scaled}
L_{N,p,u}(t)=\frac{1}{N}\max_{\Norm{\basigma}_{2,2}^2\leq u}H_{N,p,t}(\basigma)=\frac{1}{N}\max_{\Norm{\basigma}_{2,2}\leq 1}\big(uH_N(\basigma)-tu^{p/2}\norm{\basigma}_{p,2}^p\big)
\end{equation}
which holds for every $t,u>0$ by a change of variables. The absence of such an equality at the level of the constrained Lagrangian \eqref{eqn: GP Lagrangian with self-overlap D} is the reason we had to develop the results in \cref{GP sec4}.

For technical reasons, before we start thinking about proving the asymptotic version of \eqref{eqn: GP unconstrained Lagrangian supremum of localized Lagrangian}, we will have to upgrade the statement of \cref{GP limit of localized Lagrangian} to show that $L_{p,u}(t)$ is the limit of the localized Lagrangian \eqref{GP: localized Lagrangian} with probability one simultaneously over all $t,u>0$. Heuristically, this should not be too surprising. As the maximum of a collection of concave functions, the localized Lagrangian \eqref{eqn: GP localized Lagrangian re-scaled} is concave in the pair $(u,t)$ conditionally on the disorder chaos $(g_{ij})$. Since a concave function is Lipschitz continuous on compact sets, this suggests that $(u,t)\mapsto L_{N,p,u}(t)$ should be Lipschitz continuous on compact sets. This continuity would immediately promote almost sure convergence for each $t,u>0$ to a convergence with probability one simultaneously over all $t,u>0$. To make this argument rigorous, we will use an $\ell^2$-boundedness result of the $N\times N$ random matrix
\begin{equation}\label{GP GN matrix}
G_N=(g_{ij})_{i,j\leq N}.
\end{equation}
Its proof will rely upon Chevet's inequality as it appears in theorem 8.7.1 of \cite{Vershynin}.

\begin{lemma}\label{GP GN bounded with high probability}
There exist constants $C,M>0$ such that with probability at least $1-Ce^{-N/C}$,
\begin{equation}
\frac{1}{\sqrt{N}}\norm{G_N}_2\leq M.
\end{equation}
\end{lemma}

\begin{proof}
Since $\norm{G_N}_2=\max_{\norm{x}_2=1}(G_Nx,x)$ and $\E(G_Nx,x)^2=1$ whenever $\norm{x}_2=1$, the Gaussian concentration of the maximum gives a constant $C>0$ such that with probability at least $1-Ce^{-N/C}$,
$$\frac{1}{\sqrt{N}}\norm{G_N}_2\leq \frac{1}{\sqrt{N}}\E\norm{G_N}_2+1.$$
If $g$ is a standard Gaussian vector in $\R^N$, then Chevet's inequality applied with $T$ and $S$ equal to the Euclidean unit ball in $\R^N$ gives an absolute constant $M>0$ with
\begin{equation}\label{eqn: GP GN bounded with high probability key}
\E\norm{G_N}_2=\E\max_{\norm{x}_2=1}(G_Nx,x)\leq M \E \norm{g}_2.
\end{equation}
We have used the fact that the Gaussian width of the unit ball is $\E\sup_{\norm{x}_2=1}(g,x)=\E\norm{g}_2$ while its radius is one. Finally, Jensen's inequality reveals that
$$(\E \norm{g}_2)^2\leq \E \norm{g}_2^2=N\E\abs{g_1}^2=N.$$
Substituting this into \eqref{eqn: GP GN bounded with high probability key} and redefining the constant $M>0$ completes the proof.
\end{proof}

\begin{lemma}\label{GP localized Lagrangian Lipschitz}
If $2<p<\infty$, then for any $0<K_1<K_2$, there exist constants $C,M>0$ such that with probability at least $1-Ce^{-N/C}$,
\begin{equation}
\abs{L_{N,p,u}(t)-L_{N,p,u'}(t')}\leq M\big(\abs{u-u'}+\abs{t-t'}\big)
\end{equation}
for all $t,t',u,u'\in [K_1,K_2]$.
\end{lemma}

\begin{proof}
Let $\barho\in B_2^N(1)$ maximize the right-hand side of \eqref{eqn: GP localized Lagrangian re-scaled}, and define the vector spin configuration $\batau\in B_2^N(1)$ by $\atau_i=(\kappa^{-1/2},\ldots,\kappa^{-1/2})\in \R^\kappa$ for $1\leq i\leq N$. The Cauchy-Schwarz inequality shows that
\begin{align*}
uN^{1/2}\norm{G_N}_2-tu^{p/2}\norm{\barho}_{p,2}^p&\geq \frac{u}{\sqrt{N}}\sum_{k=1}^\kappa (G_N\brho(k),\brho(k))-tu^{p/2}\norm{\barho}_{p,2}^p\\
&\geq -\frac{u}{\sqrt{N}}\sum_{k=1}^\kappa \abs{(G_N\btau(k),\btau(k))}-tu^{p/2}\norm{\batau}_{p,2}^p\\
&\geq -uN^{1/2}\norm{G_N}_2-tu^{p/2}N.
\end{align*}
Rearranging and using the fact that $p>2$ gives
$$\Norm{\barho}_{p,2}^p\leq \frac{2u^{1-p/2}\norm{G_N}_2}{t\sqrt{N}}+1\leq \frac{2\norm{G_N}_2}{K_1^{p/2}\sqrt{N}}+1.$$
It follows by \eqref{eqn: GP localized Lagrangian re-scaled}, the Cauchy-Schwarz inequality and the mean value theorem that for any $u',t'\in [K_1,K_2]$,
\begin{align*}
L_{N,p,u}(t)-L_{N,p,u'}(t')&\leq N^{-1}\abs{u-u'}H_N(\barho)-\big(tu^{p/2}-t'u'^{p/2}\big)\Norm{\barho}_{p,2}^p\\
&\leq \frac{\norm{G_N}_2}{\sqrt{N}}\abs{ u-u'}+\Norm{\barho}_{p,2}^p\big(K_2^{p/2}\abs{ t-t'}+pK_2^{1+p/2}K_1^{-1}\abs{u-u'}\big)\\
&\leq M\frac{\norm{G_N}_2}{\sqrt{N}}(\abs{u-u'}+\abs{t-t'})
\end{align*}
for some constant $M>0$ that depends only on $K_1,K_2$ and $p$. Interchanging the roles of $u,u'$ and $t,t'$, it is easy to see that
$$\abs{L_{N,p,u}(t)-L_{N,p,u'}(t')}\leq M\frac{\norm{G_N}_2}{\sqrt{N}}\big(\abs{u-u'}+\abs{t-t'}\big).$$
Invoking \cref{GP GN bounded with high probability} and redefining the constant $M$ completes the proof.
\end{proof}

\begin{proposition}\label{GP simultaneous limit of localized Lagrangian}
If $2<p<\infty$, then almost surely
\begin{equation}
L_{p,u}(t)=\lim_{N\to\infty}L_{N,p,u}(t)
\end{equation}
for every $t,u>0$.
\end{proposition}

\begin{proof}
By \cref{GP localized Lagrangian Lipschitz} and a simple application of the Borel-Cantelli lemma, for any $0<K_1<K_2$ there exists some constant $M=M(K_1,K_2)$ such that almost surely
\begin{equation}\label{eqn: GP limit of localized Lagrangian Lipschitz}
\limsup_{N\to\infty}\abs{L_{N,p,u}(t)-L_{N,p,u'}(t')}\leq M\big(\abs{ u-u'}+\abs{t-t'}\big)
\end{equation}
for all $u,u',t,t'\in [K_1,K_2]$. Since $L_{p,u}(t)$ is a deterministic quantity, we also have
\begin{equation}\label{eqn: GP deterministic limit of localized Lagrangian Lipschitz}
\abs{L_{p,u}(t)-L_{p,u'}(t')}\leq M\big(\abs{u-u'}+\abs{t-t'}\big)
\end{equation}
for all $u,u',t,t'\in [K_1,K_2]$. By countability of rationals and \cref{GP limit of localized Lagrangian}, we can find a set $\O$ of probability one where \eqref{eqn: GP limit of localized Lagrangian Lipschitz} holds simultaneously for all rationals $K_1,K_2\in \Q_+$ and at the same time $L_{p,u}(t)=\lim_{N\to\infty} L_{N,p,u}(t)$ for all $u,t\in \Q_+$. The triangle inequality implies that for any $u,t>0$ and $u',t'\in \Q_+$,
\begin{align*}
\abs{L_{N,p,u}(t)-L_{p,u}(t)}&\leq \abs{ L_{N,p,u}(t)-L_{N,p,u'}(t')}+\abs{L_{N,p,u'}(t')-L_{p,u'}(t')}\\
&\quad+\abs{L_{p,u'}(t')-L_{p,u}(t)}.
\end{align*}
It follows by \eqref{eqn: GP deterministic limit of localized Lagrangian Lipschitz} that on the set $\O$,
$$\limsup_{N\to\infty}\abs{L_{N,p,u}(t)-L_{p,u}(t)}\leq 2M\big(\abs{u-u'}+\abs{t-t'}\big).$$
Letting $u'\to u$ and $t'\to t$ along rational points completes the proof. 
\end{proof}

In addition to \cref{GP simultaneous limit of localized Lagrangian}, the proof of \cref{GP limit of unconstrained Lagrangian} will rely on the fact that the $\ell^{p,2}$-norm potential in the definition of the Hamiltonian \eqref{eqn: GP unconstrained Hamiltonian} forces the maximizers of this random function to concentrate in a large enough neighbourhood of the origin with overwhelming probability.

\begin{lemma}\label{GP Lagrangian optimizers bounded}
If $2<p<\infty$, then there exist constants $C,M>0$ such that with probability at least $1-Ce^{-N/C}$,
\begin{equation}
L_{N,p}(t)= L_{N,p,M/t}(t)
\end{equation}
for all $t>0$.
\end{lemma}

\begin{proof}
Given $\basigma\in (\R^\kappa)^N$ with $H_{N,p,t}(\basigma)\geq 0$, the Cauchy-Schwarz inequality implies that
$$t\Norm{\basigma}_{p,2}^p\leq \frac{1}{N}H_N(\basigma)\leq \frac{\norm{G_N}_2}{\sqrt{N}}\Norm{\basigma}_{2,2}^2.$$
It follows by Jensen's inequality that
$$\Norm{\basigma}_{2,2}^p\leq \Norm{\basigma}_{p,2}^p\leq \frac{\norm{G_N}_2}{t\sqrt{N}}\Norm{\basigma}_{2,2}^2.$$
Since $L_{N,p}(t)\geq \frac{1}{N}H_{N,p,t}(0)=0$, rearranging shows that
$$L_{N,p}(t)=\frac{1}{N}\max\Big\{H_{N,p,t}(\basigma)\mid \Norm{\basigma}_{2,2}\leq \Big(\frac{\norm{G_N}_2}{t\sqrt{N}}\Big)^{1/(p-2)}\Big\}.$$
Invoking \cref{GP GN bounded with high probability} completes the proof.
\end{proof}

\begin{proof}[Proof (\Cref{GP limit of unconstrained Lagrangian}).]
By \cref{GP Lagrangian optimizers bounded}, there exist constants $C,M>0$ such that with probability at least $1-Ce^{-N/C}$,
$$L_{N,p}(t)=L_{N,p,M/t}(t)$$
for any $t>0$.
It follows by a simple application of the Borel-Cantelli lemma and \cref{GP simultaneous limit of localized Lagrangian} that with probability one,
$$L_{p,u}(t)=\lim_{N\to\infty}L_{N,p,u}(t)\leq \liminf_{N\to\infty}L_{N,p}(t)\leq \limsup_{N\to\infty}L_{N,p}(t)= L_{p,M/t}(t)\leq \sup_{u>0}L_{p,u}(t)$$
for every $t>0$ and $u>0$. Taking the supremum over all $u>0$ gives the almost sure existence of $L_p(t)$, and invoking \cref{GP limit of localized Lagrangian} shows that
\begin{equation}\label{eqn: GP limit of unconstrained Lagrangian key}
L_{p}(t)=\sup_{u>0}L_{p,u}(t)=\sup_{D\in \Gamma_\kappa}L_{p,D}(t).
\end{equation}
To establish the second equality in \eqref{eqn: GP limit of unconstrained Lagrangian}, fix a non-negative definite matrix $D\in \Gamma_{\kappa,u}$ as well as $0<\epsilon<\kappa^{-2}$. By \cref{GP replacing D by D+}, there exists a constant $K>0$ that depends only on $\kappa$ such that
$$L_{p,D}(t)\leq L_{p,D+\epsilon I}(t)+K(1+tp)(u+1)^{p/2}\epsilon^{1/64}.$$
It is readily verified that $D+\epsilon I\in \Gamma_\kappa^+$, so in fact
$$L_{p,D}(t)\leq \sup_{D\in \Gamma_\kappa^+}L_{p,D}(t)+K(1+tp)(u+1)^{p/2}\epsilon^{1/64}.$$
Taking the supremum over all $D\in \Gamma_{\kappa,u}$, letting $\epsilon \to 0$ and remembering \eqref{eqn: GP limit of unconstrained Lagrangian key} completes the proof.
\end{proof}

\section{The ground state energy in terms of the Lagrangian}\label{GP sec6}

In \cref{GP sec5} we proved the first noteworthy result of this paper by expressing the unconstrained Lagrangian \eqref{eqn: GP unconstrained Lagrangian} as a supremum of constrained Lagrangians \eqref{eqn: GP Lagrangian with self-overlap D} in the limit. As we will see in \cref{GP sec7} and \cref{GP sec8}, the constrained Lagrangian \eqref{eqn: GP Lagrangian with self-overlap D} can be understood using the results in \cite{PanVec}. It is for this reason that we constrained the Lagrangian \eqref{eqn: GP unconstrained Lagrangian} in the first place. However, the task that we originally set ourselves is understanding the $\ell^p$-Gaussian-Grothendieck problem with vector spins \eqref{eqn: GP vector}. In this section we connect the unconstrained Lagrangian \eqref{eqn: GP unconstrained Lagrangian} and the ground state energy \eqref{eqn: GP reformulated} by proving \cref{GP in terms of limiting Lagrangian}. This will reduce the $\ell^p$-Gaussian-Grothendieck problem with vector spins to understanding the asymptotic behaviour of the constrained Lagrangian \eqref{eqn: GP Lagrangian with self-overlap D}.

Before we proceed with the proof of \cref{GP in terms of limiting Lagrangian}, we give a formal argument that will motivate the results in this section. Given $N\in \N$ and $t>0$, let $\barho(t)$ be a point at which the Hamiltonian $H_{N,p,t}$ defined in \eqref{eqn: GP unconstrained Lagrangian} attains its supremum. Differentiating the expression $L_{N,p}(t)=\frac{1}{N}H_{N,p,t}(\barho(t))$ shows that
\begin{equation}\label{GP derivative of Lagrangian heuristic}
L'_{N,p}(t)=\frac{1}{N}\partial_tH_{N,p,t}(\barho(t))+\frac{1}{N}\big(\barho'(t),\nabla_{\basigma}H_{N,p,t}(\barho(t))\big)=-\Norm{\barho(t)}_{p,2}^p.
\end{equation}
We have used the fact that $\nabla_{\basigma} H_{N,p,t}(\barho(t))=0$. This suggests that
\begin{equation}
L_{N,p}(t)=\frac{1}{N}\max_{\Norm{\basigma}_{p,2}^p=-L_{N,p}'(t)}H_{N,p,t}(\basigma)=\frac{1}{N}\max_{\Norm{\basigma}_{p,2}^p=-L_{N,p}'(t)}H_N(\basigma)+tL_{N,p}'(t),
\end{equation}
and therefore
\begin{equation}\label{eqn: GP GSE in terms of Lagrangian heuristic}
\GSE_{N,p}=\frac{1}{N}\max_{\Norm{\basigma}_{p,2}^p=-L_{N,p}'(t)}H_N\big((-L_{N,p}'(t))^{-1/p}\basigma\big)=\frac{L_{N,p}(t)-tL_{N,p}'(t)}{(-L_{N,p}'(t))^{2/p}}.
\end{equation}
To express this ground state energy entirely in terms of the unconstrained Lagrangian \eqref{eqn: GP unconstrained Lagrangian} as in \cref{GP in terms of limiting Lagrangian}, we compute the gradient of the Hamiltonian \eqref{eqn: GP unconstrained Lagrangian}. Since our calculation will be rigorous, we formulate it as a lemma.

\begin{lemma}\label{GP gradient of Hamiltonian}
If $\basigma\in (\R^\kappa)^N$ and $t,u>0$, then, conditionally on the disorder chaos $(g_{ij})$,
\begin{equation}
\big(\nabla_{\basigma}H_{N,p,t}(\basigma),\basigma\big)=2H_N(\basigma)-tp\norm{\basigma}_{p,2}^p.
\end{equation}
\end{lemma}

\begin{proof}
Given $1\leq i\leq N$ and $1\leq k\leq \kappa$, a simple computation shows that
$$\frac{\partial H_{N,p,t}(\basigma)}{\partial \sigma_i(k)}=\frac{1}{\sqrt{N}}\sum_{j=1}^N (g_{ij}+g_{ji})\sigma_j(k)-tp\sigma_i(k)\norm{\asigma_i}_2^{p-2}.$$
It follows that
\begin{align*}
\big(\nabla_{\basigma}H_{N,p,t}(\basigma),\basigma\big)&=\sum_{k=1}^\kappa \frac{1}{\sqrt{N}}\sum_{i,j=1}^N (g_{ij}+g_{ji})\sigma_j(k)\sigma_i(k)-tp  \sum_{i=1}^N\sum_{k=1}^\kappa \norm{\asigma_i}_2^{p-2}\sigma_i(k)^2\\
&=2\sum_{k=1}^\kappa H_N^k(\bsigma(k))-tp\sum_{i=1}^N\norm{\asigma_i}_2^p=2H_N(\basigma)-tp\norm{\basigma}_{p,2}^p.
\end{align*}
This finishes the proof.
\end{proof}

\noindent This simple calculation suggests that
\begin{equation}
0=\big(\nabla_{\basigma}H_{N,p,t}(\barho(t)),\barho(t)\big)=2H_N(\barho(t))-tp\norm{\barho(t)}_{p,2}^p,
\end{equation}
which combined with \eqref{GP derivative of Lagrangian heuristic} gives
\begin{equation}\label{GP derivative of Lagrangian in terms of Lagrangian heuristic}
L_{N,p}(t)=t\Big(\frac{p}{2}-1\Big)\Norm{\barho(t)}_{p,2}^p=-t\Big(\frac{p}{2}-1\Big)L_{N,p}'(t).
\end{equation}
Substituting this into \eqref{eqn: GP GSE in terms of Lagrangian heuristic} gives \eqref{eqn: GP GSE in terms of limiting Lagrangian} upon letting $N\to \infty$. The problem with this argument is that the map $t\mapsto \barho(t)$ might not be differentiable. To overcome this issue, we will prove \eqref{GP derivative of Lagrangian in terms of Lagrangian heuristic} directly at the points of differentiability of $L_{N,p}(t)$. We will then use a convexity argument to deduce that \eqref{GP derivative of Lagrangian in terms of Lagrangian heuristic} holds for every $t>0$ in the limit.

\begin{lemma}\label{GP derivative of Lagrangian in terms of Lagrangian}
If $(g_{ij})$ is a realization of the disorder chaos for which the unconstrained Lagrangian $L_{N,p}$ is differentiable at $t>0$, then
\begin{equation}
L_{N,p}(t)=-t\Big(\frac{p}{2}-1\Big)L_{N,p}'(t).
\end{equation}
\end{lemma}

\begin{proof}
Fix $\epsilon>0$ and $\lambda>0$. For any configuration with $\Norm{\basigma}_{p,2}^p\geq -L_{N,p}'(t)+\epsilon$,
\begin{align*}
\frac{1}{N}H_{N,p,t}(\basigma)&\leq \frac{1}{N}H_{N,p,t}(\basigma)+\lambda\big(\Norm{\basigma}_{p,2}^p+L_{N,p}'(t)-\epsilon\big)\\
&\leq L_{N,p}(t-\lambda)+\lambda L_{N,p}'(t)-\lambda \epsilon\\
&=\lambda\Big(L_{N,p}'(t)-\frac{L_{N,p}(t)-L_{N,p}(t-\lambda)}{\lambda}\Big)-\lambda \epsilon+L_{N,p}(t).
\end{align*}
Similarly, for any configuration with $\Norm{\basigma}_{p,2}^p\leq -L_{N,p}'(t)-\epsilon$,
\begin{align*}
\frac{1}{N}H_{N,p,t}(\basigma)&\leq \frac{1}{N}H_{N,p,t}(\basigma)+\lambda \big(-\Norm{\basigma}_{p,2}^p-L_{N,p}'(t)-\epsilon\big)\\
&\leq L_{N,p}(t+\lambda)-\lambda L_{N,p}'(t)-\lambda \epsilon\\
&=\lambda\Big(\frac{L_{N,p}(t+\lambda)-L_{N,p}(t)}{\lambda}-L_{N,p}'(t)\Big)-\lambda \epsilon+L_{N,p}(t).
\end{align*}
The differentiability of $L_{N,p}$ at $t$ gives $\lambda=\lambda(\epsilon)>0$ small enough so that
$$\frac{1}{N}\max_{\abs{\Norm{\basigma}_{p,2}^p+L_{N,p}'(t)}\geq \epsilon}H_{N,p,t}(\basigma)\leq L_{N,p}(t)-\frac{\lambda\epsilon}{2}.$$
This means that an optimizer $\barho(t)$ of $L_{N,p}(t)$ satisfies $\abs{\Norm{\barho(t)}_p^p+L_{N,p}'(t)}< \epsilon$ for every $\epsilon>0$. Letting $\epsilon\to 0$ reveals that $\Norm{\barho(t)}_{p,2}^p=-L_{N,p}'(t)$.
It follows by \cref{GP gradient of Hamiltonian} that
$$L_{N,p}(t)-tL_{N,p}'(t)=\frac{1}{N}H_N(\barho(t))=\frac{tp}{2}\Norm{\barho(t)}_{p,2}^p=-\frac{tp}{2}L_{N,p}'(t).$$
Rearranging completes the proof.
\end{proof}

\begin{lemma}\label{GP derivative of Lagrangian}
If $2<p<\infty$, then the function $L_p(t)$ is differentiable on $(0,\infty)$ with
\begin{equation}
L_p(t)=-t\Big(\frac{p}{2}-1\Big)L_p'(t).
\end{equation}
\end{lemma}

\begin{proof}
Using \cref{GP limit of unconstrained Lagrangian}, fix a realization $(g_{ij})$ of the disorder chaos for which $L_{N,p}(t)$ converges to $L_p(t)$ for all $t>0$. Notice that $L_{N,p}$ and $L_p$ are convex functions. In particular, they are continuous everywhere on $(0,\infty)$ and differentiable almost everywhere on $(0,\infty)$. If $0<t_1<s<t_2$ are such that $L_{N,p}$ is differentiable at $s$, then the convexity of $L_{N,p}$ gives
$$\frac{L_{N,p}(s)-L_{N,p}(t_1)}{s-t_1}\leq L_{N,p}'(s)\leq \frac{L_{N,p}(t_2)-L_{N,p}(s)}{t_2-s},$$
and \cref{GP derivative of Lagrangian in terms of Lagrangian} yields
$$\frac{L_{N,p}(s)-L_{N,p}(t_1)}{s-t_1}\leq -\frac{L_{N,p}(s)}{(\frac{p}{2}-1)s}\leq \frac{L_{N,p}(t_2)-L_{N,p}(s)}{t_2-s}.$$
By continuity of $L_{N,p}$ and density of the points of differentiability of $L_{N,p}$ in $(0,\infty)$, this inequality implies that for all $0<t_1<t<t_2<\infty$,
$$\frac{L_{N,p}(t)-L_{N,p}(t_1)}{t-t_1}\leq -\frac{L_{N,p}(t)}{(\frac{p}{2}-1)t}\leq \frac{L_{N,p}(t_2)-L_{N,p}(t)}{t_2-t}.$$
Letting $N\to\infty$ and then letting $t_1\inc t$ and $t_2\dec t$ shows that at any point $t\in (0,\infty)$ of differentiability of $L_p$,
\begin{equation}\label{eqn: GP derivative of Lagrangian key}
L_p'(t)=-\frac{L_p(t)}{(\frac{p}{2}-1)t}.
\end{equation}
We will now use this equality to show that $L_p$ is differentiable everywhere on $(0,\infty)$. By convexity of $L_p$ and theorem 25.1 in \cite{Rockafellar}, it suffices to prove that the sub-differential $\partial L_p(t)$ consists of a single point for every $t>0$. Fix $t\in (0,\infty)$ as well as $a\in \partial L_p(t)$, and let $(s_k)$ and $(t_k)$ be points of differentiability of $L_p$ with $t_k\inc t$ and $s_k\dec t$. By definition of the sub-differential,
$$L_p'(t_k)\leq\frac{L_p(t)-L_p(t_k)}{t-t_k}\leq a\leq \frac{L_p(s_k)-L_p(t)}{s_k-t}\leq L_p'(s_k)$$
for every integer $k\geq 1$. Letting $k\to\infty$ and combining \eqref{eqn: GP derivative of Lagrangian key} with the continuity of $L_p$ yields
$$-\frac{L_p(t)}{(\frac{p}{2}-1)t}=\limsup_{k\to\infty}L'_p(t_k)\leq a\leq\liminf_{k\to\infty}L_p'(s_k)=-\frac{L_p(t)}{(\frac{p}{2}-1)t}.$$
This completes the proof.
\end{proof}

To leverage this result into a proof of \cref{GP in terms of limiting Lagrangian}, we must verify the legitimacy of the change of variables used in \eqref{eqn: GP GSE in terms of Lagrangian heuristic}. In other words, we must show that $L_p'(t)$ does not vanish on $(0,\infty)$. Our proof will rely upon the properties of the eigenvalues and eigenvectors of the Gaussian orthogonal ensemble discussed in chapter 2 of \cite{AGZ}. Recall the definition of the random matrix $G_N$ in \eqref{GP GN matrix}, and notice that the $N\times N$ random matrix
\begin{equation}\label{eqn: GP barGN matrix}
\bar{G}_N=\frac{G_N+G_N^T}{\sqrt{2}}    
\end{equation}
is distributed according to the Gaussian orthogonal ensemble.

\begin{lemma}\label{GP Lagrangian derivative strictly positive}
If $2<p<\infty$, then the function $L_p$ is strictly positive on $(0,\infty)$. In particular, $L_p'(t)<0$ for every $t>0$.
\end{lemma}

\begin{proof}
Given $\bsigma\in \R^N$, consider the vector spin configuration $\basigma\in (\R^\kappa)^N$ defined by
$$\basigma(k)=
\begin{cases}
\bsigma& \text{if } k=1,\\
0& \text{otherwise}.
\end{cases}$$
Notice that $\norm{\basigma}_{p,2}^p=\norm{\bsigma}_p^p=1$ and $\sum_{i,j=1}^N g_{ij}\sigma_i\sigma_j=\sum_{i,j=1}^N g_{ij}\big(\asigma_i,\asigma_j\big)$. It follows that
\begin{equation}\label{eqn: GP Lagrangian derivative strictly positive key}
L_{N,p}(t)\geq \frac{1}{N}\big(H_N(\basigma)-t\norm{\basigma}_{p,2}^p\big)=\frac{\big(\bar{G}_N\bsigma,\bsigma\big)}{\sqrt{2}N^{3/2}}-\Norm{\bsigma}_p^p.
\end{equation}
With this in mind, let $v$ denote the $\ell^2$-normalized eigenvector associated with the largest eigenvalue $\lambda_N^N$ of the Gaussian orthogonal ensemble $\bar G_N$. Given $\delta>0$, applying \eqref{eqn: GP Lagrangian derivative strictly positive key} to the spin configuration $\smash{\bsigma_\delta=\sqrt{N\delta}v}$ reveals that
$$L_{N,p}(t)\geq \frac{\big(\bar{G}_N\bsigma_\delta,\bsigma_\delta\big)}{\sqrt{2}N^{3/2}}-t\Norm{\bsigma_\delta}_p^p=\frac{\delta}{\sqrt{2}} \frac{\lambda_N^N}{\sqrt{N}}-t\delta^{p/2}N^{p/2-1}\norm{v}_p^p.$$
By corollary 2.5.4 in \cite{AGZ}, the eigenvector $v$ is equal in distribution to $g/\norm{g}_2$ for a standard Gaussian random vector $g$ in $\R^N$. Moreover, by the strong law of large numbers,
$$N^{\frac{p}{2}-1}\frac{\norm{g}_p^p}{\norm{g}_2^p}=\frac{\frac{1}{N}\sum_{i\leq N}\abs{g_i}^p}{(\frac{1}{N}\sum_{i\leq N}\abs{g_i}^2)^{p/2}}\longrightarrow \frac{\E\abs{g_1}^p}{(\E\abs{g_1}^2)^{p/2}}=\E\abs{g_1}^p$$
almost surely. Together with the asymptotics of $\lambda^N_N$ established in theorem 2.1.22 of \cite{AGZ}, this implies that
$$L_p(t)\geq \sqrt{2}\delta-t\delta^{p/2}\E\abs{g_1}^p.$$
Taking $\delta>0$ small enough  and using the fact that $p>2$ shows that $L_p$ is strictly positive on $(0,\infty)$. Invoking \cref{GP derivative of Lagrangian} completes the proof.
\end{proof}

\begin{proof}[Proof (\Cref{GP in terms of limiting Lagrangian}).]
Using \cref{GP limit of unconstrained Lagrangian}, fix a realization $(g_{ij})$ of the disorder chaos for which $L_{N,p}(t)$ converges to $L_p(t)$ for all $t>0$. Let $\O\subset (0,\infty)$ be the collection of points at which $L_{N,p}$ is differentiable for all $N\geq 1$. Fix $t\in \O$, and notice that by convexity of $L_{N,p}$,
\begin{equation}\label{eqn: GP in terms of limiting Lagrangian key}
L_{N,p}(t+h)\geq L_{N,p}(t)+L_{N,p}'(t)h
\end{equation}
for every $h\in \R$. By \cref{GP derivative of Lagrangian in terms of Lagrangian}, the sequence $(L_{N,p}'(t))_N$ is uniformly bounded. It therefore admits a subsequential limit $a$. Letting $N\to\infty$ in \eqref{eqn: GP in terms of limiting Lagrangian key} shows that $a$ belongs to the sub-differential $\partial L_p(t)$. Invoking \cref{GP derivative of Lagrangian} shows that $a=L_p'(t)$, and therefore $L_{N,p}'(t)\to L_p'(t)$. It follows by \cref{GP Lagrangian derivative strictly positive} that $L_{N,p}(t)<0$ for large enough $N$, so 
$$\GSE_{N,p}=\frac{1}{N}\max_{\Norm{\basigma}_{p,2}^p=-L_{N,p}'(t)}H_N\big((-L_{N,p}'(t))^{-1/p}\basigma\big)=\frac{L_{N,p}(t)-tL_{N,p}'(t)}{(-L_{N,p}'(t))^{2/p}}.$$
Since $\O$ is dense in $(0,\infty)$ and $L_{N,p}'$ is continuous, this equality extends to all $t>0$. Letting $N\to \infty$ and using \cref{GP derivative of Lagrangian} completes the proof.
\end{proof}

\section{Replacing the constrained Lagrangian by a free energy}\label{GP sec7}

So far, we have reduced the $\ell^p$-Gaussian-Grothendieck problem with vector spins to understanding the asymptotic behaviour of the constrained Lagrangian \eqref{eqn: GP Lagrangian with self-overlap D} with positive definite constraints. This task will occupy the remainder of the paper. The starting point of our analysis will be the Parisi-type variational formula for free energy functionals established in \cite{PanVec}. To access this result, we must first replace the constrained Lagrangian by a free energy functional. In this section, given a constraint $D\in \Gamma_\kappa$ which is fixed throughout, we introduce a free energy functional that depends on an inverse temperature parameter $\beta>0$ and is asymptotically equivalent to the constrained Lagrangian \eqref{eqn: GP Lagrangian with self-overlap D} upon letting $\beta\to \infty$. 

For each inverse temperature parameter $\beta>0$ and every $\epsilon>0$, consider the free energy
\begin{equation}
\tilde{F}_{N,\epsilon}(\beta)=\frac{1}{\beta N} \log \int_{\Sigma_\epsilon(D)}\exp \beta H_{N,p,t}(\basigma)\ud \basigma
\end{equation}
and the quenched free energy
\begin{equation}\label{eqn: GP free energy with epsilon band}
F_{N,\epsilon}(\beta)=\frac{1}{\beta N} \E \log \int_{\Sigma_\epsilon(D)}\exp \beta H_{N,p,t}(\basigma)\ud \basigma.
\end{equation}
Recall the definition of the relaxed constrained Lagrangian in \eqref{eqn: GP relaxed constrained Lagrangian}. Since the $L^q$-norm converges to the $L^\infty$-norm, it is clear that
\begin{equation}\label{eqn: GP Lagrangian by free energy heuristic}
\lim_{\epsilon\to 0} \lim_{N\to\infty}\lim_{\beta\to \infty}F_{N,\epsilon}(\beta)=\lim_{\epsilon\to 0}\lim_{N\to\infty}\E L_{N,p,D,\epsilon}(t).
\end{equation}
We will now use the continuity result in \cref{GP continuity of constrained Lagrangian} to show that the right-hand side of this equation coincides with the limit of the constrained Lagrangian \eqref{eqn: GP Lagrangian with self-overlap D}. Subsequently, we will prove that \eqref{eqn: GP Lagrangian by free energy heuristic} still holds if the limit in $\beta$ is taken after the limits in $\epsilon$ and $N$. The benefit of exchanging these limits is that the main result in \cite{PanVec} gives a Parisi-type variational formula for the limit in $\epsilon$ and $N$ of the quenched free energy \eqref{eqn: GP free energy with epsilon band} for each fixed $\beta>0$. In \cref{GP sec9} and \cref{GP sec10} we will study this formula in the limit $\beta\to \infty$ to finally prove \cref{GP main result} in \cref{GP sec11}.

\begin{proposition}\label{GP relaxing the constrained Lagrangian}
If $2<p<\infty$ and $t>0$, then
\begin{equation}
\lim_{\epsilon \to 0}\lim_{N\to\infty}L_{N,p,D,\epsilon}(t)=\lim_{\epsilon \to 0}\lim_{N\to\infty}\E L_{N,p,D,\epsilon}(t)=L_{p,D}(t)
\end{equation}
almost surely.
\end{proposition}

\begin{proof}
By the Gaussian concentration of the maximum, it suffices to prove that
\begin{equation}\label{eqn: GP relaxing the constrained Lagrangian key}
\lim_{\epsilon \to 0}\lim_{N\to\infty}\E L_{N,p,D,\epsilon}(t)=L_{p,D}(t).
\end{equation}
Given $0<\epsilon<\kappa^{-2}$ smaller than the smallest non-zero eigenvalue of $D$, the equality $D_\epsilon=D$ and \cref{GP continuity of constrained Lagrangian} imply that
\begin{equation}\label{eqn: GP asymptotic behaviour of relaxed constrained Lagrangian upper}
\lim_{\epsilon\to 0}\limsup_{N\to\infty}\E L_{N,p,D,\epsilon}(t)\leq L_{p,D}(t).
\end{equation}
On the other hand, the Gaussian concentration of the maximum reveals that for every $\epsilon>0$,
$$L_{p,D}(t)=\lim_{N\to\infty} L_{N,p,D}(t)\leq \liminf_{N\to \infty} L_{N,p,D,\epsilon}(t)=\liminf_{N\to \infty}\E L_{N,p,D,\epsilon}(t).$$
Letting $\epsilon \to 0$ and remembering \eqref{eqn: GP asymptotic behaviour of relaxed constrained Lagrangian upper} establishes \eqref{eqn: GP relaxing the constrained Lagrangian key} and completes the proof.
\end{proof}

\begin{lemma}\label{GP Lagrangian by free energy lower}
If $2<p<\infty$ and $t>0$, then
\begin{equation}
\limsup_{\beta\to \infty}\lim_{\epsilon \to 0}\lim_{N\to \infty}F_{N,\epsilon}(\beta)\leq L_{p,D}(t).
\end{equation}
\end{lemma}

\begin{proof}
Fix $\delta\in (0,t)$, and for each $t>0$ let $\barho_{t}$ be a maximizer of the relaxed constrained Lagrangian \eqref{eqn: GP relaxed constrained Lagrangian}. By Fubini-Tonelli and a change of variables,
\begin{align}\label{eqn: GP Lagrangian by free energy lower key}
\tilde{F}_{N,\epsilon}(\beta)&\leq L_{N,p,D,\epsilon}(t-\delta)+\frac{1}{\beta}\log \int_{\R^\kappa} e^{-\beta \delta\norm{\asigma}^p_2}\ud \asigma\notag \\
&=\frac{1}{N}H_{N,p,t}(\barho_{t-\delta})+\delta \Norm{\barho_{t-\delta}}_{p,2}^p-\frac{\kappa \log \beta \delta}{p\beta}+\frac{1}{\beta}\log \int_{\R^\kappa} e^{-\norm{\asigma}^p_2}\ud \asigma\notag\\
&\leq L_{N,p,D,\epsilon}(t)+\delta \Norm{\barho_{t-\delta}}_{p,2}^p-\frac{\kappa \log \beta \delta}{p\beta}+\frac{1}{\beta}\log \int_{\R^\kappa} e^{-\norm{\asigma}^p_2}\ud \asigma.
\end{align}
To bound this further, let $A\in \R^{\kappa\times \kappa}$ be a symmetric and non-negative definite matrix with $AA^T=\kappa D$, and denote by $\asigma_i\in \R^\kappa$ the $i$'th column of $A$. Consider the subsequence $M=N\kappa$, and define the $\kappa$-periodic vector spin configuration $\basigma\in (\R^\kappa)^{M}$ by $\asigma_j=\asigma_i$ whenever $j\equiv i\ \mathrm{mod} \kappa$. From \eqref{eqn: GP self-overlap}, it is clear that $\basigma\in \Sigma(D)$. Indeed,
$$R(\basigma,\basigma)=\frac{1}{M}\sum_{i=1}^M \asigma_i\asigma_i^T=\frac{1}{\kappa}\sum_{i=1}^\kappa \asigma_i\asigma_i^T=\frac{1}{\kappa}AA^T=D.$$
If $G_M$ denotes the $M\times M$ random matrix in \eqref{GP GN matrix}, then the Cauchy-Schwarz inequality implies that for each $t>0$,
$$H_{M,p,t}(\barho_{t})\leq \sqrt{M}\norm{G_M}_2\Norm{\barho_t}_{2,2}^2-t\norm{\barho_t}_{p,2}^p,
$$
and similarly,
$$H_{M,p,t}(\barho_t)\geq H_{M,p,t}(\basigma)\geq  -\sqrt{M}\norm{G_M}_2\Norm{\basigma}_{2,2}^2-t\norm{\basigma}_{p,2}^p.$$
Rearranging and remembering \eqref{eqn: GP 2-2 norm as trace of self overlap} gives
\begin{align}\label{eqn: GP lp norm of maximizer}
\Norm{\barho_t}_{p,2}^p&\leq \frac{2\norm{G_M}_2(\tr(D)+\epsilon \kappa)}{t\sqrt{M}}+\max_{1\leq i\leq \kappa}\norm{\asigma_i}_2^p\\
&=\frac{2(\tr(D)+\epsilon \kappa)\lambda_M^M}{\sqrt{2M}t}+\max_{1\leq i\leq \kappa}\norm{\asigma_i}_2^p,\notag
\end{align}
where $\lambda_M^M$ denotes the largest eigenvalue of the Gaussian orthogonal ensemble $\bar G_M$ in \eqref{eqn: GP barGN matrix}. Substituting this into \eqref{eqn: GP Lagrangian by free energy lower key}, appealing to the Gaussian concentration of the free energy and leveraging the asymptotics of $\lambda_M^M$ established in theorem 2.1.22 of \cite{AGZ} shows that
\begin{align*}
\lim_{N\to\infty}F_{N,\epsilon}(\beta)&\leq \lim_{N\to\infty}L_{N,p,D,\epsilon}(t)+\delta\Big(\frac{2\sqrt{2}(\tr(D)+\epsilon \kappa)}{t-\delta}+\max_{1\leq i\leq \kappa}\norm{\asigma_i}_2^p\Big)\\
&\quad-\frac{\kappa \log \beta \delta}{p\beta}+\frac{1}{\beta}\log \int_{\R^\kappa} e^{-\norm{\asigma}^p_2}\ud \asigma.
\end{align*}
We have implicitly used the fact that the limit of $F_{N,\epsilon}(\beta)$ exists and therefore coincides with that of $F_{M,\epsilon}(\beta)$. This can be shown using a Guerra-Toninelli argument as in \cref{GP Guerra Toninelli}, or by appealing to the results in \cite{PanVec} as we will do in \cref{GP sec8}. Letting $\epsilon\to 0$, then $\beta \to \infty$ and finally $\delta\to 0$ completes the proof upon invoking \cref{GP relaxing the constrained Lagrangian}.
\end{proof}

\begin{theorem}\label{GP Lagrangian by free energy}
If $2<p<\infty$ and $t>0$, then
\begin{equation}\label{eqn: GP Lagrangian by free energy}
\limsup_{\beta\to \infty}\lim_{\epsilon \to 0}\lim_{N\to \infty}F_{N,\epsilon}(\beta)\leq L_{p,D}(t)\leq \liminf_{\beta \to \infty}\lim_{N\to \infty}F_{N,\beta^{-1}}(\beta).
\end{equation}
\end{theorem}

\begin{proof}
By \cref{GP Lagrangian by free energy lower}, it suffices to prove the upper bound in \eqref{eqn: GP Lagrangian by free energy}. Fix $\epsilon \in (0,1)$, and let $\delta=\epsilon/K$ for a large enough $K>0$ to be determined. Consider the subsequence $M=N\kappa$ as in the proof of \cref{GP Lagrangian by free energy lower}, and let $\barho\in \Sigma_\delta(D)$ be a maximizer of the relaxed constrained Lagrangian $L_{M,p,D,\delta}(t)$ in \eqref{eqn: GP relaxed constrained Lagrangian}. Introduce the $\delta/\sqrt{\kappa}$-neighbourhood,
$$\CC_{\delta/\sqrt{\kappa}}(\barho)=\barho+[-\delta/\sqrt{\kappa},\delta/\sqrt{\kappa}]^{M\kappa}\subset \big\{\basigma\in (\R^\kappa)^M\mid \Norm{\basigma-\barho}_{2,2}\leq \delta\big\},$$
and observe that $\CC_{\delta/\sqrt{\kappa}}(\barho)\subset \Sigma_\epsilon(D)$. Indeed, the same argument used to obtain \eqref{eqn: GP modified coordinates small distortion} implies that for any $\basigma\in \CC_{\delta/\sqrt{\kappa}}(\barho)$,
\begin{align*}
\norm{R(\basigma,\basigma)-R(\barho,\barho)}_\infty&\leq \Norm{\basigma-\barho}_{2,2}\big(\Norm{\basigma}_{2,2}+\Norm{\barho}_{2,2}\big)\leq \delta\big(1+2\sqrt{\tr(D)+\kappa}\big)\\
&<\frac{\epsilon}{2}
\end{align*}
provided that $K=K(D,\kappa)$ is large enough. The second inequality uses \eqref{eqn: GP 2-2 norm as trace of self overlap}. This means that
\begin{align}\label{eqn: GP Lagrangian by free energy upper key}
\tilde{F}_{M,\epsilon}(\beta)&\geq \frac{1}{\beta M}\int_{\CC_{\delta/\sqrt{\kappa}}(\barho)}\exp \beta H_{M,p,t}(\basigma)\ud \basigma \notag\\
&\geq L_{M,p,D,\delta}(t)+\frac{1}{M}\inf_{\basigma\in \CC_{\delta/\sqrt{\kappa}}(\barho)}\big(H_{M,p,t}(\basigma)-H_{M,p,t}(\barho)\big)+\frac{\kappa}{\beta}\log \frac{2\delta}{\sqrt{\kappa}}.
\end{align}
To bound this further, fix $\basigma\in \CC_{\delta/\sqrt{\kappa}}(\barho)$ and recall the definition of the $M\times M$ random matrix $G_M$ in \eqref{GP GN matrix}. The Cauchy-Schwarz inequality implies that
\begin{align}
H_N(\basigma)-H_N(\barho)&=\frac{1}{2\sqrt{M}}\sum_{k=1}^\kappa \big((G_M+G_M^T)(\bsigma(k)-\brho(k)),\bsigma(k)+\brho(k)\big)\notag\\
&\geq -\frac{\norm{G_M}_2}{\sqrt{M}}\sum_{k=1}^\kappa\norm{\bsigma(k)-\brho(k)}_2\big(\norm{\bsigma(k)}_2+\norm{\brho(k)}_2\big)\notag\\
&\geq -\sqrt{M}\norm{G_M}_2\Norm{\basigma-\barho}_{2,2}(\Norm{\basigma}_{2,2}+\Norm{\barho}_{2,2})\notag\\
&\geq -M\delta\frac{\norm{G_M}_2}{\sqrt{M}}(1+2\sqrt{\tr(D)+\kappa}).\label{eqn: GP Lagrangian upper bound by free energy one}
\end{align}
On the other hand, an identical argument as that used to obtain \eqref{eqn: GP difference of p norms 1} shows that
$$\norm{\basigma}_{p,2}^p-\norm{\barho}_{p,2}^p\leq Mp\Norm{\basigma-\barho}_{p,2}\Norm{\basigma}_{p,2}^{p-1}\leq Mp\delta \big(1+\norm{\barho}_{p,2}\big)^{p-1}.$$
Together with \eqref{eqn: GP Lagrangian upper bound by free energy one}, \eqref{eqn: GP lp norm of maximizer} and \cref{GP GN bounded with high probability}, this gives constants $C,K'>0$ that depend only on $\kappa$, $D$, $p$ and $t$ such that with probability at least $1-Ce^{-M/C}$,
$$H_{M,p,t}(\basigma)-H_{M,p,t}(\barho)\geq -M\delta K'.$$
Substituting this lower bound into \eqref{eqn: GP Lagrangian by free energy upper key} and combining the Gaussian concentration of the free energy with the Borel-Cantelli lemma to let $N\to \infty$ yields
$$\lim_{N\to \infty}F_{N,\epsilon}(\beta)\geq \lim_{N\to \infty} L_{N,p,D,\epsilon/K}(t)-\epsilon K^{-1} K'+\frac{\kappa}{\beta}\log \frac{2\epsilon}{K\sqrt{\kappa}}.$$
Taking $\epsilon=\beta^{-1}$ and letting $\beta\to \infty$ completes the proof upon invoking \cref{GP relaxing the constrained Lagrangian}.
\end{proof}

\section{The limit of the free energy}\label{GP sec8}

In this section we describe the implications of the main result in \cite{PanVec} on the asymptotic representation of the constrained Lagrangian \eqref{eqn: GP Lagrangian with self-overlap D} established in \cref{GP Lagrangian by free energy}. Given a constraint $D\in \Gamma_\kappa$, some $t>0$ and an inverse temperature parameter $\beta>0$, all of which will remain fixed throughout this section, consider the measure on $\R^\kappa$ defined by
\begin{equation}\label{eqn: GP limit of free energy mu}
\ud\mu(\asigma)=\exp\big(-t\beta \norm{\asigma}_2^p\big)\ud \asigma.
\end{equation}
Notice that the quenched free energy \eqref{eqn: GP free energy with epsilon band} may be written as
\begin{equation}\label{eqn: GP free energy in terms of mu}
F_{N,\epsilon}(\beta)=\frac{1}{\beta N}\E \log \int_{\Sigma_\epsilon(D)}\exp \beta H_N(\basigma)\ud \mu^{\otimes N}(\basigma).
\end{equation}
If it were not for the fact that the measure $\mu$ in \eqref{eqn: GP limit of free energy mu} is not compactly supported, this free energy functional would fall into the class of free energy functionals studied in \cite{PanVec}. Fortunately, the compact support assumption in \cite{PanVec} is not necessary. Instead, it is a convenient assumption that ensures all objects introduced are well-defined and spin configurations in the set $\Sigma_\epsilon(D)$ remain bounded. Replicating the arguments in \cite{PanVec}, it is not hard to use that the measure \eqref{eqn: GP limit of free energy mu} exhibits super-Gaussian decay in the range $2<p<\infty$ to show that the analogue of the Parisi formula with vector spins proved in \cite{PanVec} for compactly supported measures also holds for the free energy functional \eqref{eqn: GP free energy in terms of mu}. We will not give any more details than this, and simply content ourselves with stating the asymptotic formula for \eqref{eqn: GP free energy in terms of mu} which we will use between \cref{GP sec9} and \cref{GP sec11} to prove \cref{GP main result}.

Denote by $\M^d$ the set of probability distributions on $[0,1]$ with finitely many atoms. Notice that any discrete measure $\alpha\in \M^d$ may be identified with two sequences of parameters
\begin{align}
0&=q_{-1}\leq q_0\leq \ldots\leq q_{r-1}\leq q_r=1, \label{eqn: GP positive temperature qs}\\
0&=\alpha_{-1}\leq \alpha_0\leq\ldots\leq \alpha_{r-1}\leq \alpha_r=1, \label{eqn: GP positive temperature alphas}
\end{align}
satisfying $\alpha(\{q_j\})=\alpha_j-\alpha_{j-1}$ for $0\leq j\leq r$. For each Lagrange multiplier $\lambda\in \R^{\kappa(\kappa+1)/2}$, define the function $f_\lambda^\beta:\R^\kappa \to \R$ by
\begin{equation}\label{eqn: GP positive temperature terminal condition}
f_\lambda^\beta(\x)=\frac{1}{\beta}\log\int_{\R^\kappa} \exp\beta\Big(\big(\asigma,\x\big)+\sum_{k\leq k'}\lambda_{k,k'}\sigma(k)\sigma(k')-t\norm{\asigma}_2^p\Big)\ud \asigma.
\end{equation}
Given a path $\pi \in \Pi_D$ defined by the sequences \eqref{eqn: GP zero temperature qs} and \eqref{eqn: GP zero temperature gammas}, for each $0\leq j\leq r$ consider an independent Gaussian vector $z_j=(z_j(k))_{k\leq \kappa}$ with covariance structure \eqref{eqn: GP covariance of z}. Define the sequence $\smash{(X_l^{\lambda,\alpha,\pi,\beta})_{0\leq l\leq r}}$ recursively as follows. Let
\begin{equation}
X_r^{\lambda,\alpha,\pi,\beta}((z_j)_{0\leq j\leq r})=f_\lambda^\beta\Big(\sqrt{2}\sum_{j=1}^rz_j\Big),
\end{equation}
and for $0\leq l\leq r-1$ let
\begin{equation}\label{eqn: GP Parisi sequence at positive temperature}
X_l^{\lambda,\alpha,\pi,\beta}((z_j)_{0\leq j\leq l})=\frac{1}{\beta \alpha_l}\log \E_{z_{l+1}}\exp \beta \alpha_l X_{l+1}^{\lambda,\alpha,\pi,\beta}((z_j)_{0\leq j\leq l+1}).
\end{equation}
This inductive procedure is well-defined by the growth bounds in \cref{GP positive temperature terminal growth bound}. Introduce the Parisi functional,
\begin{equation}\label{eqn: GP Parisi functional at positive temperature preliminary}
\Par_\beta(\lambda,\alpha,\pi)=X_0^{\lambda,\alpha,\pi,\beta}-\sum_{k\leq k'}\lambda_{k,k'}D_{k,k'}-\frac{\beta}{2}\sum_{0\leq j\leq r-1}\alpha_j\big(\norm{\gamma_{j+1}}_{\text{HS}}^2-\norm{\gamma_{j}}_{\text{HS}}^2\big).
\end{equation}
Observe that
\begin{align}\label{eqn: GP Parisi functional integral term}
\sum_{0\leq j\leq r-1}\alpha_j\big(\norm{\gamma_{j+1}}_{\text{HS}}^2-\norm{\gamma_j}_{\text{HS}}^2\big)&=\sum_{0\leq j\leq r-1}\int_{q_j}^{q_{j+1}}\alpha(s)\frac{\ud}{\ud s}\norm{\pi(s)}_{\text{HS}}^2\ud s\notag\\
&=2\int_0^1\alpha(s)\Sum\big(\pi(s)\odot \pi'(s)\big)\ud s,
\end{align}
where we have abused notation by writing $\alpha$ both for the measure and its cumulative distribution function. The Parisi functional may therefore be expressed as
\begin{equation}\label{eqn: GP Parisi functional at positive temperature}
\Par_\beta(\lambda,\alpha,\pi)=X_0^{\lambda,\alpha,\pi,\beta}-\sum_{k\leq k'}\lambda_{k,k'}D_{k,k'}-\int_0^1 \beta \alpha(s)\Sum\big(\pi(s)\odot \pi'(s)\big)\ud s.
\end{equation}
We have made all dependencies on $D, p$ and $t$ implicit for clarity of notation, but we will make them explicit whenever necessary. The proof of  \cref{GP main result} will leverage the following consequence of the main result in \cite{PanVec}.

\begin{theorem}\label{GP limit of free energy}
If $2<p<\infty$, then
\begin{equation}\label{eqn: GP limit of free energy}
\lim_{\epsilon\to 0}\lim_{N\to\infty}F_{N,\epsilon}(\beta)=\inf_{\lambda,\alpha,\pi}\Par_{\beta}(\lambda,\alpha,\pi),
\end{equation}
where the infimum is taken over all $(\lambda,\alpha,\pi)\in \R^{\kappa(\kappa+1)/2}\times \M^d\times \Pi_D$.
\end{theorem}

This result can be viewed as a positive temperature analogue of \cref{GP main result}. Together with \cref{GP Lagrangian by free energy}, it essentially reduces the proof of \cref{GP main result} to showing that 
\begin{equation}\label{eqn: GP goal recast}
\lim_{\beta\to\infty}\inf_{\lambda,\alpha,\pi}\Par_{\beta}(\lambda,\alpha,\pi)=\inf_{\lambda,\zeta,\pi}\Par_\infty(\lambda,\zeta,\pi).
\end{equation}
Notice the similarity between the Parisi functionals \eqref{eqn: GP Parisi functional at positive temperature} and \eqref{eqn: GP Parisi functional at zero temperature}. If it were not for the terms $X_0$ and $Y_0$ in \eqref{eqn: GP Parisi sequence at positive temperature} and \eqref{eqn: GP Parisi sequence at zero temperature}, there would be a natural correspondence between these two functionals obtained by setting $\zeta=\beta \alpha$. The proof of \eqref{eqn: GP goal recast} will therefore consist in showing that, when evaluated at almost minimizers, the terms $X_0$ and $Y_0$ in \eqref{eqn: GP Parisi sequence at positive temperature} and \eqref{eqn: GP Parisi sequence at zero temperature} differ by a quantity that vanishes as $\beta\to \infty$. To control this difference, we will compare the terminal conditions \eqref{eqn: GP positive temperature terminal condition} and \eqref{eqn: GP zero-temperature terminal condition} in \cref{GP sec9}. We will then use the Auffinger-Chen representation \cite{AuffingerChenREP, Jagannath} in \cref{GP sec10} to translate the bounds on the terminal conditions into control on $X_0$ and $Y_0$. This analysis will be exploited in \cref{GP sec11} to establish \eqref{eqn: GP goal recast} and therefore prove \cref{GP main result}. This strategy is considerably different to that in \cite{WeiKuo}, where the free energy functional \eqref{eqn: GP free energy in terms of mu} is truncated at some level $M>0$. This truncation simplifies much of the analysis for fixed $M>0$, but requires a lot of care when sending $M\to \infty$. By not truncating the free energy, we simplify and shorten the proof of  \cref{GP main result} even in the scalar case, $\kappa=1$, studied in \cite{WeiKuo}.

\section{Comparison of the terminal conditions}\label{GP sec9}

In this section we prove quantitative bounds on the difference between the terminal conditions \eqref{eqn: GP positive temperature terminal condition} and \eqref{eqn: GP zero-temperature terminal condition}. Although the analysis in this section uses only elementary concepts, it is the key to proving \cref{GP main result}; the rest of the paper will use tools from the literature to transform the bounds established in this section into a proof of \cref{GP main result}. To alleviate notation, the inverse temperature parameter $\beta>0$, the Lagrange multiplier $\lambda\in \R^{\kappa(\kappa+1)/2}$ and the parameters $t>0$ and $2<p<\infty$ will be fixed throughout this section. We will write $C>0$ for a constant that depends only on $\kappa$, $p$ and $t$ whose value might not be the same at each occurrence.

We begin by bounding $\smash{f^\beta_\lambda}$ from above by $\smash{f_\lambda^\infty}$ up to a small error. It will be necessary to make the dependence of these terminal conditions on $t>0$ explicit by writing $\smash{f_{\lambda,t}^\beta}$ and $\smash{f_{\lambda,t}^\infty}$.

\begin{proposition}\label{GP positive temperature terminal condition bounded by zero temperature}
If $2<p<\infty$, $\x\in \R^\kappa$ and $\delta\in (0,t)$, then
\begin{equation}\label{eqn: GP positive temperature terminal condition bounded by zero temperature}
f_{\lambda,t}^\beta(\x)\leq f^\infty_{\lambda,t-\delta}(\x)-\frac{\kappa \log \beta \delta}{p\beta}+\frac{1}{\beta}\log \int_{\R^\kappa} e^{-\norm{\asigma}_2^p}\ud \asigma.
\end{equation}
\end{proposition}

\begin{proof}
By a change of variables,
\begin{align*}
f_{\lambda,t}^\beta(\x)&\leq f^{\infty}_{\lambda,t-\delta}(\x)+\frac{1}{\beta}\log\int_{\R^\kappa} e^{-\beta \delta\norm{\asigma}^p_2}\ud \asigma\\
&=f^\infty_{\lambda,t-\delta}(\x)-\frac{\kappa \log \beta \delta}{p\beta}+\frac{1}{\beta}\log \int_{\R^\kappa} e^{-\norm{\asigma}_2^p}\ud \asigma.
\end{align*}
This finishes the proof.
\end{proof}

This result will play its part when we prove the upper bound in \eqref{eqn: GP goal recast}, at which point we will have to replace $\smash{f_{\lambda,t-\delta}^\infty}$ in \eqref{eqn: GP positive temperature terminal condition bounded by zero temperature} by $\smash{f_{\lambda,t}^\infty}$. This will be achieved through a continuity result that is an immediate consequence of the following bound on any maximizer $\smash{\asigma_{\x,\lambda,t}^*}$ of \eqref{eqn: GP zero-temperature terminal condition}. 

\begin{lemma}\label{GP bound on argmax}
If $2<p<\infty$ and $\x\in \R^\kappa$, then there exists $\asigma^*_{\x,\lambda,t}\in \R^\kappa$ which attains the maximum in \eqref{eqn: GP zero-temperature terminal condition}. Moreover,
\begin{equation}\label{eqn: GP bound on argmax}
\norm{\asigma^*_{\x,\lambda,t}}_2\leq \max\Big(\Big(\frac{2\norm{\lambda}_\infty}{t}\Big)^{\frac{1}{p-2}},\Big(\frac{2\norm{\x}_2}{t}\Big)^{\frac{1}{p-1}}\Big).
\end{equation}
\end{lemma}

\begin{proof}
Consider the function $g:\R^\kappa\to \R$ defined by
\begin{equation}\label{eqn: GP terminal function}
g(\asigma)=\big(\asigma,\x\big)+\sum_{k\leq k'}\lambda_{k,k'}\sigma(k)\sigma(k')-t\norm{\asigma}_2^p.
\end{equation}
By the Cauchy-Schwarz inequality,
$$g(\asigma)\leq \norm{\asigma}_2\norm{\x}_2+\norm{\lambda}_\infty\norm{\asigma}_2^2-t\norm{\asigma}_2^p.$$
Since $p>2$, it follows that $\smash{\lim_{\norm{\asigma}_2\to \infty}g(\asigma)=-\infty}$. Remembering that a continuous function attains its maximum on each compact set, there must exist $\smash{\asigma^*_{\x,\lambda,t}\in \R^\kappa}$ which attains the maximum in \eqref{eqn: GP zero-temperature terminal condition}. If we had
\begin{equation}\label{eqn: GP bound on argmax key}
t\norm{\asigma^*_{\x,\lambda,t}}_2^p>\max\big(2\norm{\asigma^*_{\x,\lambda,t}}_2\norm{\x}_2,2\norm{\asigma^*_{\x,\lambda,t}}_2^2\norm{\lambda}_\infty\big),  
\end{equation}
then we would have
$$0=g(0)\leq f_\lambda^\infty(\x)=g(\asigma_{\x,\lambda,t}^*)< \frac{t\norm{\asigma_{x,\lambda,t}^*}_2^p }{2}+\frac{t\norm{\asigma_{x,\lambda,t}^*}_2^p }{2}-t\norm{\asigma_{x,\lambda,t}^*}_2^p=0$$
which is not possible. Rearranging the reverse of \eqref{eqn: GP bound on argmax key} gives \eqref{eqn: GP bound on argmax} and completes the proof.
\end{proof}

\begin{proposition}\label{GP continuity in t of zero-temperature terminal condition}
If $2<p<\infty$ and $\x\in \R^\kappa$, then
\begin{equation}
\lim_{\delta\to 0}f_{\lambda,t-\delta}^\infty(\x)=f_{\lambda,t}^\infty(\x).
\end{equation}
\end{proposition}

\begin{proof}
Fix $\delta\in (0,t/2)$. It is clear that $f_{\lambda,t}^\infty(\x)\leq f_{\lambda,t-\delta}^\infty(\x)$. On the other hand, \cref{GP bound on argmax} implies that
\begin{align*}
f_{\lambda,t-\delta}^\infty(\x)&=\big(\asigma^*_{\x,\lambda,t-\delta},\x\big)+\sum_{k\leq k'}\lambda_{k,k'}\asigma^*_{\x,\lambda,t-\delta}(k)\asigma^*_{\x,\lambda,t-\delta}(k')-(t-\delta)\norm{\asigma^*_{\x,\lambda,t-\delta}}_2^p\\
&\leq f_{\lambda,t}^\infty(x)+\delta\max\Big(\Big(\frac{4\norm{\lambda}_\infty}{t}\Big)^{\frac{1}{p-2}},\Big(\frac{4\norm{\x}_2}{t}\Big)^{\frac{1}{p-1}}\Big)^p.
\end{align*}
Letting $\delta\to 0$ completes the proof.
\end{proof}

We now turn our attention to bounding $\smash{f_\lambda^\infty}$ from above by $\smash{f_\lambda^\beta}$ up to a small error. Once again, we drop the dependence of these terminal conditions on $t>0$. Through a simple calculation detailed in the proof of \cref{GP zero temperature terminal condition bounded by positive temperature}, this essentially comes down to bounding the average of the function \eqref{eqn: GP terminal function} on a cube by its value at the centre of the cube. In other words, we need a type of mean-value property for the function \eqref{eqn: GP terminal function}. There are two main issues to address: the function \eqref{eqn: GP terminal function} is not necessarily convex and, for technical reasons, we would like this mean-value property on cubes instead of balls. We will deal with the lack of convexity by adding a convex perturbation to \eqref{eqn: GP terminal function}. Replacing balls by cubes will be done by applying Jensen's inequality to a function defined on a cube of side-length $\delta>0$ centred at some $\arho\in \R^\kappa$,
\begin{equation}
\CC_\delta(\arho)=\arho+[-\delta,\delta]^{\kappa}.
\end{equation}
When we prove \eqref{eqn: GP goal recast}, the error incurred by these two fixes will vanish upon letting $\beta \to \infty$. The mean-value property for cubes takes the following form, and uses \cref{Gershgorin corollary} to establish convexity of the function to which Jensen's inequality is applied.

\begin{lemma}\label{GP bounding supremum by average}
If $2<p<\infty$, $\x\in \R^\kappa$ and $\delta>0$, then
\begin{align}
f_\lambda^\infty(\x)&\leq \frac{1}{\abs{\CC_\delta(\asigma^*_{\x,\lambda,t})}}\int_{\CC_\delta(\asigma^*_{\x,\lambda, t})}\Big(\big(\asigma,\x\big)+\sum_{k\leq k'}\lambda_{k,k'}\sigma(k)\sigma(k')-t\norm{\asigma}_2^p \Big)\ud \asigma\\
&\quad+C\delta^2\big(\norm{\asigma^*_{\x,\lambda,t}}_2^{p-2}+\delta^{p-2}+\norm{\lambda}_\infty\big)\notag
\end{align}
\end{lemma}

\begin{proof}
To simplify notation, for $k>k'$ let $\lambda_{k,k'}=\lambda_{k',k}$. Recall the function $g:\R^\kappa\to \R$ in \eqref{eqn: GP terminal function}. Given $\delta>0$ and $\arho\in \R^\kappa$, consider the function $f:\CC_{2\delta}(\arho)\to \R$ defined by
$$f(\asigma)=g(\asigma)+\sum_{k=1}^\kappa \frac{1}{2}\sigma(k)^2h(\arho),$$
where the constant $h(\arho)$ depends on $\arho$ and is given by
$$h(\arho)=\big(tp(p-1)+tp(p-2)\kappa\big)\big((2\norm{\arho}_2)^{p-2}+(2\sqrt{\kappa}\delta)^{p-2}\big)+2\abs{\lambda_{k,k}}+\sum_{k'\neq k}\abs{\lambda_{k,k'}}.$$
Fix $\asigma\in \CC_{2\delta}(\arho)$ and $1\leq k\leq \kappa$. A direct computation shows that
\begin{align*}
\partial_{\sigma(k)\sigma(k)}f(\asigma)&=2\lambda_{k,k}-tp\norm{\asigma}_2^{p-2}-tp(p-2)\norm{\asigma}_2^{p-4}\sigma(k)^2+h(\arho)\\
&\geq 2(\lambda_{k,k}+\abs{\lambda_{k,k}})+tp(p-1)\big((2\norm{\arho}_2)^{p-2}+(2\sqrt{\kappa}\delta)^{p-2}-\norm{\asigma}_2^{p-2}\big)\\
&\quad+tp(p-2)\kappa\big((2\norm{\arho}_2)^{p-2}+(2\sqrt{\kappa}\delta)^{p-2}\big)+\sum_{k'\neq k}\abs{\lambda_{k,k'}}\\
&\geq tp(p-2)\kappa\big((2\norm{\arho}_2)^{p-2}+(2\sqrt{\kappa}\delta)^{p-2}\big)+\sum_{k'\neq k}\abs{\lambda_{k,k'}}.
\end{align*}
Similarly, for $1\leq k\neq k'\leq \kappa$,
$$\partial_{\sigma(k)\sigma(k')}f(\asigma)=\lambda_{k,k'}-tp(p-2)\norm{\asigma}_2^{p-4}\sigma(k)\sigma(k').$$
It follows that
$$\sum_{k'\neq k}\abs{\partial_{\sigma(k)\sigma(k')}f(\asigma)}\leq \sum_{k'\neq k}\abs{\lambda_{k,k'}}+tp(p-2)\kappa \norm{\asigma}_2^{p-2}\leq \partial_{\sigma(k)\sigma(k)}f(\asigma).$$
Invoking \cref{Gershgorin corollary} shows that the Hessian of $f$ is non-negative definite, and therefore $f$ is convex. With this in mind, let $X=(X_i)_{i\leq \kappa}$ be a vector of independent random variables with $X_i$ uniformly distributed on the interval $[\rho(i)-\delta,\rho(i)+\delta]$. Jensen's inequality implies that
$$f(\arho)=f(\E X)\leq \E f(X)= \frac{1}{\abs{\CC_\delta(\arho)}}\int_{\CC_\delta(\arho)}f(\asigma)\ud \asigma$$
Substituting the definition of $f$ into the right-hand side of this inequality and integrating yields
$$f(\arho)\leq \int_{\CC_\delta(\arho)}g(\asigma)\ud \asigma+\sum_{k=1}^\kappa \frac{1}{2}\rho(k)^2h(\arho)+\frac{\kappa \delta^2}{6}h(\arho).$$
Rearranging and taking $\arho=\asigma^*_{\x,\lambda,t}$ completes the proof.
\end{proof}

\begin{proposition}\label{GP zero temperature terminal condition bounded by positive temperature}
If $2<p<\infty$ and $L>0$, then for any $\x\in \R^\kappa$ with $\norm{\x}_2\leq L$ and every $0<\delta<L^{\frac{1}{p-1}}$,
\begin{equation}\label{eqn: GP zero temperature terminal condition bounded by positive temperature}
f_\lambda^\infty(\x)\leq f_\lambda^\beta(\x)+C\delta^2\Big(\norm{\lambda}_\infty+L^{\frac{p-2}{p-1}}\Big)-\frac{\kappa\log 2\delta}{\beta}.
\end{equation}
\end{proposition}

\begin{proof}
Given $\arho\in \R^\kappa$, Jensen's inequality implies that
\begin{align}\label{eqn: GP Jensen bound for positive temperature terminal condition}
f_\lambda^\beta(\x)&\geq \frac{1}{\beta}\log \int_{\CC_\delta(\arho)}\exp\beta\Big(\big(\asigma,\x\big)+\sum_{k\leq k'}\lambda_{k,k'}\sigma(k)\sigma(k')-t\norm{\asigma}_2^p\Big)\ud \asigma \notag\\
&\geq \frac{1}{\abs{\CC_\delta(\arho)}}\int_{\CC_\delta(\arho)}\Big(\big(\asigma,\x\big)+\sum_{k\leq k'}\lambda_{k,k'}\sigma(k)\sigma(k')-t\norm{\asigma}_2^p\Big)\ud \asigma+\frac{\kappa \log 2\delta}{\beta}.
\end{align}
Applying this with $\arho=\asigma^*_{\x,\lambda,t}$ and invoking \cref{GP bounding supremum by average} yields
$$f_\lambda^\beta(\x)\geq f_\lambda^\infty(\x)-C\delta^2\big(\norm{\asigma^*_{\x,\lambda,t}}_2^{p-2}+\delta^{p-2}+\norm{\lambda}_\infty\big)+\frac{\kappa \log 2\delta}{\beta}.$$
The result now follows by \cref{GP bound on argmax}.
\end{proof}

This result will play its part when we prove the lower bound in \eqref{eqn: GP goal recast}, at which point we will have to carefully deal with the fact that it only gives a bound of $\smash{f_\lambda^\infty}$ by $\smash{f_\lambda^\beta}$ for values of $\x$ in a (possibly large) neighbourhood of the origin. Fortunately, this will not be a problem. It turns out that the bound \eqref{eqn: GP zero temperature terminal condition bounded by positive temperature} will be applied to one of the Auffinger-Chen control processes introduced in the next section. The generalization of the moment bound in lemma 12.3 of \cite{WeiKuo} to the vector spin setting, which corresponds to \cref{GP moment bound on zero-temperature process} in this paper, will be used to show that dominating $\smash{f_\lambda^\infty}$ by $\smash{f_\lambda^\beta}$ around the origin is sufficient for our purposes.

\section{The Auffinger-Chen representation}\label{GP sec10}

In this section we extend the Auffinger-Chen stochastic control representation established for $\kappa=2$ and Lipschitz terminal conditions in \cite{WeiKuo2DPar} to the setting of arbitrary integer $\kappa\geq 1$ and terminal conditions with sub-quadratic growth such as \eqref{eqn: GP positive temperature terminal condition} and \eqref{eqn: GP zero-temperature terminal condition}. The results in this section will be combined with the bounds obtained in \cref{GP sec9} to compare the quantities $X_0$ and $Y_0$ in \eqref{eqn: GP Parisi sequence at positive temperature} and \eqref{eqn: GP Parisi sequence at zero temperature}. This will lead to a proof  of \cref{GP main result} in \cref{GP sec11}. 

Throughout this section, a constraint $D\in \Gamma_\kappa$, an inverse temperature parameter $\beta>0$, a Lagrange multiplier $\lambda\in \R^{\kappa(\kappa+1)/2}$, a  $\kappa$-dimensional Brownian motion $\bW=(W_1,\ldots,W_\kappa)$ and
parameters $t>0$ and $2<p<\infty$ will be fixed. We will also give ourselves a piecewise linear path $\pi \in \Pi_D$ defined by the sequences \eqref{eqn: GP zero temperature qs} and \eqref{eqn: GP zero temperature gammas}, as well as a discrete probability distribution $\alpha \in \M^d$ defined by the sequences \eqref{eqn: GP zero temperature qs} and \eqref{eqn: GP positive temperature alphas}. To prove the Auffinger-Chen representation, it will be convenient to replace the Gaussian random vectors $z_j$ with covariance structure \eqref{eqn: GP covariance of z} appearing in the definition of the Parisi functional \eqref{eqn: GP Parisi functional at positive temperature} by a continuous time stochastic process $\bB=(\bB(s))_{s\geq 0}$ that plays the same role,
\begin{equation}\label{eqn: GP definition of B process}
\bB(s)=\sqrt{2}\int_0^s \pi'(r)^{\frac{1}{2}}\ud \bW(r).
\end{equation}
Since $\pi'(r)=(q_j-q_{j-1})^{-1}(\gamma_j-\gamma_{j-1})\in \Gamma_\kappa$ for $r\in (q_{j-1},q_j)$, this process is well-defined. Moreover, the Ito isometry shows that
\begin{equation}\label{eqn: GP covariance of B process}
\Cov\big(\bB(q_j)-\bB(q_{j-1})\big)=2\int_{q_{j-1}}^{q_j}\pi'(r)\ud r=2(\gamma_j-\gamma_{j-1}).
\end{equation}
If we introduce the function $\Phi:[0,1]\times \R^\kappa\to \R$ defined recursively by
\begin{equation}\label{eqn: GP positive temperature Parisi function}
\left\{\begin{aligned}
\Phi(1,\x)&=f_\lambda^\beta(\x),\\
\Phi(s,\x)&=\frac{1}{\beta \alpha_j}\log \E \exp \beta \alpha_j\Phi(q_{j+1},\x+ \bB(q_{j+1})- \bB(s)), \, s\in [q_j,q_{j+1}),
\end{aligned}\right.    
\end{equation}
then the independence of the increments of $\bB$ and \eqref{eqn: GP covariance of B process} imply that the Parisi functional \eqref{eqn: GP Parisi functional at positive temperature} may be written as
\begin{equation}
\Par_\beta(\lambda,\alpha,\pi)=\Phi(0,0)-\sum_{k\leq k'}\lambda_{k,k'}D_{k,k'}-\int_0^1 \beta \alpha(s)\Sum\big(\pi(s)\odot \pi'(s)\big)\ud s.
\end{equation}
We have made all dependencies of $\Phi$ on $D$, $\beta$, $\lambda$, $\alpha$, $\pi$, $p$ and $t$ implicit for clarity of notation, but we will make them explicit whenever necessary. To obtain the Auffinger-Chen representation, we first use Gaussian integration by parts (see for instance lemma 1.1 in \cite{PanSKB}) to show that $\Phi$ satisfies a non-linear parabolic PDE.

\begin{lemma}\label{GP positive temperature Parisi PDE}
If $2<p<\infty$ and $(s,\x)\in [0,1]\times \R^\kappa$, then 
\begin{equation}\label{eqn: GP positive temperature Parisi PDE}
\partial_s\Phi(s,\x)=-\Big(\big(\pi'(s),\nabla^2\Phi(s,\x)\big)+\beta \alpha(s)\big(\pi'(s)\nabla \Phi(s,\x),\nabla \Phi(s,\x)\big)\Big),
\end{equation}
where $\partial_s\Phi$ is understood as the right-derivative at the points of discontinuity of $\alpha$.
\end{lemma}

\begin{proof}
Introduce the process $\bY(s)=\x+\bB(q_{j+1})-\bB(s)$, and fix $s\in [q_j,q_{j+1})$. A direct computation shows that
$$\Phi_{x_l}(s,\x)=\E \Phi_{x_l}(q_{j+1},\bY(s))Z(s)$$
for the process $Z(s)=\exp \beta \alpha_j(\Phi(q_{j+1},\bY(s))-\Phi(s,\x))$. Differentiating again yields
\begin{align*}
\Phi_{x_lx_{l'}}(s,\x)&=\E\big( \Phi_{x_lx_{l'}}(q_{j+1}, \bY(s))+\beta \alpha_j\Phi_{x_l}(q_{j+1},\bY(s))\Phi_{x_{l'}}(q_{j+1},\bY(s))\big)Z(s)\\
&\quad-\beta \alpha_j\Phi_{x_l}(s,\x)\Phi_{x_{l'}}(s,\x).
\end{align*}
To compute the time derivative of $\Phi$, let $g$ be a standard Gaussian vector in $\R^\kappa$ and consider the function
$$v(s)=\sqrt{\frac{2(q_{j+1}-s)}{q_{j+1}-q_j}}(\gamma_{j+1}-\gamma_j)^{1/2}.$$
Since $\Phi(s,\x)=\frac{1}{\beta \alpha_j}\log \E \exp \beta \alpha_j\Phi(q_{j+1}, \x+v(s) g)$, the Gaussian integration by parts formula gives
\begin{align*}
\Phi_s(s,x)&=\sum_{l,l'=1}^\kappa v'(s)_{l,l'}\E g_{l'}\Phi_{x_l}(q_{j+1},\x+v(s)g)\exp\beta \alpha_j(\Phi(q_{j+1},\x+v(s)g)-\Phi(s,\x))\\
&=\sum_{l,l',i=1}^\kappa v'(s)_{l,l'}v(s)_{il'}\E\big(\Phi_{x_lx_{i}}(q_{j+1},\bY(s))\\
&\qquad\qquad\qquad\qquad\qquad\qquad+\beta\alpha_j\Phi_{x_l}(q_{j+1},\bY(s))\Phi_{x_{i}}(q_{j+1},\bY(s))\big)Z(s)\\
&=-\sum_{l,l'=1}^\kappa \frac{(\gamma_{j+1}-\gamma_j)_{l,l'}}{q_{j+1}-q_j} \E\big(\Phi_{x_lx_{l'}}(q_{j+1},\bY(s))\\
&\qquad\qquad\qquad\qquad\qquad\qquad+\beta\alpha_j\Phi_{x_l}(q_{j+1},\bY(s))\Phi_{x_{l'}}(q_{j+1},\bY(s))\big)Z(s)\\
&=-\Big(\big(\pi'(s),\nabla^2\Phi(s,\x)\big)+\beta \alpha_j \big(\pi'(s)\nabla\Phi(s,\x),\nabla\Phi(s,\x)\big)\Big).
\end{align*}
Remembering that $\alpha(s)=\alpha_j$ completes the proof.
\end{proof}

The Hamilton-Jacobi equation \eqref{eqn: GP positive temperature Parisi PDE} is the vector spin analogue of the Parisi PDE \cite{PanSKB}. We now use similar ideas to those in \cite{AuffingerChenREP, Jagannath, PanSKB} to obtain the vector spin analogue of the Auffinger-Chen representation from \eqref{eqn: GP positive temperature Parisi PDE}. To overcome the lack of Lipschitz continuity in the terminal condition \eqref{eqn: GP positive temperature terminal condition}, we will rely upon three classical results in stochastic analysis: the Ito formula, the Girsanov theorem and the Novikov condition \cite{LeGall,Shreve}. Given a filtration $\F=(\F_s)_{0\leq s\leq 1}$, it will be convenient to denote by $\A$ the class of admissible control processes,
\begin{align}
\A&=\Big\{\bv=(v_1,\ldots,v_\kappa)\mid \bv=(\bv(s))_{0\leq s\leq 1} \text{ is progressively measurable }\\
&\qquad\qquad\qquad\qquad\qquad \text{and }\E\int_0^1 \norm{\bv(s)}_2^2 \ud s<\infty\Big\}\notag
\end{align}

\begin{proposition}\label{GP positive temperature Auffinger-Chen}
If $2<p<\infty$, then there exists a probability space $(\O,\F_1,\P)$, a filtration $\F=(\F_s)_{0\leq s\leq 1}$, a Brownian motion $\bW=(\bW_s)_{0\leq s\leq 1}$ and a continuous adapted process $\bX=(\bX_s)_{0\leq s\leq 1}$ which together form a weak solution to the stochastic differential equation
\begin{equation}\label{eqn: GP Parisi SDE at positive temperature}
\ud \bX(s)=2\beta \alpha(s)\pi'(s)\nabla \Phi(s,\bX(s))\ud s+\sqrt{2}\pi'(s)^{1/2}\ud \bW(s),\,\,\,\bX(0)=0.
\end{equation}
Moreover,
\begin{align}\label{eqn: GP Auffinger-Chen positive temperature}
\Phi(0,0)&=\sup_{\bv\in \A}\E\Big[f_\lambda^\beta\Big(2\int_0^1 \beta \alpha(s)\pi'(s)\bv(s)\ud s +\bB(1)\Big)\\
&\qquad\quad\quad\,\,-\int_0^1 \beta \alpha(s)\big(\pi'(s)\bv(s),\bv(s)\big)\ud s\Big]\notag
\end{align}
with the supremum attained by the admissible process $\bv(s)=\nabla \Phi(s,\bX(s))$.
\end{proposition}

\begin{proof}
To alleviate notation, let $C>0$ denote a constant that depends only on $\kappa$, $\lambda$, $\alpha$, $\pi$, $\beta$, $D$, $p$ and $t$ whose value might not be the same at each occurrence. An induction based on \cref{GP positive temperature terminal growth bound} and \cref{GP growth bound for derivative persists} can be used to show that for any $\smash{(s,\x)\in (0,1]\times \R^\kappa}$,
\begin{equation}\label{eqn: GP positive temperature Parisi function growth bound}
\norm{\nabla\Phi(s,\x)}_2\leq C\Big(1+\norm{\x}_2^{\frac{1}{p-1}}\Big).
\end{equation}
With this in mind, consider the process $\bL=(\bL(s))_{0\leq s\leq 1}$,
$$\bL(s)=\sqrt{2}\int_0^s \beta \alpha(r)\pi'(r)^{1/2}\nabla \Phi(s,\bB(r))\ud r.$$
The growth bound \eqref{eqn: GP positive temperature Parisi function growth bound} and the assumption $\frac{1}{p-1}<1$ imply that
$$\E \exp \int_0^1 \norm{\bL(s)}_2^2\ud s\leq C\E \exp \big(C\sup_{0\leq s\leq 1}\norm{\bB(s)}_2^2\big)\leq C\E \exp C\big(\sup_{0\leq s\leq 1}\norm{\bW(s)}_2^2\big),$$
where the last inequality uses the fact that $\pi'$ is piecewise constant. Combining this with Doob's maximal inequality reveals that
$$\E \exp \int_0^1 \norm{\bL(s)}_2^2\ud s\leq C\E \sup_{0\leq s\leq 1}\exp C\norm{\bW(s)}_2^2\leq C\E \exp C\norm{\bW(1)}_2^2<\infty.$$
It follows by the Novikov condition that the stochastic exponential $\EE(\bL)$ is a martingale. If we denote by $\Q$ the measure under which $\bW$ is a Brownian motion and introduce the measure $\ud \P=\EE(\bL(1))\ud \Q$, then Girsanov's theorem implies that
$$\tilde{\bW}(s)=\bW(s)-\sqrt{2}\int_0^s \beta \alpha(r)\pi'(r)^{1/2}\nabla \Phi(s,\bB(r))\ud r$$
is a $\P$-Brownian motion. Rearranging shows that $\smash{(\bB,\tilde{\bW})}$ is a weak solution to \eqref{eqn: GP Parisi SDE at positive temperature}. Henceforth, we will write $\smash{\F=(\F_s)_{0\leq s\leq 1}}$, $\smash{\bW=(\bW_s)_{0\leq s\leq 1}}$ and $\smash{\bX=(\bX_s)_{0\leq s\leq 1}}$ for a filtration, a Brownian motion and a continuous adapted process which together form a weak solution to \eqref{eqn: GP Parisi SDE at positive temperature}. Given $\bv\in \A$, let $\bY(s)=2\int_0^s \beta \alpha(r)\pi'(r)\bv(r)\ud r+\bB(s)$. By Ito's formula and the Parisi PDE \eqref{eqn: GP positive temperature Parisi PDE},
\begin{align*}
\ud \Phi&=\Phi_s\ud s+2\beta \alpha(s)\big(\nabla \Phi,\pi'(s)\bv(s)\big)\ud s+\big(\nabla^2\Phi,\pi'(s)\big)\ud s+\sqrt{2}\big(\nabla \Phi, \pi(s)\ud \bW(s)\big)\\
&=-\beta\alpha(s)\Big(\big(\pi'(s)(\nabla \Phi-\bv(s)),\nabla \Phi-\bv(s)\big)-\big(\pi'(s)\bv(s),\bv(s)\big)\Big)\ud s\\
&\quad+\sqrt{2}\big(\nabla \Phi, \pi(s)\ud \bW(s)\big),
\end{align*}
where $\Phi$ and its derivatives are evaluated at $(s,\bY(s))$. The growth bound \eqref{eqn: GP positive temperature Parisi function growth bound}, the Cauchy-Schwarz inequality, the boundedness of $\pi'$ and the Ito isometry reveal that
$$\E \int_0^1 \norm{\nabla \Phi(s,\bY(s))}_{\infty}^2\ud s\leq C\big(1+\sup_{0\leq s\leq 1}\E \norm{\bY(s)}_2^2\big)\leq C\Big(1+\E \int_0^1 \norm{\bv(s)}_2^2\ud s\Big)<\infty.$$
This means that $(\sqrt{2}\int_0^s(\nabla \Phi, \pi(s)\ud \bW(s)))_{s\leq 1}$ is a martingale. Together with the non-negative definiteness of $\pi'$, this implies that
$$\E \Phi(1,\bY(1))-\Phi(0,0)\leq \int_0^1 \beta \alpha(s)\E \big(\pi'(s)\bv(s),\bv(s)\big)\ud s$$
with equality for the process $\bv(s)=\nabla \Phi(s,\bX(s))$. Rearranging gives the lower bound in \eqref{eqn: GP Auffinger-Chen positive temperature}. To prove the matching upper bound, it suffices to show that $\bv(s)=\nabla \Phi(s,\bX(s))$ belongs to the admissible class $\A$. Fix $0<s\leq r\leq 1$. By the triangle inequality, the Cauchy-Schwarz inequality and the growth bound \eqref{eqn: GP positive temperature Parisi function growth bound},
\begin{equation}\label{eqn: GP Gronwall bound on X}
\sup_{0\leq s\leq r}\norm{\bX(s)}^2_2\leq C\Big(1+\int_0^r \sup_{0\leq s\leq w}\norm{\bX(s)}_2^2\ud w+\sup_{0\leq s\leq r}\norm{\bB(s)}_2^2\Big).
\end{equation}
On the other hand, Doob's maximal inequality and the Ito isometry yield
\begin{equation}\label{eqn: GP Ito isometry applied to B}
\E \sup_{0\leq s\leq r}\norm{\bB(s)}_2^2\leq \E \norm{\bB(r)}_2^2\leq C \tr \int_0^r \pi'(w)\ud w\leq C\tr(D).
\end{equation}
Substituting this into \eqref{eqn: GP Gronwall bound on X} and applying Gronwall's inequality to the resulting bound shows that $\E \sup_{0\leq s\leq 1}\norm{\bX(s)}_2^2\leq C$. Invoking \eqref{eqn: GP positive temperature Parisi function growth bound} one last time completes the proof.
\end{proof}

Of course, an analogous result holds for the random variable $Y_0$ in \eqref{eqn: GP Parisi sequence at zero temperature}. Given a discrete measure $\zeta\in \mathcal{N}^d$ defined by the sequences \eqref{eqn: GP zero temperature qs} and \eqref{eqn: GP zero temperature zetas}, introduce the function $\Psi:[0,1]\times \R^\kappa\to\R$ defined recursively by
\begin{equation}\label{eqn: GP zero-temperature Parisi function}
\left\{\begin{aligned}
\Psi(1,\x)&=f_\lambda^\infty(\x),\\
\Psi(s,\x)&=\frac{1}{\zeta_j}\log \E \exp \zeta_j\Psi(q_{j+1},\x+ \bB(q_{j+1})- \bB(s)), \, s\in [q_j,q_{j+1}).
\end{aligned}\right.    
\end{equation}
The Parisi functional \eqref{eqn: GP Parisi functional at zero temperature} may be written as
\begin{equation}
\Par_\infty(\lambda,\zeta,\pi)=\Psi(0,0)-\sum_{k\leq k'}\lambda_{k,k'}D_{k,k'}-\int_0^1  \zeta(s)\Sum\big(\pi(s)\odot \pi'(s)\big)\ud s,
\end{equation}
and the Gaussian integration by parts formula can be used as in \cref{GP positive temperature Parisi PDE} to show that \eqref{eqn: GP zero-temperature Parisi function} satisfies the Parisi PDE,
\begin{equation}\label{eqn: GP zero-temperature Parisi PDE}
\partial_s\Psi(s,\x)=-\Big(\big(\pi'(s),\nabla^2\Psi(s,\x)\big)+ \zeta(s)\big(\pi'(s)\nabla \Psi(s,\x),\nabla \Psi(s,\x)\big)\Big),
\end{equation}
where $\partial_s\Psi$ is understood as the right derivative at the points of discontinuity of $\zeta$. An identical argument to that in \cref{GP positive temperature Auffinger-Chen} gives the following weak form of the Auffinger-Chen representation.

\begin{proposition}\label{GP zero-temperature Auffinger-Chen}
If $2<p<\infty$, then there exists a probability space $(\O,\F_1,\P)$, a filtration $\F=(\F_s)_{0\leq s\leq 1}$, a Brownian motion $\bW=(\bW_s)_{0\leq s\leq 1}$ and a continuous adapted process $\bX=(\bX_s)_{0\leq s\leq 1}$ which together form a weak solution to the stochastic differential equation
\begin{equation}\label{eqn: GP Parisi SDE at zero temperature}
\ud \bX(s)=2\zeta(s)\pi'(s)\nabla \Psi(s,\bX(s))\ud s+\sqrt{2}\pi'(s)\ud \bW(s),\,\,\, \bX(0)=0.
\end{equation}
Moreover,
\begin{align}\label{eqn: GP Auffinger-Chen zero temperature}
\Psi(0,0)&=\sup_{\bv\in \A}\E\Big[f_\lambda^\infty\Big(2\int_0^1 \zeta(s)\pi'(s)\bv(s)\ud s +\bB(1)\Big)\\
&\qquad\quad\quad\,\,-\int_0^1 \zeta(s)\big(\pi'(s)\bv(s),\bv(s)\big)\ud s\Big]\notag
\end{align}
with the supremum attained by the admissible process $\bv(s)=\nabla \Psi(s,\bX(s))$.
\end{proposition}

We close this section with a moment bound on the weak solution to the stochastic differential equation \eqref{eqn: GP Parisi SDE at zero temperature} which will allow us to deal with the fact that \cref{GP zero temperature terminal condition bounded by positive temperature} only holds for bounded values of $\x$.

\begin{lemma}\label{GP moment bound on zero-temperature process}
If $(\bX(s))_{0\leq s\leq 1}$ is a weak solution to \eqref{eqn: GP Parisi SDE at zero temperature} and $\eta=\max(1+\frac{1}{p-1},\frac{2}{p-1})\in (1,2)$, then
\begin{equation}
\E\norm{\bX(1)}_2^2\leq C\Big(1+\norm{\zeta}_\infty(1+\norm{\lambda}_\infty)^{1+\frac{2}{p-2}}\Big)^{\frac{2}{2-\eta}}
\end{equation}
for some constant $C>0$ that depends only on $\kappa, p,t$ and $D$.
\end{lemma}

\begin{proof}
To alleviate notation, let $C>0$ denote a constant that depends only on $\kappa, p,t$ and $D$ whose value might not be the same at each occurrence. If $\E \norm{\bX(1)}_2^2<1$ the result is trivial, so assume without loss of generality that $\E\norm{\bX(1)}_2^2\geq 1$. Introduce the process $\bv(s)=\nabla\Psi(s,\bX(s))$ in such a way that
$$\bX(1)=\int_0^1 2\zeta(s)\pi'(s)\bv(s)\ud s+\bB(1).$$
The triangle inequality and \eqref{eqn: GP Ito isometry applied to B} reveal that
\begin{equation}\label{eqn: GP moment bound on zero-temperature process key}
\E\norm{\bX(1)}_2^2\leq C\Big(1+\E \Big\lVert \int_0^1 \zeta(s)\pi'(s)\bv(s)\ud s \Big\rVert^2_2\Big).
\end{equation}
With this in mind, fix $1\leq l \leq \kappa$. The Cauchy-Schwarz inequality implies that
\begin{align*}
\Big(\int_0^1 \zeta(s)(\pi'(s)\bv(s))_l\ud s\Big)^2&\leq C\sum_{k=1}^\kappa \Big(\int_0^1 \zeta(s)\pi'(s)^{1/2}_{lk}\sum_{i=1}^\kappa \pi'(s)^{1/2}_{ki}v_i(s)\ud s\Big)^2\\
&\leq C\sum_{k=1}^\kappa \int_0^1 \zeta(s)\pi'(s)_{lk}^{1/2}\pi'(s)_{lk}^{1/2}\ud s\\
&\qquad\qquad\int_0^1 \zeta(s)\Big(\sum_{i=1}^\kappa \pi'(s)_{ki}^{1/2}v_i(s)\Big)^2\ud s\\
&\leq C\norm{\zeta}_\infty \int_0^1 \zeta(s)\big(\pi'(s)\bv(s),\bv(s)\big)\ud s.
\end{align*}
Substituting this back into \eqref{eqn: GP moment bound on zero-temperature process key} yields
$$\E \norm{\bX(1)}_2^2\leq C\Big(1+\norm{\zeta}_\infty \int_0^1 \zeta(s)\E\big(\pi'(s)\bv(s),\bv(s)\big)\ud s\Big).$$
On the other hand, taking the zero process in \cref{GP zero-temperature Auffinger-Chen} shows that
$$\E\Big[f_\lambda^{\infty}(\bX(1))-\int_0^1\zeta(s) \big(\pi'(s)\bv(s),\bv(s)\big)\ud s\Big]=\Psi(0,0)\geq \E f_\lambda^{\infty}(\bB(1))\geq 0,$$
and therefore
\begin{equation}\label{eqn: GP L2 bound on Parisi process key}
\E\norm{\bX(1)}_2^2\leq C\big(1+\norm{\zeta}_\infty \E f_\lambda^\infty(\bX(1))\big).
\end{equation}
To bound this further, observe that by \eqref{eqn: GP zero temperature terminal condition explicit growth bound} in \cref{GP app1} and Jensen's inequality,
\begin{align*}
\E f_\lambda^\infty(\bX(1))&\leq C(1+\norm{\lambda}_\infty)^{1+\frac{2}{p-2}}\Big((\E \norm{\bX(1)}_2^2)^{\frac{1}{2}+\frac{1}{2(p-1)}}+(\E\norm{\bX(1)}_2^2)^{\frac{1}{2}}\\
&\qquad\qquad\qquad\qquad\quad\quad+(\E\norm{\bX(1)}_2^2)^{\frac{1}{p-1}}+1\Big)\\
&\leq C(1+\norm{\lambda}_\infty)^{1+\frac{2}{p-2}}(\E\norm{\bX(1)}_2^2)^{\eta/2},
\end{align*}
where we have used the assumption that $\E\norm{\bX(1)}_2^2\geq 1$. Substituting this back into \eqref{eqn: GP L2 bound on Parisi process key} and again using the fact that $\E\norm{\bX(1)}_2^2\geq 1$ gives
$$\E \norm{\bX(1)}_2^2\leq  C\Big(1+\norm{\zeta}_\infty (1+\norm{\lambda}_\infty)^{1+\frac{2}{p-2}}\Big)(\E\norm{\bX(1)}_2^2)^{\eta/2}.$$
Rearranging completes the proof.
\end{proof}

\section{Proof of the main result}\label{GP sec11}

In this section we finally prove \cref{GP main result}. The proof of the upper bound will follow section 12.2 of \cite{WeiKuo}. On the other hand, the proof of the lower bound will be considerably shorter and less involved than its one-dimensional analogue in \cite{WeiKuo}. In particular, it will leverage the results of \cref{GP sec9} to avoid the technicalities associated with truncating. Specializing our arguments to the scalar, $\kappa=1$, case gives a shorter and more direct proof of the main result in \cite{WeiKuo} when $2<p<\infty$.

\begin{lemma}\label{GP main result upper bound}
If $2<p<\infty$, $D\in \Gamma_\kappa$ and $t>0$, then
\begin{equation}
L_{p,D}(t)\leq \inf_{\lambda,\zeta,\pi}\Par_\infty(\lambda,\zeta,\pi),
\end{equation}
where the infimum is taken over all $(\lambda,\zeta,\pi)\in \R^{\kappa(\kappa+1)/2}\times \mathcal{N}^d\times \Pi_D$.
\end{lemma}

\begin{proof}
By \cref{GP Lagrangian by free energy}, it suffices to show that
\begin{equation}\label{eqn: GP main result upper bound goal}
\liminf_{\beta\to \infty}\lim_{N\to \infty}F_{N,\beta^{-1}}(\beta)\leq \inf_{\lambda,\zeta,\pi}\Par_\infty(\lambda,\zeta,\pi).    
\end{equation}
Fix a Lagrange multiplier $\lambda\in \R^{\kappa(\kappa+1)/2}$, a piecewise linear path $\pi \in \Pi_D$ and a discrete measure $\zeta\in \mathcal{N}^d$ defined by the sequences \eqref{eqn: GP zero temperature measure qs} and \eqref{eqn: GP zero temperature zetas}. Given an inverse temperature parameter $\beta>0$, introduce the measure
$$\alpha^\beta(s)=\beta^{-1}\zeta(s)\1_{[0,1)}(s)+\1_{\{1\}}(s)$$
on $[0,1]$. It is clear that $\alpha^\beta\in \M^d$ for $\beta$ large enough. Moreover, $\smash{\alpha^\beta(\{q_j\})=\alpha^\beta_j-\alpha^\beta_{j-1}}$ for the sequence of parameters \eqref{eqn: GP positive temperature alphas} defined by $\smash{\alpha^\beta_j=\beta^{-1}\zeta_j}$. The Guerra replica symmetry breaking bound in lemma 2 of \cite{PanVec} implies that
\begin{equation}\label{eqn: GP Guerra RSB}
\lim_{N\to \infty}F_{N,\beta^{-1}}(\beta)\leq \Par_\beta(\lambda,\alpha^\beta,\pi)+\beta^{-1}\norm{\lambda}_1+L\beta^{-1}.
\end{equation}
for some constant $L>0$ independent of $\beta$. To bound this further, recall that by \cref{GP positive temperature terminal condition bounded by zero temperature},
$$f_{\lambda,t}^\beta(\x)\leq f^\infty_{\lambda,t-\delta}(\x)-\frac{\kappa \log \beta \delta}{p\beta}+\frac{1}{\beta}\log \int_{\R^\kappa} e^{-\norm{\asigma}_2^p}\ud \asigma$$
for any $\delta\in (0,t/2)$. If we make the dependence of the functions $\Phi$ and $\Psi$ in \eqref{eqn: GP positive temperature Parisi function} and \eqref{eqn: GP zero-temperature Parisi function} on the underlying measure and parameter $t>0$ explicit by writing $\Phi^{\alpha^\beta,t}$ and $\Psi^{\zeta,t}$, then a simple induction yields
$$\Phi^{\alpha^\beta,t}(0,0)\leq \Psi^{\zeta,t-\delta}(0,0)-\frac{\kappa \log \beta \delta}{p\beta}+\frac{1}{\beta}\log \int_{\R^\kappa} e^{-\norm{\asigma}_2^p}\ud \asigma.$$
It follows that
\begin{align}\label{eqn: GP positive temperature infimum bounded by zero temperature functional}
\Par_\beta(\lambda,\alpha^\beta,\pi)&\leq \Psi^{\zeta,t-\delta}(0,0)-\sum_{k\leq k'}\lambda_{k,k'}D_{k,k'}-\int_0^1  \zeta(s)\Sum\big(\pi(s)\odot \pi'(s)\big)\ud s\\
&\quad-\frac{\kappa \log \beta \delta}{p\beta}+\frac{1}{\beta}\log \int_{\R^\kappa} e^{-\norm{\asigma}_2^p}\ud \asigma.\notag
\end{align}
Substituting this into \eqref{eqn: GP Guerra RSB} and letting $\beta \to \infty$ gives
\begin{equation*}
\lim_{\beta\to\infty}\lim_{N\to \infty}F_{N,\beta^{-1}}(\beta)\leq \Psi^{\zeta,t-\delta}(0,0)-\sum_{k\leq k'}\lambda_{k,k'}D_{k,k'}-\int_0^1  \zeta(s)\Sum\big(\pi(s)\odot \pi'(s)\big)\ud s.
\end{equation*}
By \cref{GP continuity in t of zero-temperature terminal condition} and an induction that combines the dominated convergence theorem with \eqref{eqn: GP zero temperature terminal condition explicit growth bound} in \cref{GP app1}, it is readily verified that
$\smash{\lim_{\delta\to 0}\Psi^{\zeta,t-\delta}(0,0)=\Psi^{\zeta,t}(0,0)}$. Letting $\delta\to 0$ in the above inequality and taking the infimum over $\lambda,\zeta$ and $\pi$ establishes \eqref{eqn: GP main result upper bound goal} and completes the proof.
\end{proof}

The proof of the matching lower bound in \eqref{eqn: GP main result} requires more work. Given an inverse temperature parameter $\beta>0$, denote by $(\lambda^\beta,\alpha^\beta,\pi^\beta)$ a triple of almost minimizers defined by the condition
\begin{equation}\label{eqn: GP almost minimizer condition}
\Par_\beta(\lambda^\beta,\alpha^\beta,\pi^\beta)\leq \inf_{\lambda,\alpha,\pi}\Par_\beta(\lambda,\alpha,\pi)+\beta^{-1}.
\end{equation}
It will be convenient to control the magnitude of $\lambda^\beta$. It is at this point that we have to specialize the claim of \cref{GP main result} to positive definite matrices $D\in \Gamma_\kappa^+$. The author's inability to extend the following result to the space of non-negative definite matrices is the reason for proving the second equality in \eqref{eqn: GP limit of unconstrained Lagrangian} and extending \cref{GP sec4} beyond \cref{GP continuity of constrained Lagrangian}.

\begin{lemma}\label{GP bound on almost maximizer}
If $D\in \Gamma_\kappa^+$ and $(\lambda^\beta,\alpha^\beta,\pi^\beta)\in \R^{\kappa(\kappa+1)/2}\times \M^d\times \Pi_D$ satisfies \eqref{eqn: GP almost minimizer condition} for some $\beta>1$, then there exists a constant $C>0$ that depends only on $p,\kappa, D$ and $t$ with
\begin{equation}
\norm{\lambda^\beta}_\infty\leq C\beta.
\end{equation}
\end{lemma}

\begin{proof}
To alleviate notation, let $C>0$ denote a constant that depends only on $p,\kappa,D$ and $t$ whose value might not be the same at each occurrence. For each pair $k>k'$, let $\smash{\lambda^\beta_{k,k'}=\lambda^\beta_{k',k}}$. Consider the $\kappa\times \kappa$ symmetric matrix $\smash{\Lambda^\beta=(\lambda^\beta_{kk'})}$ as well as the $\kappa \times \kappa$ symmetric matrix $\smash{\sgn\Lambda^\beta=(\sgn \lambda^\beta_{kk'})}$ containing its signs. We adopt the convention that $\sgn(0)=0$. Since $D$ is positive definite and $\norm{\sgn\Lambda^\beta}_\infty\leq 1$, using \cref{positive definite perturbation} it is possible to find $\epsilon>0$ small enough that depends only on $D$ and $\kappa$ such that
$$D'=D+\epsilon \sgn \Lambda^\beta$$
is positive definite. Introduce the Gaussian measure
$$\ud \mu(\asigma)=\Big(\frac{1}{\sqrt{2\pi}}\Big)^\kappa \frac{1}{\sqrt{\det D'}}\exp \Big(-\frac{1}{2}\asigma^TD'^{-1}\asigma\Big)\ud \asigma$$
associated with a centred Gaussian random vector $X$ having covariance matrix $D'$. Denote by $\Phi$ the function \eqref{eqn: GP positive temperature Parisi function} corresponding to the parameters $\lambda^\beta,\alpha^\beta,\pi^\beta$ and $\beta>0$. Since the terminal condition \eqref{eqn: GP positive temperature terminal condition} is convex, taking the zero process in \cref{GP positive temperature Auffinger-Chen} and invoking Jensen's inequality shows that
$$\Phi(0,0)\geq \E f_{\lambda^\beta}^\beta (\bB(1))\geq f_{\lambda^\beta}^\beta(\E \bB(1))=f_{\lambda^\beta}^\beta(0).$$
Another application of Jensen's inequality gives
\begin{align*}
\Phi(0,0)&\geq \frac{\log \sqrt{\det D'}}{\beta}+\frac{1}{\beta}\log \int_{\R^\kappa} \exp\beta \Big(\sum_{k\leq k'}\lambda_{k,k'}^\beta\sigma(k)\sigma(k')-t\norm{\asigma}_2^p\Big)\ud \mu(\asigma)\\
&\geq \frac{\log \sqrt{\det D'}}{\beta}+\sum_{k\leq k'}\lambda_{k,k'}^\beta\E X_kX_{k'}-t\E\norm{X}_2^p.
\end{align*}
It follows by definition of the Parisi functional in \eqref{eqn: GP Parisi functional at positive temperature preliminary} and the equality $D_{k,k'}'-D_{k,k'}=\epsilon \sgn \lambda_{k,k'}^\beta$ that
\begin{align*}
\Par_\beta(\lambda^\beta,\alpha^\beta,\pi^\beta)&\geq \epsilon \sum_{k\leq k'}\abs{\lambda_{k,k'}^\beta}-t\E\norm{X}_2^p+\frac{\log \sqrt{\det D'}}{\beta}\\
&\quad-\frac{\beta}{2}\sum_{0\leq j\leq r-1}\alpha_j^\beta\big(\norm{\gamma_{j+1}^\beta}_{\text{HS}}^2-\norm{\gamma_{j}^\beta}_{\text{HS}}^2\big).
\end{align*}
To bound this further, observe that
$$\sum_{0\leq j\leq r-1}\alpha_j^\beta\big(\norm{\gamma_{j+1}^\beta}_{\text{HS}}^2-\norm{\gamma_{j}^\beta}_{\text{HS}}^2\big)=\alpha_{r-1}^\beta \norm{\gamma_r^\beta}_{\text{HS}}^2-\sum_{j=1}^r (\alpha_j^\beta-\alpha_{j-1}^\beta)\norm{\gamma_j^\beta}_{\text{HS}}^2\leq \norm{D}_{\text{HS}}^2,$$
where the last inequality uses \cref{HS norm preserves order}, and denote by $\Psi^{\lambda,\zeta,\pi,t}$ the function \eqref{eqn: GP zero-temperature Parisi function} associated with the Lagrange multiplier $\lambda\in \R^{\kappa(\kappa+1)/2}$, the discrete measure $\zeta \in \mathcal{N}^d$, the piecewise linear path $\pi \in \Pi_D$ and the parameter $t>0$. By \eqref{eqn: GP almost minimizer condition} and \eqref{eqn: GP positive temperature infimum bounded by zero temperature functional},
\begin{align*}
\Par_\beta(\lambda^\beta,\alpha^\beta,\pi^\beta)&\leq 1-\frac{\kappa \log \beta t/2}{p\beta}+\frac{1}{\beta}\log \int_{\R^\kappa} e^{-\norm{\asigma}_2^p}\ud \asigma+\inf_{\lambda,\zeta,\pi}\Big(\Psi^{\lambda,\zeta,\pi,t/2}(0,0)\\
&\quad-\sum_{k\leq k'}\lambda_{k,k'}D_{k,k'}-\int_0^1  \zeta(s)\Sum\big(\pi(s)\odot \pi'(s)\big)\ud s\Big).
\end{align*}
Combining these three bounds and rearranging yields
\begin{equation}\label{eqn: GP bound on almost maximizer key}
\epsilon \sum_{k\leq k'}\abs{\lambda_{k,k'}^\beta}\leq C\beta+t\E\norm{X}_2^p-\frac{\log \sqrt{\det D'}}{\beta}.
\end{equation}
Notice that the matrix $D'$ depends only on $D,\kappa$ and $\sgn \Lambda^\beta$. Since there are only finitely many choices for the matrix $\sgn \Lambda^\beta$ and $\beta^{-1}\leq 1$, we can absorb the term $\beta^{-1}\log \sqrt{\det D'}$ into the constant $C>0$. To deal with the term involving the random vector $X$, let $M\in \R^{\kappa\times \kappa}$ be a positive definite matrix with $M^TM=D'$. If $g$ is a standard Gaussian vector in $\R^\kappa$, then the Cauchy-Schwarz inequality implies that
$$\E\norm{X}_2^p=\E\norm{Mg}_2^p\leq \norm{M}_{\text{HS}}^p\E\norm{g}_2^p=\tr(D')^{p/2}\E\norm{g}_2^p\leq (\tr(D)+\epsilon)^{p/2}\E\norm{g}_2^p.$$
Substituting this into \eqref{eqn: GP bound on almost maximizer key} completes the proof.
\end{proof}

\begin{proof}[Proof (\Cref{GP main result}).]
To alleviate notation, let $C>0$ denote a constant that depends only on $p,\kappa,D$ and $t$ whose value might not be the same at each occurrence. By \cref{GP Lagrangian by free energy}, \cref{GP main result upper bound}, \cref{GP limit of free energy} and the choice of the minimizing sequence $(\lambda^\beta,\alpha^\beta,\pi^\beta)$ satisfying \eqref{eqn: GP almost minimizer condition}, it suffices to show that
\begin{equation}\label{eqn: GP main result lower bound goal}
\inf_{\lambda,\zeta,\pi}\Par_\infty(\lambda,\zeta,\pi)\leq \limsup_{\beta\to \infty}\Par_\beta(\lambda^\beta,\alpha^\beta,\pi^\beta).
\end{equation}
Fix $\beta>1$, $L>0$ and $0<\delta<\min(L^{1/(p-1)},1/2)$. Consider the measure $\zeta^\beta=\beta\alpha^\beta\in \mathcal{N}^d$, and denote by $\Psi$ the function \eqref{eqn: GP zero-temperature Parisi function} associated with the Lagrange multiplier $\lambda^\beta$, the measure $\zeta^\beta$ and the path $\pi^\beta$. Let $\bX^\beta$ be a weak solution to the stochastic differential equation \eqref{eqn: GP Parisi SDE at zero temperature}, and write $\bv^\beta(s)=\nabla \Psi(s,\bX^\beta(s))$ for its corresponding optimal control process. Consider the set on which $\bX^\beta(1)$ lies within the ball of radius $L>0$ around the origin,
$$B=\big\{\norm{\bX^\beta(1)}_2\leq L\big\},$$
and notice that by \cref{GP zero temperature terminal condition bounded by positive temperature},
\begin{equation}\label{eqn: GP upper bound preliminary around origin}
\E f_{\lambda^\beta}^\infty(\bX^\beta(1))\1_B\leq \E f_{\lambda^\beta}^\beta(\bX^\beta(1))\1_B+C\delta^2\Big(\norm{\lambda^\beta}_\infty+L^{\frac{p-2}{p-1}}\Big)-\frac{\kappa \log 2\delta}{\beta}.
\end{equation}
To bound this further, observe that by \eqref{eqn: GP Jensen bound for positive temperature terminal condition} and symmetry,
\begin{align*}
f_{\lambda^\beta}^\beta(\x)&\geq \frac{1}{\abs{\CC_\delta(0)}}\int_{\CC_\delta(0)}\Big(\big(\asigma,\x\big)+\sum_{k\leq k'}\lambda^\beta_{k,k'}\sigma(k)\sigma(k')-t\norm{\asigma}_2^p\Big)\ud \asigma+\frac{\kappa \log 2\delta}{\beta}\\
&\geq \sum_{k=1}^\kappa \frac{\delta^2\lambda^\beta_{k,k}}{3}-\frac{t}{(2\delta)^\kappa}\int_{[-\delta,\delta]^\kappa}\norm{\asigma}_2^p\ud \asigma+\frac{\kappa \log 2\delta}{\beta}.
\end{align*}
It follows that
\begin{align*}
\E f_{\lambda^\beta}^\beta(\bX^\beta(1))\1_B&\leq \E\Big(f_{\lambda^\beta}^\beta(\bX^\beta(1))-\sum_{k=1}^\kappa \frac{\delta^2\lambda_{k,k}
^\beta}{3}+\frac{t}{(2\delta)^\kappa}\int_{[-\delta,\delta]^\kappa}\norm{\asigma}_2^p\ud \asigma\\
&\quad-\frac{\kappa \log 2\delta}{\beta}\Big)\1_B+\kappa \delta^2\norm{\lambda^\beta}_\infty+\frac{\kappa \abs{\log 2\delta}}{\beta}\\
&\leq \E f_{\lambda^\beta}^\beta(\bX^\beta(1))+C\delta^2\norm{\lambda^\beta}_\infty+\frac{t}{(2\delta)^\kappa}\int_{[-\delta,\delta]^\kappa} \norm{\asigma}_2^p\ud \asigma - \frac{2\kappa \log 2\delta}{\beta}.
\end{align*}
Together with \eqref{eqn: GP upper bound preliminary around origin} and \cref{GP bound on almost maximizer}, this implies that
\begin{align}\label{eqn: GP upper bound around origin}
\E f_{\lambda^\beta}^\infty(\bX^\beta(1))\1_B&\leq \E f_{\lambda^\beta}^\beta(\bX^\beta(1))+C\delta^2\Big(\beta+L^{\frac{p-2}{p-1}}\Big)\\
&\quad+\frac{t}{(2\delta)^\kappa}\int_{[-\delta,\delta]^\kappa}\norm{\asigma}_2^p\ud \asigma-\frac{3\kappa \log 2\delta}{\beta}. \notag 
\end{align}
On the other hand, if $\eta=\max(1+\frac{1}{p-1},\frac{2}{p-1})\in (1,2)$ as in \cref{GP moment bound on zero-temperature process},
then Hölder's inequality and Chebyshev's inequality give
\begin{align}\label{eqn: GP upper bound away from origin}
\E f_{\lambda^\beta}^\infty(\bX^\beta(1))\1_{B^c}&\leq (\E\abs{ f_{\lambda^\beta}^\infty(\bX^\beta(1))}^{2/\eta})^{\eta/2}\P(B^c)^{1-\eta/2}\notag\\
&\leq \frac{1}{L^{2-\eta}}(\E\abs{ f_{\lambda^\beta}^\infty(\bX^\beta(1))}^{2/\eta})^{\eta/2} (\E \norm{\bX^\beta(1)}_2^2)^{1-\eta/2}.
\end{align}
Remembering that $\zeta^\beta=\beta\alpha^\beta$ and invoking \cref{GP moment bound on zero-temperature process} as well as \cref{GP bound on almost maximizer} shows that
\begin{equation}\label{eqn: GP upper bound away from origin first}
\E\norm{\bX^\beta(1)}_2^2\leq C\Big(1+\norm{\zeta^\beta}_\infty(1+\norm{\lambda^\beta}_\infty)^{1+\frac{2}{p-2}}\Big)^{\frac{2}{2-\eta}}\leq C\beta^{K},
\end{equation}
where $K>0$ is a constant that depends only on $p$ whose value might not be the same at each occurrence. Leveraging \eqref{eqn: GP zero temperature terminal condition explicit growth bound} in \cref{GP app1} and \cref{GP bound on almost maximizer} yields
\begin{align}\label{eqn: GP upper bound away from origin second}
\E\abs{f_{\lambda^\beta}^\infty(\bX^\beta(1))}^{2/\eta}&\leq C\Big(\E \norm{\bX^\beta(1)}_2^2+\norm{\lambda^\beta}_\infty^{K}\E\norm{\bX^\beta(1)}_2^2\notag\\
&\qquad\quad+\norm{\lambda^\beta}_\infty^{K}\E\norm{\bX^\beta(1)}_2^2+\norm{\lambda^\beta}_\infty^{K}\Big)\leq C\beta^{K}.
\end{align}
Combining the bound resulting from substituting \eqref{eqn: GP upper bound away from origin first} and \eqref{eqn: GP upper bound away from origin second} into \eqref{eqn: GP upper bound away from origin} with  \eqref{eqn: GP upper bound around origin} reveals that
\begin{align*}
\E f_{\lambda^\beta}^\infty(\bX^\beta(1))&\leq \E f_{\lambda^\beta}^\beta(\bX^\beta(1))+C\delta^2\Big(\beta+L^{\frac{p-2}{p-1}}\Big)+\frac{C\beta^{K}}{L^{2-\eta}}\\
&\quad+\frac{t}{(2\delta)^\kappa}\int_{[-\delta,\delta]^\kappa}\norm{\asigma}_2^p\ud \asigma-\frac{3\kappa\log 2\delta}{\beta}. 
\end{align*}
If we write $\Phi$ for the function \eqref{eqn: GP positive temperature Parisi function} associated with the inverse temperature $\beta>0$, the Lagrange multiplier $\lambda^\beta$, the measure $\alpha^\beta$ and the path $\pi^\beta$, then \cref{GP positive temperature Auffinger-Chen} and \cref{GP zero-temperature Auffinger-Chen} imply that
\begin{align*}
\Psi(0,0)&=\E\Big[f_{\lambda^\beta}^\infty(\bX^\beta(1))-\int_0^1 \zeta^\beta(s)\big((\pi^\beta)'(s)\bv^\beta(s),\bv^\beta(s)\big)\ud s\Big]\\
&\leq\Phi(0,0)+C\delta^2\Big(\beta+L^{\frac{p-2}{p-1}}\Big)+\frac{C\beta^{K}}{L^{2-\eta}}+\frac{t}{(2\delta)^\kappa}\int_{[-\delta,\delta]^\kappa}\norm{\asigma}_2^p\ud \asigma-\frac{3\kappa\log 2\delta}{\beta}.
\end{align*}
It follows that
\begin{align*}
\inf_{\lambda,\zeta,\pi}\Par_\infty(\lambda,\zeta,\pi)\leq \Par_\infty(\lambda^\beta,\zeta^\beta,\pi^\beta)&\leq \Par_{\beta}(\lambda^\beta,\alpha^\beta,\pi^\beta)+C\delta^2\Big(\beta+L^{\frac{p-2}{p-1}}\Big)+\frac{C\beta^{K}}{L^{2-\eta}}\\
&\quad+\frac{t}{(2\delta)^\kappa}\int_{[-\delta,\delta]^\kappa}\norm{\asigma}_2^p\ud \asigma-\frac{3\kappa\log 2\delta}{\beta}.
\end{align*}
Taking $L=\beta^m$ and $\delta=\frac{1}{\beta^{(m+1)/2}}$ for $m=\frac{K+1}{2-\eta}$ and letting $\beta\to \infty$ establishes \eqref{eqn: GP main result lower bound goal} and completes the proof.
\end{proof}

\renewcommand{\theequation}{A.\arabic{equation}}
\setcounter{equation}{0}

\begin{appendix}
\section{Terminal conditions growth bounds}\label{GP app1}

In this appendix we include the technical bounds on the terminal conditions $f_\lambda^\beta$ and $f_\lambda^\infty$ in \eqref{eqn: GP positive temperature terminal condition} and \eqref{eqn: GP zero-temperature terminal condition} that make it possible to define \eqref{eqn: GP Parisi sequence at zero temperature} and \eqref{eqn: GP Parisi sequence at positive temperature}. We also show that these bounds are preserved by the iterative procedure used to define $X_0$ and $Y_0$ in \eqref{eqn: GP Parisi sequence at positive temperature} and \eqref{eqn: GP Parisi sequence at zero temperature}, which plays an instrumental role in the proofs of \cref{GP positive temperature Auffinger-Chen} and \cref{GP zero-temperature Auffinger-Chen}.

\begin{lemma}\label{GP positive temperature terminal growth bound}
If $2<p<\infty$, $\beta>0$, $\x\in \R^\kappa$, $\lambda \in \R^{\kappa(\kappa+1)/2}$ and $t>0$, then
\begin{align}
\abs{f_\lambda^\beta(\x)}&\leq C\Big(1+\norm{\x}_2^{1+\frac{1}{p-1}}\Big),\label{eqn: GP positive temperature terminal growth bound}\\
\norm{\nabla f_\lambda^\beta(\x)}_\infty&\leq C\Big(1+\norm{\x}_2^{\frac{1}{p-1}}\Big)\label{eqn: GP positive temperature terminal derivative growth bound}
\end{align}
for some constant $C>0$ that depends only on $\beta,\kappa,\lambda,p$ and $t$.
\end{lemma}

\begin{proof}
Fix $1\leq i\leq \kappa$. Since $p>2$, a direct computation shows that
$$\partial_{x_i}f_\lambda^\beta(\x)=\frac{\int_{\R^\kappa} \sigma(i) \exp \beta\Big(\big(\asigma,\x\big)+\sum_{k\leq k'}\lambda_{k,k'}\sigma(k)\sigma(k')-t\norm{\asigma}_2^p\Big)\ud \asigma}{\int_{\R^\kappa}  \exp \beta\Big(\big(\asigma,\x\big)+\sum_{k\leq k'}\lambda_{k,k'}\sigma(k)\sigma(k')-t\norm{\asigma}_2^p\Big)\ud \asigma}.$$
To simplify notation, write $(I)$ for the numerator and $(II)$ for the denominator in this expression. Introduce the set
$$A=\Big\{\asigma\in \R^\kappa\mid \norm{\x}_2<\frac{t}{2}\norm{\asigma}_2^{p-1}\Big\}=\Big\{\asigma\in \R^\kappa\mid \norm{\asigma}_2>\Big(\frac{2\norm{\x}_2}{t}\Big)^{\frac{1}{p-1}}\Big\}.$$
On the one hand, the Cauchy-Schwarz inequality implies that
\begin{align*}
(I) &\leq \int_A \norm{\asigma}_2 \exp\beta \Big( \norm{\asigma}_2\norm{\x}_2+\sum_{k\leq k'}\lambda_{k,k'}\sigma(k)\sigma(k')-t\norm{\asigma}_2^p\Big)\ud \asigma\\
&\quad+\int_{A^c} \norm{\asigma}_2 \exp\beta \Big( \big(\asigma,\x\big)+\sum_{k\leq k'}\lambda_{k,k'}\sigma(k)\sigma(k')-t\norm{\asigma}_2^p\Big)\ud \asigma\\
&\leq\int_{\R^\kappa} \norm{\asigma}_2 \exp\beta\Big(\sum_{k\leq k'}\lambda_{k,k'}\sigma(k)\sigma(k')-\frac{t}{2}\norm{\asigma}_2^p\Big)\ud \asigma\\
&\quad+\Big(\frac{2\norm{\x}_2}{t}\Big)^{\frac{1}{p-1}}\int_{\R^\kappa} \exp\beta\Big(\big(\asigma,\x\big)+\sum_{k\leq k'}\lambda_{k,k'}\sigma(k)\sigma(k')-t\norm{\asigma}_2^p\Big)\ud \asigma.
\end{align*}
On the other hand, Jensen's inequality and symmetry yield
\begin{align*}
(II)&\geq \int_{[-\frac{1}{2},\frac{1}{2}]^\kappa}\exp \beta\Big(\big(\asigma,\x\big)+\sum_{k\leq k'}\lambda_{k,k'}\sigma(k)\sigma(k')-t\norm{\asigma}_2^p\Big)\ud \asigma\\
&\geq \exp\beta \int_{[-\frac{1}{2},\frac{1}{2}]^\kappa}\Big( \big(\asigma,\x\big)+\sum_{k\leq k'}\lambda_{k,k'}\sigma(k)\sigma(k')-t\norm{\asigma}_2^p\Big)\ud \asigma\\
&=\exp\beta \int_{[-\frac{1}{2},\frac{1}{2}]^\kappa}\Big( \sum_{k\leq k'}\lambda_{k,k'}\sigma(k)\sigma(k')-t\norm{\asigma}_2^p\Big)\ud \asigma.
\end{align*}
Combining these two bounds gives \eqref{eqn: GP positive temperature terminal derivative growth bound}. The fundamental theorem of calculus reveals that
\begin{align*}
\abs{f_\lambda^\beta(\x)}- \abs{f_\lambda^\beta(0)}&\leq  \int_0^1\abs{\partial_s f_\lambda^\beta (s\x)}\ud s\leq \int_0^1 \abs{(\nabla f_\lambda^\beta(s\x) , \x)}\ud s\\
&\leq \int_0^1 \norm{\nabla f_\lambda^\beta(s\x)}_2\norm{\x}_2\ud s\leq C\Big(1+\norm{\x}_2^{1+\frac{1}{p-1}}\Big)
\end{align*}
which establishes \eqref{eqn: GP positive temperature terminal growth bound} and completes the proof.
\end{proof}

\begin{lemma}\label{GP zero temperature terminal growth bound}
If $2<p<\infty$, $\x\in \R^\kappa$, $\lambda \in \R^{\kappa(\kappa+1)/2}$ and $t>0$, then
\begin{equation}
\abs{f_\lambda^\infty(x)}\leq C\Big(1+\norm{\x}_2^{1+\frac{1}{p-1}}\Big)
\end{equation}
for some constant $C>0$ that depends only on $\lambda, p$ and $t$. Moreover, the function $f_\lambda^\infty$ is differentiable for almost every $\x\in \R^\kappa$ with
\begin{equation}
\norm{\nabla f_\lambda^\infty(\x)}_\infty\leq C\Big(1+\norm{\x}_2^{\frac{1}{p-1}}\Big)\label{eqn: GP zero temperature terminal derivative growth bound}
\end{equation}
for a possibly different constant $C>0$ that depends only on $\lambda, p$ and $t$.
\end{lemma}

\begin{proof}
Consider the function $g:\R^\kappa\times \R^\kappa \to \R$ defined by
$$g(\x,\asigma)=\big(\asigma,\x\big)+\sum_{k\leq k'}\lambda_{k,k'}\sigma(k)\sigma(k')-t\norm{\asigma}_2^p.$$
Notice that $f_\lambda^\infty(\x)=\sup_{\asigma\in \R^\kappa}g(\x,\asigma)$. By \cref{GP bound on argmax}, there exists $\asigma(\x)\in \R^\kappa$ with $f_\lambda^\infty(\x)=g(\x,\asigma(\x))$ and
\begin{equation}\label{eqn: GP zero temperature terminal condition growth bound key}
\norm{\asigma(\x)}_2\leq \max\Big(\Big(\frac{2\norm{\lambda}_\infty}{t}\Big)^{\frac{1}{p-2}},\Big(\frac{2\norm{\x}_2}{t}\Big)^{\frac{1}{p-1}}\Big).
\end{equation}
It follows by the Cauchy-Schwarz inequality that
\begin{align}\label{eqn: GP zero temperature terminal condition explicit growth bound}
0=g(\x,0)\leq f_\lambda^\infty(\x)&\leq \norm{\asigma(\x)}_2\norm{\x}_2+\norm{\lambda}_\infty \norm{\asigma(\x)}_2^2\notag\\
&\leq \Big(\frac{2}{t}\Big)^{\frac{1}{p-1}}\norm{\x}_2^{1+\frac{1}{p-1}}+\Big(\frac{2\norm{\lambda}_\infty}{t}\Big)^{\frac{1}{p-2}}\norm{\x}_2+\\
&\quad+\Big(\frac{2}{t}\Big)^{\frac{2}{p-1}}\norm{\lambda}_\infty\norm{\x}_2^{\frac{2}{p-1}}+\Big(\frac{2\norm{\lambda}_\infty}{t}\Big)^{\frac{2}{p-2}}\norm{\lambda}_\infty\notag\\
&\leq C\Big(1+\norm{\x}_2^{1+\frac{1}{p-1}}\Big)\notag,
\end{align}
where the last inequality uses the fact that $\smash{\frac{1}{p-1}<1}$. To establish \eqref{eqn: GP zero temperature terminal derivative growth bound}, notice that $\x\mapsto g(\x,\asigma)$ is convex for each $\asigma\in \R^\kappa$. As the pointwise supremum of a family of convex functions, the function $f_\lambda^\infty$ is also convex and therefore differentiable almost everywhere. If $\x\in \R^\kappa$ is a point of differentiability of $f_\lambda^\infty$, then for any other $\x'\in \R^\kappa$,
$$f_\lambda^\infty(\x')-f_\lambda^\infty(\x)=\big(\x'-\x,\nabla f_\lambda^\infty(\x)\big)+o\big(\abs{\x'-\x}\big).$$
Combining this with the fact that
$$f_\lambda^\infty(\x')-f_\lambda^\infty(\x)\geq g(\x',\asigma(\x))-g(\x,\asigma(\x))=\big(\asigma(\x),\x'-\x\big)$$
yields
$$\big(\x'-\x,\nabla f_\lambda^\infty(\x)-\asigma(\x)\big)\geq o\big(\abs{\x'-\x}\big).$$
This is only possible if $\nabla f_\lambda^\infty(\x)=\asigma(\x)$ so the result follows by \eqref{eqn: GP zero temperature terminal condition growth bound key}.
\end{proof}

\begin{lemma}\label{GP growth bound for derivative persists}
Let $f:\R^\kappa \to \R$ be a convex and differentiable function with
\begin{equation}
-M\leq f(\x)\leq C\big(1+\norm{\x}_2^{a+1}\big) \quad \text{and} \quad \norm{\nabla f(\x)}_2\leq C\big(1+\norm{\x}_2^a\big)
\end{equation}
for some $a\in (0,1)$ and some constants $C,M>0$. If $F:[0,1)\times \R^\kappa\to \R$ is defined by
\begin{equation}
F(s,\x)=\frac{1}{m}\log \E \exp mf(\x+A(s)g)
\end{equation}
for some $m>0$, a standard Gaussian vector $g$ in $\R^\kappa$ and a non-negative definite matrix $A(s)$ with uniformly bounded norm, $\norm{A(s)}_{\mathrm{HS}}^2\leq C$, then there exists $C'>0$ that depends on $\kappa, a,C,m$ and $M$ such that
\begin{equation}
\norm{\nabla F(s,\x)}_2\leq C'\big(1+\norm{\x}_2^a\big)
\end{equation}
for all $(s,\x)\in [0,1)\times \R^\kappa$.
\end{lemma}

\begin{proof}
To simplify notation, write $C'>0$ for a constant that depends on $\kappa, a,C,m$ and $M$ whose value might not be the same at each occurrence. Fix $s\in [0,1)\times \R^\kappa$ and $1\leq l\leq \kappa$. Since $a\in (0,1)$,
$$\partial_{x_l}F(s,\x)=\frac{\E \partial_{x_l}f(\x+A(s)g)\exp mf(\x+A(s)g)}{\E \exp m f(\x+A(s)g)}.$$
With this in mind, consider the set
$B=\big\{\norm{\x}_2\leq \norm{A(s)g}_2\big\}$. On the one hand,
\begin{align*}
\E \abs{\partial_{x_l}f(\x+A(s)g)}e^{ mf(\x+A(s)g)}\1_B&\leq C'\Big(1+\E \norm{A(s)g}_2^{a}e^{mC'\norm{A(s)g}_2^{a+1}}\Big)\\
&\leq C'\Big(1+\norm{A(s)}_{\text{HS}}^a\E \norm{g}_2^ae^{mC'\norm{A(s)}_{\text{HS}}^{a+1}\norm{g}_2^{a+1}}\Big)\\
&\leq C'\Big(1+\E \norm{g}_2^ae^{C'\norm{g}_2^{a+1}}\Big)\leq C',
\end{align*}
where the last inequality uses the fact that $a+1<2$. On the other hand, Jensen's inequality and the convexity of $f$ imply that
$$e^{-mM}\leq e^{mf(\x)}=e^{mf(\x+\E A(s)g)}\leq \E e^{mf(\x+A(s)g)}.$$
It follows that
\begin{align*}
\abs{\partial_{x_l}F(x,\x)}&\leq C'+\frac{\E \abs{\partial_{x_l}f(\x+A(s)g)}e^{mf(\x+A(s)g)}\1_B}{\E e^{mf(\x+A(s)g)}}\\
&\leq C'\Big(1+\frac{\E \norm{\x}_2^a e^{mf(\x+A(s)g)}}{\E e^{mf(\x+A(s)g)}}\Big)\leq C'\big(1+\norm{\x}_2^a\big).
\end{align*}
This completes the proof.
\end{proof}

\renewcommand{\theequation}{B.\arabic{equation}}
\setcounter{equation}{0}
\section{Background material}\label{GP app2}

In this appendix we collect a number of elementary results from linear algebra.

\begin{theorem}[Gershgorin]\label{Gershgorin}
If $A\in \R^{n\times n}$ and $R_i=\sum_{j\neq i}\abs{a_{ij}}$ is the sum of the absolute values of the non-diagonal entries in the $i$'th row of $A$, then the eigenvalues of $A$ are all contained in the union of the Gershgorin discs,
\begin{equation}
G(A)=\bigcup_{i=1}^n \{z\in \C\mid \abs{z-a_{ii}}\leq R_i\}.
\end{equation}
\end{theorem}

\begin{proof}
See theorem 6.1.1 in \cite{Matrix}.
\end{proof}

\begin{corollary}\label{Gershgorin corollary}
If $A\in \R^{n\times n}$ is a symmetric matrix with non-negative diagonal entries satisfying
\begin{equation}\label{eqn: diagonally dominant}
a_{ii}=\abs{a_{ii}}\geq \sum_{j\neq i}\abs{a_{ij}}
\end{equation}
for $1\leq i\leq n$, then $A$ is non-negative definite.
\end{corollary}

\begin{proof}
In the notation of Gershgorin's theorem, condition \eqref{eqn: diagonally dominant} may be written as $a_{ii}\geq R_i$. It follows by the symmetry of $A$ and Gershgorin's theorem that all the eigenvalues of $A$ are non-negative. This completes the proof. 
\end{proof}

\begin{lemma}\label{trace dominates HS norm}
If $A\in \R^{n\times n}$ is a non-negative definite and symmetric matrix, then
\begin{equation}
\norm{A}_{\mathrm{HS}}\leq \tr(A).
\end{equation}
\end{lemma}

\begin{proof}
Let $\lambda_n\geq \lambda_{n-1}\geq \cdots\geq \lambda_1\geq 0$ be the real and non-negative eigenvalues of the matrix $A$.
Since the trace of a matrix is the sum of its eigenvalues,
$$\norm{A}_{\text{HS}}^2=\tr(A A^T)=\tr(A^2)=\sum_{i=1}^n \lambda_i^2\leq \Big(\sum_{i=1}^n \lambda_i\Big)^2=\tr(A)^2.$$
We have used the fact that the eigenvalues of $A^2$ are $\lambda_n^2\geq \lambda_{n-1}^2\geq \cdots\geq \lambda_1^2\geq 0$ in the third equality and the non-negativity of the eigenvalues of $A$ in the inequality.
\end{proof}

\begin{lemma}\label{HS norm preserves order}
If $A,B,C\in \R^{n\times n}$ are non-negative definite and symmetric matrices with $B\leq C$, then $\tr(AB)\leq \tr(AC)$. In particular, $\norm{B}_{\mathrm{HS}}\leq \norm{C}_{\mathrm{HS}}$ whenever $B\leq C$.
\end{lemma}

\begin{proof}
Since $A$ is symmetric and non-negative definite, there exists a symmetric and non-negative definite matrix $M$ with $M^TM=A$. If $M=(m_1,\ldots,m_n)$, where $m_i\in \R^{n}$ denotes the $i$'th column of $M$, then
$$A=M^TM=\sum_{i=1}^n m_im_i^.$$
The linearity and cyclicity of the trace imply that
$$\tr(AB)=\sum_{i=1}^n \tr(m_im_i^TB)=\sum_{i=1}^n \tr(m_i^TBm_i)=\sum_{i=1}^n m_i^TBm_i.$$
Invoking the assumption that $B\leq C$ yields $\tr(AB)\leq \sum_{i=1}^n m_iCm_i^T=\tr(AC)$. To complete the proof observe that
$$\norm{B}_{\text{HS}}^2=\tr(B^TB)\leq \tr(B^TC)=\tr(C^TB)\leq \tr(C^TC)=\norm{C}_{\text{HS}}^2,$$
where we have used the fact that the trace of a matrix coincides with the trace of its transpose.
\end{proof}

\begin{lemma}\label{positive definite perturbation}
If $A\in \R^{n\times n}$ is a symmetric and positive definite matrix and $P\in \R^{n\times n}$ is a symmetric matrix, then there exists $\epsilon^*=\epsilon^*(A,\norm{P}_\infty,n)>0$ such that
\begin{equation}
A+\epsilon P
\end{equation}
is symmetric and positive definite for every $\epsilon<\epsilon^*$.
\end{lemma}

\begin{proof}
Denote by $\lambda_1$ the smallest eigenvalue of $A$. For any $x\in \R^n$ and every $\epsilon>0$,
$$x^T(A+\epsilon P)x\geq x^TAx-\epsilon \norm{P}_\infty\norm{x}_1^2\geq (\lambda_1-\epsilon n\norm{P}_\infty)\norm{x}_2^2.$$
Since $\lambda_1>0$ by positive definiteness of $A$, setting $\epsilon^*=\frac{\lambda_1}{n\norm{P}_\infty}$ completes the proof.
\end{proof}

\end{appendix}

\end{document}